\documentclass[12pt,leqno]{article}

\pdfoutput=1
\usepackage[a4paper,top=1.2in,bottom=1in,left=1.2in,right=1.2in]{geometry}

\usepackage{caption}
\usepackage{titlesec}
\renewcommand{\thesubsection}{\thesection.\arabic{subsection}}
\titleformat{\section}[hang]{\normalfont\bfseries}{\thesection}{.5em}{}
\titlelabel{\subsubsection}{\empty}
\titleformat{\subsection}{\normalfont\bfseries}{\thesubsection.}{.5em}{}[]
\titleformat{\subsubsection}[hang]{\normalfont\it}{\thesubsubsection.}{.4em}{}[]

\renewcommand{\abstract}[1]{{\gdef\thepoabstract{#1}}}

\usepackage{amsthm}
\usepackage{thmtools}
\declaretheoremstyle[spaceabove=1em,spacebelow=1em,headfont=\footnotesize\bfseries,bodyfont=\footnotesize]{problemstyle}

\makeatletter
\renewcommand\maketitle
{
\thispagestyle{empty}
\rmfamily\selectfont
\ifdefined\@title{\noindent\Large\bfseries\centering \@title \par}\else\fi
\vspace{2em}
\ifdefined\@author{\centering\normalfont \@author \par}
\vspace{.5em}
\ifdefined\thepoaffiliation{\noindent\small\thepoaffiliation\par}\fi\vspace{5em}\else\fi
\ifdefined\thepoabstract{\small\noindent{{\bfseries Abstract}.\;\;}\thepoabstract\par\vspace{3.5em}}\else\fi
\ifdefined\thepokeywords{{\noindent\bfseries Keywords\;\;}\thepokeywords\par\vspace{5em}}\else\fi
\ifdefined\theporuntitle{\fancyhead[L]{\footnotesize\theporuntitle}}\fi
}
\makeatother

\usepackage[usenames,dvipsnames]{color}
\definecolor{RefColor}{rgb}{0,0,.85}
\definecolor{UrlColor}{rgb}{.5,.5,.5}%
\RequirePackage[colorlinks,linkcolor=RefColor,citecolor=RefColor,urlcolor=UrlColor,linktoc=page]{hyperref}
\hypersetup{pdfinfo={Subject={ }}}
\RequirePackage[OT1]{fontenc}
\usepackage[numbers,sort]{natbib}

\usepackage[english]{babel}

\usepackage{enumitem}
\setlist[itemize]{leftmargin=1.5em}

\definecolor{CommentColor}{rgb}{0,0,.50}
\newcounter{amargincounter}

\definecolor{CommentColor2}{rgb}{0.5,0,0}

\usepackage{amsmath,amssymb,amscd,amsfonts,amsthm,mathtools}
\usepackage[noabbrev,capitalize]{cleveref}
\usepackage{dsfont}
\usepackage{thmtools}
\usepackage{nccmath}
\usepackage{scalerel}
\theoremstyle{plain}
\declaretheoremstyle[postheadspace=.4em,headfont=\bfseries,bodyfont=\itshape,spaceabove=8pt,
spacebelow=10pt]{basic}
\theoremstyle{basic}
\declaretheorem[style=basic,name={Theorem}]{theorem}
\declaretheorem[style=basic,sibling=theorem,name={Fact}]{fact}
\declaretheorem[style=basic,sibling=theorem,name={Lemma}]{lemma}
\declaretheorem[style=basic,sibling=theorem,name={Proposition}]{proposition}
\declaretheorem[style=basic,sibling=theorem,name={Corollary}]{corollary}
\theoremstyle{definition}

\newtheorem{example}[theorem]{Example}

\newtheorem*{remark*}{Remark}
\declaretheoremstyle[postheadspace=1em,
  mdframed={backgroundcolor=gray!10!white,
    hidealllines=true,
    innertopmargin=4pt,
    innerbottommargin=4pt,
    innerleftmargin=7pt,
    skipabove=8pt,
    skipbelow=10pt,
  nobreak=false}
]{grayboxed}

\DeclareMathOperator{\tsum}{{\textstyle\sum}}

\DeclareMathOperator{\msum}{\medmath\sum}

\newcommand{\mint}{\medint\int}
\DeclareMathOperator*{\argmax}{arg\,max}
\DeclareMathOperator*{\argmin}{arg\,min}

\newcommand{\argdot}{{\,\vcenter{\hbox{\tiny$\bullet$}}\,}}
\newcommand{\darrow}{\xrightarrow{\;\text{\tiny\rm d}\;}}

\newcommand{\mean}{\mathbb{E}}
\newcommand{\equdist}{\stackrel{\text{\rm\tiny d}}{=}}

\newcommand{\braces}[1]{{\lbrace #1 \rbrace}}

\DeclareRobustCommand{\svdots}{%
  \vbox{%
    \baselineskip=0.33333\normalbaselineskip
    \lineskiplimit=0pt
    \hbox{.}\hbox{.}\hbox{.}%
    \kern-0.2\baselineskip
  }%
}

\usepackage{bibentry}

\usepackage{booktabs}
\usepackage{tikz}
\usepackage[low-sup]{subdepth}
\usetikzlibrary{calc,shapes,shadings,positioning}
\usetikzlibrary{decorations.markings,decorations.pathreplacing}
\tikzstyle{mybraces}=[mirrorbrace/.style={
          decoration={brace, mirror},
          decorate},brace/.style={
          decoration={brace},
          decorate}]
\usepackage{pst-node}
\usepackage{tikz-cd}

\DeclareMathOperator*{\medcap}{\raisebox{-.15em}{$\mathbin{\scalebox{1.5}{\ensuremath{\cap}}}$}}

\newcommand{\step}{\refstepcounter{proofstep}$\theproofstep^\circ$\ }
\newcounter{proofstep}
\AtBeginEnvironment{proof}{\setcounter{proofstep}{0}}

\newcommand{\A}{\mathbf{A}}

\newcommand{\TV}{{\text{\rm\tiny TV}}}

\newcommand{\ca}{\mathbf{ca}}
\renewcommand{\L}{\mathbf{L}}
\renewcommand{\sp}[1]{\left<\mkern2mu\smash{#1}\mkern2mu\right>}

\newcommand{\ex}{\text{\rm ex}}
\newcommand{\Folner}{F{\o}lner }

\newcommand{\orb}{\Pi}
\renewcommand{\P}{\mathcal{P}}
\newcommand{\ch}{\text{\rm co}\,}
\newcommand{\cch}{\overline{\text{\rm co}}\,}

\newcommand{\cid}{\text{\tiny\rm D}}
\newcommand{\empavg}{\mathbf{F}}
\newcommand{\ba}{\mathbf{ba}}
\newcommand{\XtimesY}{\Omega_1\!\times\Omega_2}
\newcommand{\cc}{\theta}
\newcommand{\kword}[1]{{\bf #1}\index{#1}}

\newcommand{\xspace}{\mathbf{X}}
\newcommand{\yspace}{\mathbf{Y}}

\newcommand{\group}{\mathbb{G}}

\renewcommand{\sp}[1]{\left<\mkern2mu\smash{#1}\mkern2mu\right>}

\newcommand{\C}{\mathbf{C}}
\renewcommand{\L}{\mathbf{L}}

\newcommand{\gnorm}[1]{{\left\vert\kern-0.25ex\left\vert\kern-0.25ex\left\vert #1 \right\vert\kern-0.25ex\right\vert\kern-0.25ex\right\vert}}

\newcommand{\mysetminusD}{\hbox{\tikz{\draw[line width=0.6pt,line cap=round] (3pt,0) -- (0,6pt);}}}
\newcommand{\mysetminusT}{\mysetminusD}
\newcommand{\mysetminusS}{\hbox{\tikz{\draw[line width=0.45pt,line cap=round] (2pt,0) -- (0,4pt);}}}
\newcommand{\mysetminusSS}{\hbox{\tikz{\draw[line width=0.4pt,line cap=round] (1.5pt,0) -- (0,3pt);}}}
\newcommand{\mysetminus}{\mathbin{\mathchoice{\mysetminusD}{\mysetminusT}{\mysetminusS}{\mysetminusSS}}}
\renewcommand{\setminus}{\mysetminus}

\renewcommand{\xspace}{X}
\renewcommand{\yspace}{Y}

\begin{document}

\title{Global optimality under amenable symmetry constraints}
\author{Peter Orbanz}

\begin{abstract}
  {
    Consider a convex function that is invariant under an group of transformations. If it has a minimizer, does it also have an invariant minimizer? Variants of this problem appear in nonparametric
    statistics and in a number of adjacent fields. The answer depends on the choice of function, and on what one may
    loosely call the geometry of the problem---the interplay between convexity, the group,
    and the underlying vector space, which is typically infinite-dimensional. We observe that this geometry is completely
    encoded in the smallest closed convex invariant subsets of the space, and proceed to study these sets,
    for groups that are amenable but not necessarily compact.
    We then apply this toolkit to the invariant optimality problem. It yields new results on invariant kernel mean embeddings and risk-optimal invariant couplings,
    and clarifies relations between seemingly distinct ideas, such as the summation trick used in machine learning to construct equivariant neural networks and the
    classic Hunt-Stein theorem of statistics.
  }
\end{abstract}

\maketitle

\section{Introduction}

We consider the following problem: Given are a group $\group$ of linear bijections of a (topological) vector space $\xspace$,
and a function ${f:\xspace\to\mathbb{R}\cup\braces{\infty}}$ that is convex, lower semi-continuous (lsc), and invariant under $\group$.
If $f$ has a minimizer, we ask whether it also has a minimizer that is
$\group$-invariant, i.e.\ a simultaneous fixed point of all ${\phi\in\group}$.
That is obviously true if $f$ is strictly convex---since $f$ is invariant, its minimizers form an invariant set, and if
the minimizer is unique, it must itself be invariant---but there are a range of problems
where $\xspace$ is infinite-dimensional and one cannot assume strict convexity.
Examples include the Hunt-Stein theorem in statistics \citep{LeCam:1986,Eaton:George:2021},
various problems that arise in machine learning
applications to science (such as whether certain energies in a crystalline solid have symmetric ground states)
\citep[e.g.][]{Pfau:Spencer:Matthews:Foulkes:2020}, the existence of invariant optimal transportation plans,
and the principle of symmetric criticality in variational analysis \citep{Willem:1996}.
Although these problems seem to differ at first glance, we show in the following that they all share the same
underlying structure, and studying this structure in its own right leads to new applications.
We use the remainder of this section to summarize our approach and results, and postpone a detailed review
of related work to \cref{sec:related}.

\newpage
{\noindent\bf Problem sketch}.
Our vectors ${x\in\xspace}$ are typically functions or measures on a suitable space $\Omega$, which we loosely think of as
a sample space, and $f$ is a risk, energy, or similar functional on $\xspace$.
The linear transformations ${\phi:\xspace\to\xspace}$ often arise as follows: We start with 
a group $\group$ of measurable or continuous bijections ${\phi:\Omega\to\Omega}$.
Given a function ${x=h}$ or a measure ${x=\mu}$ on $\Omega$, we define
\begin{equation}
  \label{transforming:functions}
  \phi h\;:=\;h\circ\phi^{-1}
  \qquad\text{ or }\qquad
  \phi \mu\;:=\;\mu\circ\phi^{-1}\qquad\text{ for }\phi\in\group\;,
\end{equation}
i.e.~$\phi x$ is the composition or image measure. In either case, this specifies a bijection ${\phi:\xspace\to\xspace}$, which is linear even if the map
${\phi:\Omega\to\Omega}$ is not. The function $f$ is $\group$-invariant if ${f(\phi x)=f(x)}$
for all ${\phi\in\group}$ and ${x\in\xspace}$.
Now consider the optimization problem stated at the outset. Even if $f$ is not strictly convex, a simple solution exists if $\group$ is finite:
Set
\begin{equation}
  \label{summation:trick}
  \bar{x}\;=\;\mfrac{1}{|\group|}\msum_{\phi\in\group}\phi x
  \;=\;
  \mfrac{1}{|\group|}\msum_{z\in\group(x)}z
  \qquad\text{ for }x\in\xspace\;,
\end{equation}
where ${\group(x):=\braces{\phi(x)|\phi\in\group}}$ is the orbit of $x$.
Then $\bar{x}$ is $\group$-invariant. This ``summation trick'' is commonly used to construct invariant objects
\citep[e.g.][]{Minsky:Papert:1969,Shawe-Taylor:1989,Wood:Shawe-Taylor:1996}.
Clearly, the vector $\bar{x}$ satisfies
\begin{equation*}
  f(\bar{x})\;\leq\;\sup\nolimits_{\phi\in\group}f(\phi x)
  \quad\text{ in general and }\quad
  f(\bar{x})\;\leq\;f(x)
  \quad\text{ if $f$ is $\group$-invariant}.
\end{equation*}
For invariant functions, we can therefore turn an arbitrary minimizer $x$ of $f$ into a $\group$-invariant
minimizer $\bar{x}$.
\\[.5em]
We are interested in problems where $\group$ is not finite or compact.
In this case, \eqref{summation:trick} is not defined. Our strategy is as follows:
Suppose for the moment that $\group$ is countable. (The uncountable case requires some additional formalism,
but is conceptually similar, see \cref{sec:day}.)
We approximate \eqref{summation:trick}
by partial group averages
\begin{equation}
  \label{summation:trick:Folner}
  \empavg_n(x)\;:=\;\mfrac{1}{|\A_{n}|}\msum_{\phi\in\A_{n}}\phi x\;,
\end{equation}
where ${\A_1,\A_2,\ldots}$ is a sequence of finite subsets of ${\group}$ that satisfy
\begin{equation}
  \label{eq:Folner:countable}
  \frac{|\A_n\cap\phi\A_n|}{|\A_n|}\;\xrightarrow{n\rightarrow\infty}\;1
  \qquad\text{ for each }\phi\in\group\;.
\end{equation}
Such a sequence is called a \Folner sequence, and a group that contains
a \Folner sequence is called {amenable} (\cref{sec:folner}).
One can show that, if a limit point ${\bar{x}:=\lim_{n\to\infty}\empavg_{i(n)}(x)}$ along some subsequence
${i(1)<i(2)<\ldots}$ exists, then $\bar{x}$ is a $\group$-invariant element of $\xspace$.
We combine this with the fact that, if $f$ is $\group$-invariant, its sublevel sets ${[f\leq t]}$ are closed, convex and $\group$-invariant sets. It follows that
\begin{equation*}
  \group(x)\;\subset\;[f\leq f(x)]
  \quad\text{ and therefore }\quad
  \bar{x}\;\in\;[f\leq f(x)]
  \qquad\text{ for each }x\in\xspace\;,
\end{equation*}
since $\bar{x}$ is a limit of convex combinations.
If ${[f\leq f(x)]}$ is also compact, we also know that \eqref{summation:trick:Folner} has a convergent subsequence,
and our problem is solved: If $x$ is a minimizer, it is in the 
sublevel set ${\argmin f=[f\leq\min f]}$. This sublevel set also contains the invariant element $\bar{x}$, which is therefore again a minimizer.
\footnotemark
\footnotetext{
  Some of these ideas go back a long way: Amenability was first applied to optimization problems by
  G.~Hunt and C.~Stein, and the compact sublevel set argument is used, at least implicitly, by L.~Le Cam.
  The definition of amenability via \Folner sequences is not common in statistics, but is a staple of modern 
  ergodic theory. \cref{sec:related} provides details and references.
}
\\[.5em]
{\bf Approach}.
We now separate the geometry of the problem from the specific choice of $f$ as follows. Since $\bar{x}$ is a limit of convex combinations, it is always in the set
\begin{equation}
  \label{orbitope}
  \orb(x)\;:=\;
  \text{closed convex hull of }\group(x)\text{ in }\xspace\;.
\end{equation}
We call $\orb(x)$ the \kword{orbitope} of $x$.\footnotemark
\footnotetext{The term orbitope is due to \citet*{Sanyal:Sottile:Sturmfels:2011}.
  Our definition specializes to theirs if $\group$ is compact and $\xspace$ Euclidean.
  See \cref{sec:related} for references on the Euclidean case.}
This is the smallest closed convex $\group$-invariant set containing $x$,
and also contains the limit $\bar{x}$ if it exists. It therefore satisfies
\begin{equation}
  \label{intro:sublevel:inclusion}
  \bar{x}\;\in\;\orb(x)\;\subset\;[f\leq f(x)]
  \qquad\text{ for each }x\in\xspace\;\;,
\end{equation}
but does not depend on $f$. All arguments above can now be applied to $\orb(x)$ instead of the sublevel sets.
That this approach turns out to be useful is largely for two reasons:
\begin{itemize}
\item It separates the question of invariant elements from the properties of $f$---observe
  that $\orb(x)$ in \eqref{intro:sublevel:inclusion} does not depend on $f$.
  To ensure the a limit $\bar{x}$ exists, we do not need the sublevel sets to be compact; it suffices that
  $\orb(x)$ is compact.
\item The combination of amenable invariance and convexity endows orbitopes with a lot of structure.
  These structural properties can then be used to reason about
  optimization problems.
\end{itemize}
Additionally adopting the definition of amenability via \Folner sequences---which departs from common practice in statistics,
see \cref{sec:related}---clarifies the relationship between amenability and \eqref{summation:trick},
and greatly simplifies a number of arguments.
\\[.5em]
{\noindent\bf Result summary}.
In \cref{sec:day}, we establish basic properties of orbitopes.
One of our main technical tools is \cref{theorem:day}: If $\group$ is amenable and 
$f$ is a convex lsc and $\group$-invariant function, each compact orbitope $\orb(x)$ contains a
$\group$-invariant element such that
\begin{equation}
  \label{intro:day:minimizer}
  f(\bar{x})\;\leq\; f(x)\qquad\text{ and }\qquad \empavg_{i(n)}(x)\;\xrightarrow{n\to\infty}\;\bar{x}\;,
\end{equation}
where the convergence holds along a subsequence (in the sense that this subsequence exists, but is not known).
For any compact convex $\group$-invariant set $K$, we also have
\begin{equation}
  \label{intro:day:K}
  \inf\braces{f(z)\,|\,z\in K}\;=\inf\braces{f(z)\,|\,z\in K\text{ and $z$ is $\group$-invariant}}\;,
\end{equation}
and if $f$ is linear, the infima are attained at extreme points.

Since $\xspace$ is a vector space, it has a dual space $\yspace$. We see in \cref{sec:duality}
that a group acting on $\xspace$ induces a dual action on $\yspace$, and that orbitopes
of such dual actions have interesting duality properties.
A useful consequence of \eqref{intro:day:minimizer} is that, 
if ${E\subset\xspace}$ and ${F\subset\yspace}$ are suitable invariant sets, if $H(\argdot|E)$ is the
support function of $E$, and if the problem
\begin{equation}
\label{intro:minimize:H}
  \text{minimize }\;H(y|E)\quad\text{ subject to }y\in F
\end{equation}
has a minimizer, it also has a $\group$-invariant solution (\cref{HS:general}).
If $\xspace$ is a Banach space whose norm is invariant under $\group$, and $\orb(x)$ is weakly compact,
\cref{result:contraction} shows that the image $\empavg_n(\orb(x))$ contracts around an invariant element $\bar{x}$,
\begin{equation}
  \label{intro:contraction}
  \|\empavg_n(\orb(x))\|\;\xrightarrow{n\to\infty}\;0
  \qquad\text{ and }\qquad
  \bar{x}\in\empavg_n(\orb(x))\quad\text{ for all }n\in\mathbb{N}\;.
\end{equation}
This makes \eqref{intro:day:minimizer} constructive, as convergence now holds for
the entire sequence rather than an (unknown) subsequence.

Some properties of orbitopes become more concrete in specific spaces. 
In \cref{sec:hilbert:Lp,sec:probatopes}, we study three cases:
\begin{itemize}
\item The geometry of orbitopes in Hilbert spaces (\cref{sec:hilbert}) most closely resembled that of the
  convex hulls of orbits in Euclidean space studied in \citep{Sanyal:Sottile:Sturmfels:2011}. All orbitopes are weakly compact,
  and loosely speaking are contained in closed discs orthogonal to an axis consisting of all $\group$-invariant elements.
  The contraction property \eqref{intro:contraction} implies the mean ergodic theorem.
\item $\L_p$ spaces (\cref{sec:L1}) are, in a sense, the simplest spaces with non-trivial duality, i.e.\ where
  $\xspace$ and its dual cannot be identified. As an example of \eqref{intro:minimize:H}, we consider the
  Hunt-Stein theorem, see \cref{sec:hunt:stein}.
\item In \cref{sec:probatopes}, we consider orbitopes in the set $\P$ of probability measures.
  There are two natural topologies, convergence in distribution and in total variation.
  Both require some additional work, since duality is not directly applicable.
\end{itemize}
Among these, Hilbert spaces and $\P$ are, informally, the most and least similar to Euclidean space, and
orbitopes in $\P$ have distinctly non-Euclidean geometry.
\\[.5em]
{\noindent\bf Applications}.
We consider two applications in detail, one in Hilbert space and one in $\P$. In \cref{sec:mmd}, we consider
kernel mean embeddings, which are used in machine learning to represent probability distributions 
\citep{Sejdinovic:Sriperumbudur:Gretton:Fukumizu:2013}. Since the embedding space is Hilbert,
mean embeddings satisfy a strong combination of \eqref{intro:day:minimizer}, \eqref{intro:minimize:H}
and \eqref{intro:contraction},
see \cref{::mmd:day}. We characterize
the convex set of $\group$-invariant mean embeddings by a property reminiscent of 
de Finetti's theorem (\cref{result:mmd:extremal}).

In \cref{sec:mk}, we consider couplings of two $\group$-invariant probability measures $P_1$ and $P_2$.
Such couplings need not be invariant. There is hence a set $\Lambda$
of all couplings, and a subset ${\Lambda_\group\subsetneq\Lambda}$ of $\group$-invariant ones.
\cref{result:extremal:couplings} characterizes the extreme points of $\Lambda_\group$.
By the Monge-Kantorovich theorem, the expectation
$P(c)$ of a suitable cost function $c$ under a coupling $P$ satisfies
\begin{equation*}
  \min\braces{P(c)\,|\,P\in\Lambda}
  \;=\;
  \sup\braces{P_1(g_1)+P_2(g_2)|(g_1,g_2)\in\Gamma}\;,
\end{equation*}
where $\Gamma$ is a certain class of minorants of $c$. Combining \cref{result:extremal:couplings}
and \eqref{intro:day:K}
shows that, if $c$ is $\group$-invariant, one can restrict $\Lambda$ to $\Lambda_\group$ and
$\Gamma$ to invariant minorants without introducing a duality gap (\cref{result:kantorovich}).
If one adopts the transportation interpretation of couplings, this means that
if supply, demand and transportation cost
are invariant, there is an optimal transportation plan
that is also invariant, and
neither sellers nor buyers have anything to gain by setting price functions
that resolve non-invariant details. \cref{example:transport} relates this fact to work by
\citet{McGoff:Nobel:2020} on coupled dynamical systems.
\\[.5em]
{\noindent\bf Cocycles}. All results above consider elements of $\xspace$ that are invariant under $\group$.
In applications, one is often interested in different symmetry properties; equivariance and skew-invariance
are common examples. \cref{sec:cocycles} introduces a simple trick using an algebraic
structure called a cocycle. If a symmetry property can be expressed as a cocycle, it can be expressed
as invariance under a surrogate action, which makes our results applicable.

\newpage

\section{Notation and terminology}
\label{sec:preliminaries}
Throughout, we work with a Polish space $\Omega$, a topological group $\group$, and a topological
vector space $\xspace$, which is always a locally convex Hausdorff space.
The elements of $\xspace$ are often functions or measures on $\Omega$.
We write $\P(S)$ for the set of Radon probability measures on a topological space $S$.
If $\mu$ is a measure and $f$ a measurable map, ${f_*\mu:=\mu\circ f^{-1}}$ denotes the image measure.
\\[.5em]
{\bf Convention}. We call $\group$ \kword{nice} if its topology is locally compact and Polish
(equivalently, if it is locally compact, second-countable and Hausdorff).
On a nice group, there exists an invariant $\sigma$-finite measure, or Haar measure, which we
denote $|\argdot|$.
\\[.5em]
\emph{Actions}. An action $T$ of $\group$ on $\Omega$ is a map ${\group\times\Omega\to\Omega}$ that satisfies
\begin{equation*}
   T(\phi\psi,\omega)=T(\phi,T(\psi,\omega))\quad\text{ and }\quad T(\text{identity},\omega)=\omega
   \qquad\text{ for }\phi,\psi\in\group\text{ and }\omega\in\Omega\;.
\end{equation*}
We shorten notation to ${\phi(\omega):=T(\phi,\omega)}$.
The action is \kword{continuous} if $T$ is jointly continuous, \kword{measurable} if $T$ is jointly measurable,
and \kword{linear} if the map ${T(\phi,\argdot)}$ is linear for each $\phi$.
The action on $\Omega$ induces an action on the vector space of real-valued functions $f$ on $\Omega$,
and, if it is measurable, on the vector space of signed measures $\mu$ on $\Omega$, which we have already
defined in \eqref{transforming:functions}. These actions cohere with each other:
Since a function $h$ is $\phi_*\mu$-integrable iff ${h\circ\phi}$ is $\mu$-integrable, and
\begin{equation}
  \label{eq:image:measure:int}
  \mint h d(\phi_*\mu)\;=\;\mint(h\circ\phi)d\mu
  \qquad\text{ or in short }\qquad
  \phi_*\mu(h)\;=\;\mu(h\circ\phi)\;.
\end{equation}
Function or measures are \kword{$\group$-invariant} if they satisfy ${h\circ\phi=h}$ or ${\phi_*\mu=\mu}$ for all ${\phi\in\group}$,
that is, if they are invariant under the actions in \eqref{transforming:functions}.
A function with two arguments is called \kword{diagonally} $\group$-invariant if
${h(\phi\upsilon,\phi\omega)=h(\upsilon,\omega)}$ holds for all ${\phi\in\group}$ and all
${\upsilon,\omega\in\Omega}$. It is \kword{separately} $\group$-invariant if 
${h(\phi\upsilon,\psi\omega)=h(\upsilon,\omega)}$ for all pairs ${\phi,\psi\in\group}$.
Clearly, separate implies diagonal invariance.

\section{The basic objects}
\label{sec:day}

All results in the following involve two basic types of objects, a general form of the partial group average
$\empavg_n(x)$ in \eqref{summation:trick:Folner}, which we define next, and orbitopes. \cref{theorem:day}
below relates the two to each other.

\subsection{\Folner averages}
\label{sec:folner}

To generalize \eqref{summation:trick:Folner} to uncountable groups,
we approximate a possibly non-compact $\group$ from within by compact subsets.
A \kword{\Folner sequence} is a sequence ${\A_1,\A_2,\ldots}$ of compact subsets of $\group$ such that
\begin{align}
  \label{eq:folner}
  {|\A_n\cap K\A_n|}\,&/\,{|\A_n|}\;\xrightarrow{n\rightarrow\infty}\;1\qquad\text{ for all compact }K\subset\group\;,
\end{align}
where $K\A_n$ is the set $\braces{\phi\psi|\phi\in K,\psi\in\A_n}$.
A nice group contains a \Folner sequence if and only if it is \kword{amenable}
\citep{Bekka:delaHarpe:Valette:2008,Grigorchuk:Harpe:2017}.
If $\group$ is countable, compact sets are finite, and a sequence is a \Folner sequence if
and only if it satisfies \eqref{eq:Folner:countable}.
\begin{example}
  (i) Let $\mathbb{S}_n$ be the group of permutations of $n$ elements. Then
  ${\mathbb{S}=\cup_{n\in\mathbb{N}}\mathbb{S}_n}$ is the (countable) group of finitely supported permutations of $\mathbb{N}$,
  and $(\mathbb{S}_n)$ is a \Folner sequence in
  $\mathbb{S}$, since every ${\phi\in\mathbb{S}}$ satisfies ${\phi\mathbb{S}_n=\mathbb{S}_n}$ for $n$ large enough.
  \\[.2em]
  (ii) If $\group$ is a normed vector space with addition as group operation, and $\A_n$ is the closed norm ball of radius
  $n$ around the origin, $(\A_n)$ is a \Folner sequence.
  \\[.2em]
  See \cref{appendix:amenable} for more examples of amenable groups.
\end{example}
An average in the (possibly infinite-dimensional) space $\xspace$ is an $X$-valued integral.
By the integral of a measurable function ${f:S\to\xspace}$ on some measure space $(S,\mu)$, we mean
the unique element ${\mu(f)=\int fd\mu}$ of $\xspace$ that satisfies
\begin{equation*}
  \ell(\mu(f))\;=\;\int_{S}(\ell\circ f)d\mu
  \qquad\text{ for every continuous linear }\ell:X\to\mathbb{R}\;,
\end{equation*}
where the integral on the right is the real-valued Lebesgue integral.
If this integral exists, it is unique.\footnote{The integral
$\mu(f)$ is known as the weak integral, the Pettis integral, or as the Gelfand integral
if $X$ has a weak* topology. If $\xspace$ is a Banach space and ${\int\|f\|d\mu<\infty}$,
$\mu(f)$ coincides with the strong (or Bochner) integral. If $\xspace$ is Euclidean,
it reduces to the Lebesgue integral. See \citep{Aliprantis:Border:2006} for more on such integrals.} 
Given a \Folner sequence, we define the \kword{\Folner average}
\begin{equation}
  \label{eq:Fn}
  \empavg_n(x)\;:=\;\mfrac{1}{|\A_n|}\mint_{\A_n}\phi(x)|d\phi|\;,
\end{equation}
provided the integral exists. (In terms of the definition above, this means we choose $S$ as $\A_n$ and
${f(\phi):=\phi x}$.)

\begin{example}
  (i) \Folner averages under finite groups are of the form \eqref{summation:trick}.
  \\[.2em]
  (ii) The sample average ${n^{-1}\sum_{i\leq n}g(\omega_i)}$ of a statistic $g$ over observations ${\omega_1,\ldots,\omega_n}$ is a \Folner average in disguise: Define a function $x$ on infinite sequences ${\omega=(\omega_1,\omega_2,\ldots)}$ as
  ${x(\omega):=g(\omega_1)}$. Then
  \begin{equation*}
    \mfrac{1}{n}\msum_{i\leq n}g(\omega_{i})
    \;=\;
    \mfrac{1}{|\mathbb{S}_n|}\msum_{\phi\in\mathbb{S}_n}g(\omega_{\phi(1)})
    \;=\;
    \mfrac{1}{|\mathbb{S}_n|}\mint_{\mathbb{S}_n}x(\phi\omega)|d\phi|
    \;=\;
    (\empavg_n(x))(\omega)
  \end{equation*}
  (iii) Window estimators for time series and random fields are \Folner averages over shift
  groups, and subgraph counts in network analysis over permutation groups
  \citep{Austern:Orbanz:2022}. In other words, one can generalize the sample average
  in (ii) by changing the group and the action. These generalized sample averages
  have a law of large numbers \citep{Lindenstrauss:2001} and a central limit theorem
  \citep{Austern:Orbanz:2022}.
  \\[.2em]
  (iv) The $\group$-invariant solution to an invariant testing
  problem in the Hunt-Stein theorem is always a limiting \Folner average,
  as we will see in \cref{sec:hunt:stein}.
\end{example}

\begin{figure}
  \makebox[\textwidth][c]{
    \begin{tikzpicture}
      \begin{scope}[xshift=0]
        \node at (0,0) {\includegraphics[height=2cm]{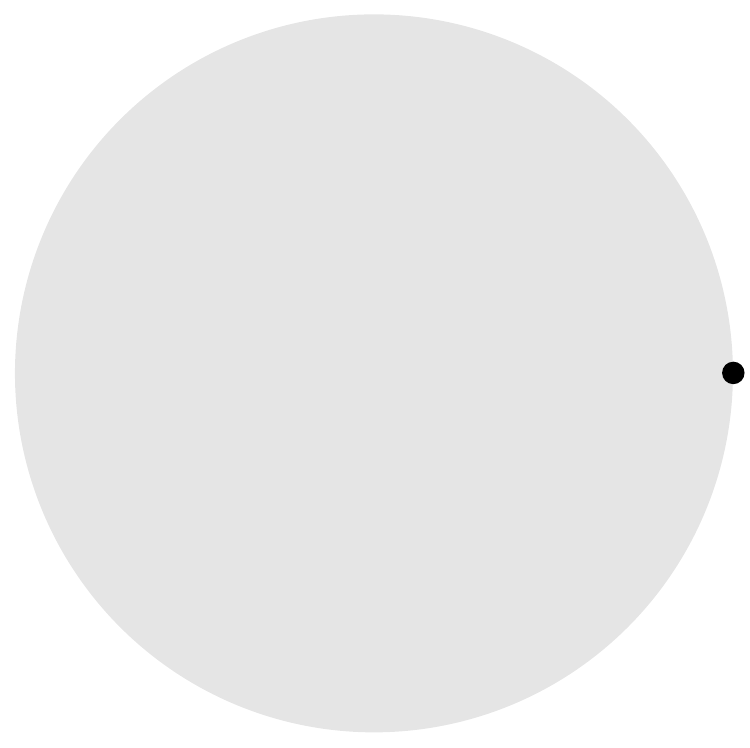}};
        \node[scale=.15,gray,fill,circle] at (0,0) {};
        \node at (1.3,0) {\scriptsize $x$};
      \end{scope}
      \begin{scope}[xshift=3.5cm]
        \node at (0,0) {\includegraphics[height=2cm]{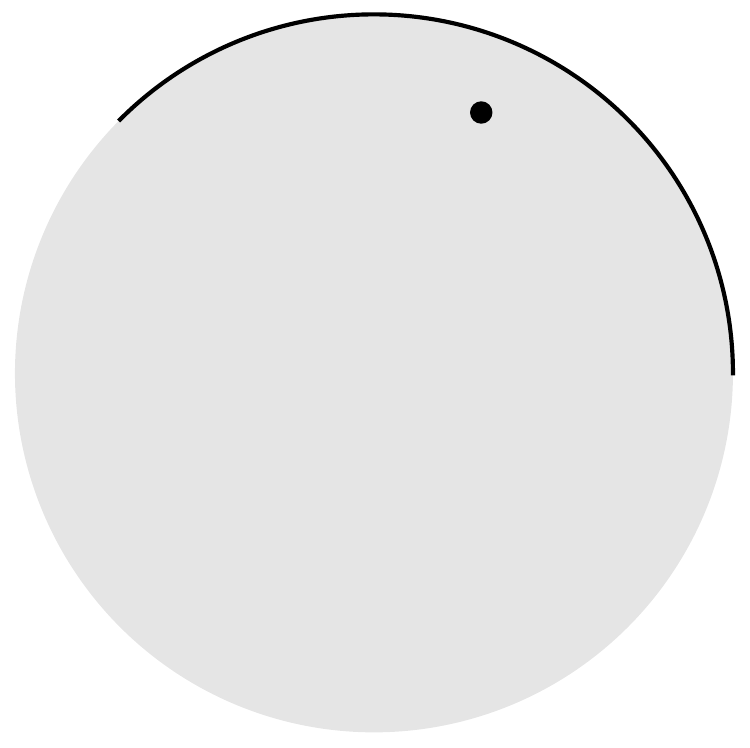}};
        \node[scale=.15,gray,fill,circle] at (0,0) {};
        \node at (0,-1.3) {\scriptsize ${\A_1=[0,3\pi/4]}$};
      \end{scope}
      \begin{scope}[xshift=7cm]
        \node at (0,0) {\includegraphics[height=2cm]{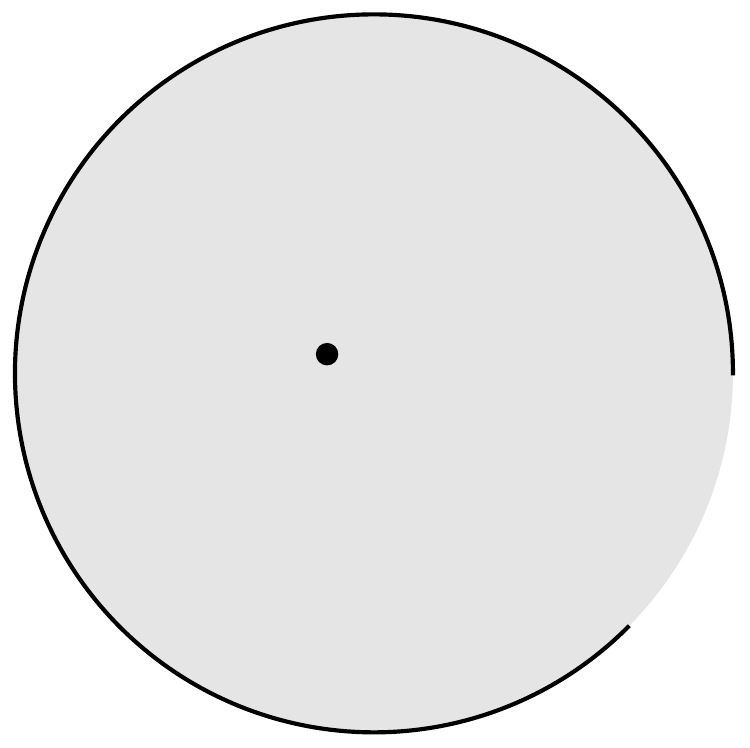}};
        \node[scale=.15,gray,fill,circle] at (0,0) {};
        \node at (0,-1.3) {\scriptsize ${\A_2=[0,7\pi/4]}$};
      \end{scope}
      \begin{scope}[xshift=10.5cm]
        \node at (0,0) {\includegraphics[height=2cm]{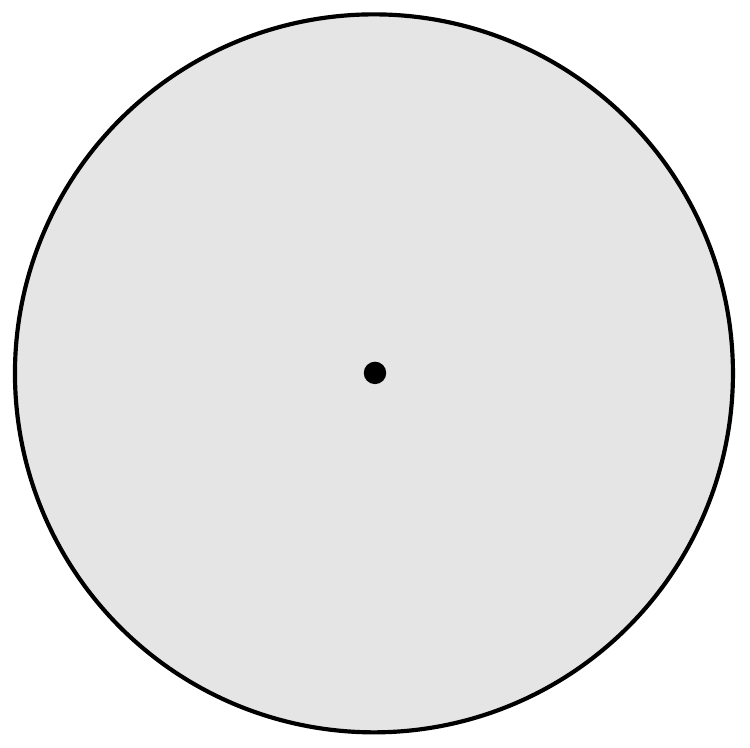}};
        \node at (0,-1.3) {\scriptsize ${\A_3=[0,2\pi]=\group}$};
      \end{scope}
    \end{tikzpicture}
  }
  \caption{
    Finite-dimensional \Folner averages in the (trivial) case where $\group$ is compact.
    Here, the rotation group $\group$ acts on ${\xspace=\mathbb{R}^2}$.
    \emph{Left}: A point $x$ and its orbitope (the gray disc).
    \emph{Middle left/right}: The average $\empavg_n(x)$ is the barycenter of the
    uniform distribution on the black line segment $\A_n(x)$. It is not an element of the orbit $\group(x)$, but
    is in $\orb(x)$.
    \emph{Right}: If ${\A_n=\group}$ the barycenter $\empavg_n(x)$ is $\group$-invariant.
  }  
  \label{fig:empavg}
\end{figure}
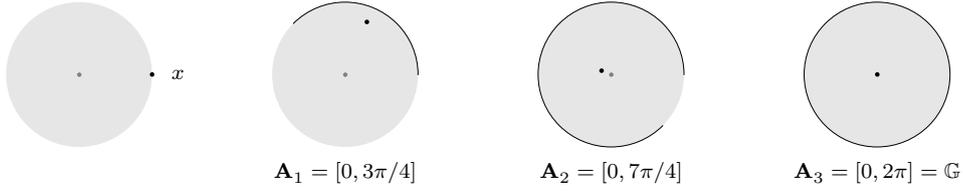

\subsection{Orbitopes}

We define the orbitope of ${x\in\xspace}$ as in \eqref{orbitope}. The next result collects some  basic properties.
Recall that the \kword{barycenter} of a Radon probability measure $P$ on $\xspace$ is the unique element $\sp{P}$ of $\xspace$
that satisfies
\begin{equation*}
  P(\ell)\;=\;\ell(\sp{P})\qquad\text{ for every linear continuous }\ell:\xspace\rightarrow\mathbb{R}\;,
\end{equation*}
provided such an element exists \citep{Phelps:2001}. This is the integral ${\sp{P}=\int xP(dx)}$, if the integral exists.
If $M$ is a set of Radon measures, $\sp{M}$ denotes its set of barycenters.
\begin{lemma}
  \label{lemma:orbitope:invariant}
  If $\group$ acts linearly and continuously on $X$, the following holds:
  \\[.5em]
  (i) Each orbitope is a $\group$-invariant set, and so are its interior and boundary.
  A closed convex $\group$-invariant set contains $x$ if and only it contains $\orb(x)$.
  \\[.5em]
  (ii)
  In a Banach space, $\orb(x)$ is norm compact (resp.~weakly compact) if and only if the norm closure (resp.~weak closure)
  of the orbit $\group(x)$ is compact.
  \\[.5em]
  (iii) If $\orb(x)$ is compact, its extreme points lie in the closure of $\group(x)$, and
  \begin{equation*}
    \Pi(x)\;=\;\sp{\P(\text{\rm closure}(\group x))}\;=\;\sp{\P(\Pi(x))}
  \end{equation*}
  (iv)
  The orbitope of a convex combination ${x=\sum c_ix_i}$ of vectors ${x_1,\ldots,x_n\in\xspace}$ is the set
  ${\orb(x)\;=\;\braces{\,\sum c_iz_i\,|\,z_i\in\orb(x_i)}}$.
\end{lemma}
\begin{proof}
  See \cref{proofs:sec:day}.
\end{proof}
The orbits of $\group$ form a partition of $X$ into disjoint sets, and every $\group$-invariant set is a disjoint
union of orbits. The orbits of $\group$ are therefore the smallest $\group$-invariant subsets of $X$.
By \cref{lemma:orbitope:invariant}, orbitopes are similarly the smallest closed convex $\group$-invariant
sets (although they are not generally mutually disjoint).
In particular, an element ${z\in\xspace}$ satisfies ${z\in\Pi(x)}$ if and only if ${\Pi(z)\subset\Pi(x)}$.  
It follows that any two orbitopes $\orb_1$ and $\orb_2$ of the same action in $X$ satisfy
\begin{equation*}
  \orb_1\subset\orb_2
  \quad\text{ or }\quad
  \orb_1\supset\orb_2
  \quad\text{ or }\quad
  \orb_1\cap\orb_2=\varnothing\;.
\end{equation*}

\begin{example} \cref{fig:day} gives a simple finite-dimensional example of an orbitope, for a compact group.
  As an infinite-dimensional example, consider the set $\P$ of probability measures on $\Omega$, topologized by
  convergence in distribution. The vector space $\xspace$ is the span of $\P$.
  \\[.5em]
  (i) If $\group$ acts continuously on $\Omega$, the induced action \eqref{transforming:functions} on $\P$ is
  continuous. Choose any ${P\in\P}$. The orbit $\group(P)$ is known in statistics as the \kword{group family}
  of $P$. By \cref{lemma:orbitope:invariant}, $\orb(P)$ is the set of mixtures of distributions in the closure
  of $\group(P)$.
  \\[.5em]
  (ii) If the group is compact, the orbits are closed, so $\orb(P)$ consists of the mixtures of orbit elements. For instance,
  let $P_t$ denote the isotropic Gaussian on $\mathbb{R}^d$ with mean $t$, and $\group$ the group
  of rotations around the origin. Then $\group(P_t)$ is the set $\braces{P_{s}|s\in C}$, for the circle $C$ of radius $\|t\|$,
  and ${\orb(P_t)=\braces{\int_{C}P_s\mu(ds)|\mu\in\P(C)}}$.
  \\[.5em]
  (iii) Let $\delta_t$ be a point mass at ${t\in\mathbb{R}^2}$, and let ${\group\cong\mathbb{R}}$ be the group of vertical shifts on the plain.
  The orbit of $t$ is the straight vertical line $L_t$ through $t$, which is an affine subspace, and closed.
  We show in \cref{sec:probatopes} that ${\orb(\delta_t)=\P(L_t)}$, and that it is a closed face of the convex set of probability measures on $\mathbb{R}^2$.
  \\[.5em]
  (iv) Here is a finite-dimensional example, which can be found in \citep{Atiyah:1982}: Consider a vector ${\lambda\in\mathbb{R}^n}$, and let
  $\group$ be the finite group of permutations of the coordinates. One can then show that the orbitope of $\lambda$ is
  \begin{equation*}
    \orb(\lambda)\;=\;\braces{\text{diagonal of }A\,|\,A\in\mathbb{R}^{n\times n}\text{ is Hermitian and has eigenvalues }\lambda_1,\ldots,\lambda_n}\;.
  \end{equation*}
\end{example}

\subsection{A variant of Day's theorem}

\begin{figure}
  \makebox[\textwidth][c]{
    \resizebox{!}{3.5cm}{
\begin{tikzpicture}
  \begin{scope}[scale=.75]
\fill[
  top color=gray!20,
  bottom color=gray!20,
  shading=axis,
  ] 
  (0,4.5) circle (2cm and 0.5cm);
\fill[
  left color=gray!20!white,
  right color=gray!20!white,
  middle color=gray!20,
  shading=axis,
  ] 
  (2,4.5) -- (0,0) -- (-2,4.5) arc (180:360:2cm and 0.5cm) coordinate[near start] (x);
\draw (-2,4.5) arc (180:360:2cm and 0.5cm) -- (0,0) -- cycle;
\draw (-2,4.5) arc (180:0:2cm and 0.5cm);
\draw[dashed] (0,0) -- (0,4.5) ;
\draw[thick] (0,4.5) -- (0,6) ;
\draw[thick] (0,0) -- (0,-.5) ;
\node[circle,black,fill,scale=.4,label=above right:$x$] at (x) {};
  \end{scope}
  \begin{scope}[xshift=6cm,scale=.75]
\fill[
  top color=gray!20,
  bottom color=gray!20,
  shading=axis,
  ] 
  (0,4.5) circle (2cm and 0.5cm);
\fill[
  opacity=0
  ] 
  (2,4.5) -- (0,0) -- (-2,4.5) arc (180:360:2cm and 0.5cm) coordinate[near start] (x);
\draw (-2,4.5) arc (180:360:2cm and 0.5cm);
\draw (-2,4.5) arc (180:0:2cm and 0.5cm);
\draw[dashed] (0,0) -- (0,4.5) ;
\draw[thick] (0,4.5) -- (0,6) ;
\draw[thick] (0,0) -- (0,-.5) ;
\node[circle,black,fill,scale=.4,label=below left:$x$] at (x) {};
\node[circle,black,fill,scale=.4,label=right:$\bar{x}$] at (0,4.5) {};
\end{scope}
\begin{scope}[xshift=12cm,yshift=.5cm,scale=.75]
  \node (t) at (0,4) {};
  \node (b) at (0,0) {};
  \node (x) at (-2,2) {};
  \node (y) at (2,2) {};
  \draw[fill=gray!20!white] (x.center)--(t.center)--(y.center)--(b.center)--cycle;
  \draw[thick,black] (0,4)--(0,5);
  \draw[dashed,black] (0,0)--(0,4);
  \draw[thick,black] (0,-1)--(0,0);
  \node[circle,fill,black,scale=.4,label=left:$x$] at (x) {};
  \draw[thick,black] (x.center)--(y.center);
  \node[circle,fill,black,scale=.4,label=above right:$\bar{x}$] at (0,2) {};
\end{scope}
\end{tikzpicture}
  }}
  \caption{
    Finite-dimensional illustration of some of the sets in \cref{theorem:day}.
    \emph{Left}: A compact set $K$ in which $x$ is extreme. $K$ is invariant under rotations around its middle axis,
    and $K_\group$ is the intersection of this axis and $K$.
    \emph{Middle}: The orbitope $\orb(x)$ is a closed disc that contains a single invariant element ${\bar{x}}$.
    \emph{Right}: A compact set invariant under reflections over the vertical axis. The orbitope of $x$ is the
    intersection of $K$ with the horizontal line through $x$. Even though $x$ is an extreme of $K$, $\bar{x}$ is not.
  }
  \label{fig:day}
\end{figure}
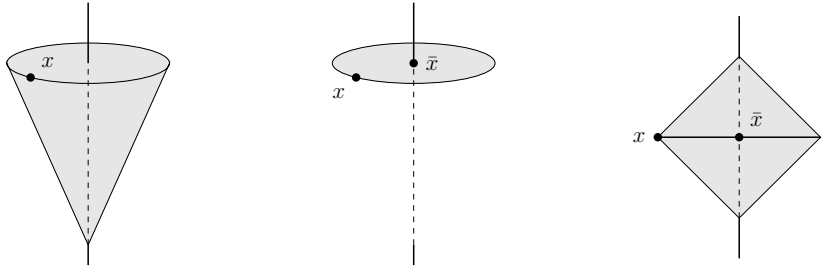

Our next step is to make the relationship between compactness of orbitopes and the existence of $\group$-invariant elements
precise. Compact orbitopes of amenable groups always contain an invariant element. We had already sketched a proof of this
fact in the introduction for countable groups. For general amenable groups, it
follows from a classic fixed-point theorem of M.~M. Day \citep{Day:1961}:
\begin{fact}[Day]
  \label{fact:day}
  If a compact convex set $K$ in a locally convex Hausdorff space is invariant under a linear continuous action of a nice amenable group,
  $K$ has a $\group$-invariant element.
\end{fact}
The next result refines Day's theorem for the purposes of optimization, and is our main technical tool.
\begin{theorem}
  \label{theorem:day}
  Let a nice group $\group$ act linearly and continuously on a locally convex vector space $\xspace$. Let $(\A_n)$ be a \Folner sequence
  in $\group$, and $x$ a point in $\xspace$.\\[.2em]
  (i) If $\Pi(x)$ is compact, it contains a $\group$-invariant element $\bar{x}$ that satisfies
  \begin{equation}
    \label{eq:day}
    f(\bar{x})
    \;\leq\;
    \sup_{\phi\in\group}f(\phi(x))
    \qquad\text{ for every convex lsc }f:{\orb(x)}\rightarrow\mathbb{R}\cup\braces{\infty}\;.
  \end{equation}
  (ii)
  The integral $\empavg_n(x)$ exists for each ${n\in\mathbb{N}}$. If $\orb(x)$ is compact,
  a subsequence of ${(\empavg_{n}(x))}$ converges to a $\group$-invariant element 
  $\bar{x}$ of $\orb(x)$ that satisfies \eqref{eq:day}.
  \\[.2em]
  (iii) \def\ell{g}
  The set ${\xspace_\group:=\braces{x\in\xspace|x\text{ is $\group$-invariant}}}$ is a closed linear
  subspace of $\xspace$. Any compact, convex and $\group$-invariant set ${K\subset\xspace}$ satisfies
  \begin{equation*}
    \inf_{\;z\,\in\, K\phantom{_\group}}\ell(z)\;=\inf_{z\,\in\,K\cap\xspace_\group}\ell(z)
    \quad\text{ for every $\group$-invariant convex lsc }\ell:K\rightarrow\mathbb{R}\cap\braces{\infty}\;.
  \end{equation*}
  If $\ell$ is linear, both infima are attained at extreme points.
\end{theorem}
\begin{proof}
  See \cref{proofs:sec:day}.
\end{proof}
Informally, this implies orbitopes are most useful for optimization if they are compact
(so that $\orb(x)$ contains an invariant element) and small (so that this element is close
to $x$). The size of $\orb(x)$ depends both on the
orbit, and on the choice topology of $\xspace$. This choice involves a trade-off, since
stronger topologies result in smaller orbitopes, whereas weaker topologies have more compact sets.

\section{Duality}
\label{sec:duality}

Let $\yspace$ be the dual space of $\xspace$, that is, the set of all continuous maps
${y:\xspace\rightarrow\mathbb{R}}$. We use the customary notation ${\sp{x,y}:=y(x)}$ for ${x\in X}$
and ${y\in\yspace}$. If $X$ is a normed space, the dual norm on $\yspace$ is ${\|y\|:=\sup\braces{\sp{x,y}|\,\|x\|\leq 1}}$.

\begin{remark*}
  Recall a topology on $\xspace$ (resp.~$\yspace$) is \kword{consistent} with the dual pair $\sp{\xspace,\yspace}$ if it defines
  the same linear functionals as the weak (resp.~weak*) topology. If $\xspace$ is normed, the norm topology
  is consistent, whereas the dual norm topology on $\yspace$ is generally not.  
  It is helpful to keep in mind that all consistent topologies have the same closed convex set and the same convex lsc functionals.
  See \citep{Aliprantis:Border:2006,Borwein:Vanderwerff:2010}.
\end{remark*}
An orbitope in $\xspace$ or $\yspace$ is
\kword{consistent} if it defined by a consistent
topology. Since all consistent topologies have the same closed convex sets, each point has only one consistent orbitope.

\subsection{Dual actions}
Given a linear continuous action of $\group$ on $\xspace$, we define the \kword{dual action}
as the unique action of $\group$ on $\yspace$ that satisfies
\begin{equation}
  \label{dual:action}
  \sp{\phi x,\phi y}\;=\;\sp{x,y}
  \qquad\text{ for each }\phi\in\group\text{ and }(x,y)\in\xspace\!\times\!\yspace\;.
\end{equation}
We then say that $\group$ \kword{acts dually} on $X$ and $Y$.
Since \eqref{dual:action} is equivalent to choosing each map
${\phi:Y\rightarrow Y}$ as the adjoint of the bounded linear operator ${\phi^{-1}:X\rightarrow X}$,
the dual action is linear and weak* continuous. An action on a normed space has \kword{bounded orbits}
if ${\sup_{\phi\in\group}\|\phi x\|<\infty}$ for all ${x\in\xspace}$. It is
\kword{isometric} if it leaves the norm invariant, ${\|\phi x\|=\|x\|}$ for all $\phi$ and $x$.
Clearly, isometric actions have bounded orbits.
\begin{lemma}
  \label{lemma:alaoglu}
  Consider a linear continuous action on a normed space $X$.
  \\[.5em]
  (i) The action has bounded orbits if and only if the dual action had bounded orbits.
  \\[.5em]
  (ii) The action is isometric if and only if the dual action is isometric.
  \\[.5em]
  (iii) If the orbits are bounded, all consistent orbitopes in $\yspace$ are weak* compact.
\end{lemma}
\begin{proof}
  See \cref{proof:sec:duality}.
\end{proof}

\subsection{Support functions}
The support function of a set $E$ in $X$ or $Y$ is defined as
\begin{equation*}
  H(\argdot|E)\;:=\;
  \begin{cases}
  \;\sup\braces{\sp{\argdot,y}|\,y\in E} & \text{ if }E\subset\yspace\\
  \;\sup\braces{\sp{x,\argdot}|\,x\in E} & \text{ if }E\subset\xspace\\
  \end{cases}\;.
\end{equation*}
This function is convex, and weakly lsc or weak* lsc.
It takes values in $\mathbb{R}\cup\braces{\infty}$, is proper if $E$ is not empty, and is finite if $E$ is compact, see
\citep{Borwein:Vanderwerff:2010,Aliprantis:Border:2006}.
If $\group$ acts dually, substituting \eqref{dual:action} into the definition of $H$ shows that
$H$ characterizes invariance of $E$,
\begin{equation*}
  H(\argdot|E)\text{ is $\group$-invariant }
  \qquad\Leftrightarrow\qquad
  E\text{ is $\group$-invariant.}
\end{equation*}
Since a set and its closed convex hull have the same support function \citep{Borwein:Vanderwerff:2010},
we also have 
\begin{equation}
  \label{eq:H:equals:sup}
  H(\argdot|\orb(y))\;=\;H(\argdot|\group(y))\;.
\end{equation}
This implies the duality
\begin{equation}
  \label{eq:support:function:duality}
  H(x|\orb(y))\;=\;H(y|\orb(x))\qquad\text{ for }(x,y)\in\xspace\!\times\!\yspace\;.
\end{equation}
\cref{theorem:day} implies a simple but useful optimality principle:
\begin{proposition}
  \label{HS:general}
  Let a nice amenable group $\group$ act dually on a nice space $X$
  and its dual $Y$, with bounded orbits. Let ${E\subset X}$ be $\group$-invariant, and
  ${F\subset Y}$ closed, convex and $\group$-invariant. If the
  optimization problem
  \begin{equation*}
    \text{minimize }\quad H(x|E)\qquad\text{ subject to }y\in F
  \end{equation*}
  has a solution $\hat{y}$, the orbitope $\orb(\hat{y})$ contains a $\group$-invariant solution $\bar{y}$, and
  \begin{equation*}
    -H(-x\,|\orb(\hat{y}))\;\leq\;\sp{x,\bar{y}}\;\leq\;H(x\,|\orb(\hat{y}))
    \qquad\text{ for each }x\in E\;.
  \end{equation*}
\end{proposition}
\begin{proof}
  The orbitope $\orb(\hat{y})$ is compact, by \cref{lemma:alaoglu}. It contains a $\group$-invariant element $\bar{y}$, by \cref{theorem:day}.
  Since $H(\argdot|E)$ is convex lsc and $\group$-invariant, $\argmin H(\argdot|E)$ is a convex, closed, $\group$-invariant set.
  Since orbitopes are the smallest such sets, ${\hat{y}\in\argmin H(\argdot|E)}$ implies ${\bar{y}\in\orb(\hat{y})\subset\argmin H(\argdot|E)}$.
  To obtain the upper bound, observe that
  \begin{equation*}
    \sp{x,\bar{y}}
    \;\leq\;
    \sup\nolimits_{z\in\orb(\bar{y})}
    \sp{x,z}
    \;\leq\;
    \sup\nolimits_{z\in\orb(\hat{y})}
    \sp{x,z}
    \;=\;
    H(\argdot|\orb(\hat{y}))\;,
  \end{equation*}
  since ${\bar{y}\in\orb(\hat{y})}$ implies ${\orb(\bar{y})\subset\orb(\hat{y})}$.
  The lower bound follows analogously.
\end{proof}

\subsection{Orbitopes and annihilators}
\label{:::annihilators}

In the examples in \cref{fig:empavg,fig:day}, all orbitopes are orthogonal
to the subspace $\xspace_\group$ of invariant elements. The next result shows that this is no coincidence.
To say that $\orb(x)$ is orthogonal to $\xspace_\group$ is equivalent to saying that the shifted orbitope
${\orb(x)-x}$ is in the orthocomplement of $\xspace_\group$. In both figures, $\xspace$ is Euclidean and hence a
Hilbert space. In a general locally convex space, there is no notion of orthogonality, and the
orthocomplement in $\xspace$ generalizes to the annihilator in the dual space,
\begin{equation*}
    \xspace_\group^\perp\;:=\;\braces{y\in\yspace\,|\,\sp{x,y}=0\text{ for all }x\in\xspace_\group}\;.
\end{equation*}
Observe also that the shifted orbitope is the set
\begin{equation*}
  \orb(x)-x
  \;=\;
  \cch(\group(x)-x)
  \;=\;
  \cch\braces{\phi x-x|\phi\in\group}\;.
\end{equation*}
The map ${\phi\mapsto\phi(x)-x}$ is known as a \kword{coboundary} of $x$
\citep[e.g.][]{Bekka:delaHarpe:Valette:2008}.
The set $\orb(x)-x$ is hence the closed convex hull of the image $\group x-x$ of a coboundary.
\begin{lemma}
  \label{::annihilators}
  If $\group$ acts dually on $\xspace$ and its dual $\yspace$, the
  vectors ${\phi y-y}$, for ${\phi\in\group}$ and ${y\in\yspace}$, form a dense subset of
  the annihilator $\xspace_\group^\perp$,
  and the vectors ${\phi x-x}$ similarly lie dense in $\yspace_\group^\perp$. All orbitopes satisfy
  \begin{equation*}
    \orb(x)\;\subset\;\yspace_\group^\perp
    \qquad\text{ and }\qquad
    \orb(y)\;\subset\;\xspace_\group^\perp\;
    \qquad\text{ for }x\in\xspace\text{ and }y\in\yspace\;.
  \end{equation*}
\end{lemma}
\begin{proof}
  See \cref{proof:sec:duality}.
\end{proof}

\subsection{Contraction under averages}

If $\empavg_n(z)$ is well-defined at each point $z$ in an orbitope $\orb$, we may consider the image set ${\empavg_n(\orb)=\braces{\empavg_n(z)|\,z\in\orb}}$.
If $\orb$ has an invariant element $\bar{x}$, one would expect this image to contract around $\bar{x}$ as $n$ grows.
If $\xspace$ has a norm, we can quantify contraction by the width
\begin{equation*}
  \|\empavg_n(\orb)\|\;:=\;\sup\braces{\|z-z'\|\,|\,z,z'\in\empavg_n(\orb)}\;=\;\sup\braces{\|\empavg_n(x-x')\|\,|\,{x,x'\in\orb}}\;.
\end{equation*}
For isometric actions of amenable groups, orbitopes indeed contract if they are compact.\nolinebreak
\begin{theorem}[Contraction of orbitopes]
  \label{result:contraction}
  Let an amenable group $\group$ act isometrically on a Banach space $X$.
  \\[.5em]
  (i) If an orbitope $\orb$ in $\xspace$ is weakly compact, the
  convex function ${x\mapsto\|\empavg_n(x)\|}$ is well-defined and weakly lsc on $\orb$, and satisfies ${\|\empavg_n(\orb)\|\to 0}$
  as ${n\to\infty}$.
  \\[.5em]
  (ii) The function ${y\mapsto\|\empavg_n(y)\|}$ on the dual space is convex and norm lsc everywhere on $\yspace$. Each point ${y\in\yspace}$,
  and each $z$ in the convex hull of $\group(y)$, satisfy
  \begin{equation}
    \label{eq:contraction:2}
    \|\empavg_n(z-y)\|\;\leq\;\mfrac{|\A_n\vartriangle\phi_z^{-1}\A_n|}{|\A_n|}\|y\|
    \quad\text{ for some ${\phi_z\in\group}$ and all }n\in\mathbb{N}\;.
  \end{equation}
\end{theorem}
\begin{proof}
  See \cref{proof:sec:duality}.
\end{proof}

\begin{remark*} We note how the behavior of $\empavg_n$ on $\xspace$ differs from that on the dual:\\[.5em]
(i) On $\xspace$, \cref{result:contraction} applies only to points with compact orbitope, but at these, we 
  we obtain a strong statement: The entire sequence ${(\empavg_n(z))}$ converges,
  rather than a subsequence as in \cref{theorem:day}. This holds for all ${z\in\orb}$
  uniformly over $\orb$, and the limits coincide for all $z$. This limit is an invariant element of $\orb$.
  \\[.5em]
  (ii) On $\yspace$, \eqref{eq:contraction:2} holds everywhere, but pairwise rather than uniformly, as $\phi_z$ depends on $z$.
  It also does not extend from ${\ch\group(y)}$ to the closure $\orb(y)$, because
    ${\|\empavg_n(\argdot)\|}$ is not generally weak* lsc.
\end{remark*}

\section{Orbitopes I: Hilbert and Lebesgue spaces}
\label{sec:hilbert:Lp}

This and the following section study properties of orbitopes in specific spaces. The simplest case are Hilbert spaces,
where orbitopes are orthogonal to the subspace of invariant elements, and \cref{result:contraction}
becomes the mean ergodic theorem. We also consider $\L_p$ spaces (which are not Hilbert unless ${p=2}$).
\cref{sec:hunt:stein} uses the Hunt-Stein theorem to illustrate results of \cref{sec:duality}.

\subsection{Orbitopes in Hilbert spaces}
\label{sec:hilbert}

If $\xspace$ is a Hilbert space, it can be identified with its dual, and the dual pairing
becomes an inner product. If $\group$ acts dually, it therefore leaves the inner
product invariant,
\begin{equation}
  \label{H:inner:product:invariant}
  \sp{\phi x,\phi y}=\sp{x,y}
  \quad\text{ and hence }\quad
  \|\phi x\|=\sqrt{\sp{\phi x,\phi x}}=\|x\|
  \quad\text{ for all }\phi\in\group\;.
\end{equation}
Linear actions that satisfy \eqref{H:inner:product:invariant} are called \kword{unitary},
since \eqref{H:inner:product:invariant} makes each map ${\phi:\xspace\to\xspace}$
a unitary linear operator \citep{Einsiedler:Ward:2011}. Since $\xspace$
has an inner product, the annihilator ${\xspace_\group^\perp}$ is the
orthocomplement of $\xspace_\group$, and each element $x$ of $\xspace$ has a unique
decomposition
\begin{equation*}
  x\;=\;\bar{x}\,+\,x^\perp
  \qquad\text{ where }
  \bar{x}\;:=\;\text{projection onto }\xspace_\group 
  \qquad\text{ and }\qquad
  x^{\perp}\;\in\;\xspace_\group^\perp\;.
\end{equation*}
Since the weak and weak* topology coincide in Hilbert spaces, the norm-, weak- and weak* topologies 
define the same orbitopes. 
Orbitopes of unitary actions are contained in closed discs orthogonal to the (possibly infinite-dimensional) axis $\xspace_\group$:
\begin{proposition}
  \label{result:orbitope:unitary:action}
  If $\xspace$ is a Hilbert space and \eqref{H:inner:product:invariant} holds, $\orb(x)$
  is weakly compact and orthogonal to $\xspace_\group$ for each ${x\in\xspace}$. All ${y\in\orb(x)}$ satisfy
  ${\bar{y}=\bar{x}}$ and ${\|y\|\leq\|x\|}$. If $\group$ is amenable, then
  ${\orb(x)\cap\xspace_\group=\braces{\bar{x}}}$.
\end{proposition}
\begin{proof}
  See \cref{proofs:sec:hilbert:Lp}.
\end{proof}
\begin{fact}[Mean ergodic theorem {\citep[e.g.][]{Einsiedler:Ward:2011,Weiss:2003:1}}]
  \label{fact:mean:ergodic:theorem}
  Let a nice amenable $\group$ acts unitarily on a Hilbert space $\xspace$. Then
  \begin{equation*}
    \Big\|\,\mfrac{1}{|\A_n|}\mint_{\A_n^{-1}}\phi x\,|d\phi|\,-\,\bar{x}\,\Big\|\;\xrightarrow{n\rightarrow\infty}\;0\qquad\text{ for each }x\in\xspace\;,
  \end{equation*}
  where $\bar{x}$ is the projection of $x$ onto $\xspace_\group$.
\end{fact}
By \cref{result:contraction}, we can interpret this result in terms of orbitopes contracting around $\xspace_\group$:
\begin{proof}
  The orbitope of $x$ is weakly compact and intersects $\xspace_\group$ in $\bar{x}$. \cref{result:contraction} shows
  that ${\|\empavg_n(x-\bar{x})\|\leq\|\empavg_n(\orb(x))\|\to 0}$.
\end{proof}

\subsection{Orbitopes in $\L_p$ spaces}
\label{sec:L1}

For a $\sigma$-finite measure $\mu$ on $\Omega$ and ${p\in[1,\infty]}$, we
denote by $\L_p(\mu)$ the space of (equivalence classes of) functions $f$ on $\Omega$ with
finite norm ${\|f\|_p:=(\int|f|^pd\mu)^{1/p}}$.
Recall that $\L_p$ is a Banach space with dual ${\L_p(\mu)'=\L_q(\mu)}$ for ${1\leq p\leq q\leq\infty}$ and ${1/p+1/q=1}$.
If $\mu$ is $\group$-invariant, each map ${\phi:\Omega\rightarrow\Omega}$ takes null sets to null sets, so we can define
\begin{equation}
  \label{Lp:action}
  (\mu\text{-equivalence class of }x)\circ\phi\;:=\;\mu\text{-equivalence class of }(x\circ\phi)
  \quad\text{ for }\phi\in\group\;.
\end{equation}
Since $\phi$-invariance of $\mu$ implies ${\mu(x\circ\phi)=\mu(x)}$ for every integrable function $x$, we have
\begin{equation*}
  \|x\circ\phi\|_p\;=\;\|x\|_p
  \quad\text{ and }\quad
  \sp{x\circ\phi,y\circ\phi}\;=\;\sp{x,y}\;.
\end{equation*}
In particular, the orbit of every ${x\in\L_p(\mu)}$ is again in $\L_p(\mu)$. 
It is straightforward to verify that \eqref{Lp:action} is a linear action on $\L_p(\mu)$, which is continuous since the norm is invariant.
Let ${\orb_p(x)}$ be the $\L_p$-norm orbitope of $x$. Since the $q$-norm dominates the $p$-norm for ${p\leq q}$, we have ${\group(x)\subset\orb_q(x)\subset\orb_p(x)}$,
and taking closures shows
\begin{equation*}
  \orb_q(x)\;=\;\text{norm-dense subset of }\orb_p(x)\qquad\text{ whenever }1\leq p\leq q\leq\infty\;.
\end{equation*}
Invariance of the norm also implies the orbits are bounded, so \cref{lemma:alaoglu} applies.
In summary, we have shown:
\begin{lemma}
  If $\group$ acts measurable on $\Omega$ and leaves $\mu$ invariant, \eqref{Lp:action} defines linear continuous actions
  on $\L_p(\mu)$ and $\L_q(\mu)$. If ${1/p+1/q=1}$, the actions are dual, and every weak* orbitope in $\L_q(\mu)$ is compact.
\end{lemma}
Let ${\Sigma_\group}$ denote the $\sigma$-algebra
of $\group$-invariant Borel sets. A measurable function ${f:\Omega\to\mathbb{R}}$ is $\group$-invariant if and only if
it is $\Sigma_\group$-measurable.
The $\L_p$ of equivalence classes of $\Sigma_\group$-measurable functions is hence
\begin{equation*}
  \L_p(\Sigma_\group,\mu)\;=\;\braces{\text{equivalence class of }f\,|\,f\text{ $\group$-invariant and }\|f\|_p<\infty}\;,
\end{equation*}
which is the set of elements of $\L_p(\mu)$ invariant under the action \eqref{Lp:action}. In other words,
if we choose ${\xspace}$ as ${\L_p(\mu)}$, then $\L_p(\Sigma_\group,\mu)$ is the closed linear subspace $\xspace_\group$.

\subsection{Illustration: The Hunt-Stein theorem}
\label{sec:hunt:stein}

Consider a testing problem, in which a set $M$ of probability measures on $\Omega$ is partitioned into two sets
$H$ (the hypothesis) and $A$ (the alternative).
Recall that a set $M$ of probability measures is
\kword{dominated} if there is a $\sigma$-finite measure $\mu$ such that ${M\ll\mu}$.
If so, we denote the corresponding set of densities
\begin{equation*}
  dM/d\mu:=\braces{dP/d\mu\,|\,P\in M}\subset\L_1(\mu)\;.
\end{equation*}
Given an observation ${x\in\Omega}$, one must decide
whether $x$ was generated by an element of $H$ (``accept the
hypothesis'') or of $A$ (``reject''). The decision procedure
may be randomized, and is represented by a \kword{critical function}
$w$. The value $w(x)$ is read as the probability (under the randomization of the test)
that the test rejects when $x$ is observed.
A measurable function $w$ is a valid critical function
if it takes values in $[0,1]$ $P$-almost surely, for all ${P\in M}$.
If $M$ is dominated by a $\sigma$-finite measure $\mu$, 
the set of such functions is
\begin{equation*}
  W=\braces{w\in\L_\infty(\mu)\,|\,w\geq 0\;\mu\text{-a.e.\ and }\|w\|_{\infty}\leq 1}\;.
\end{equation*}
The objective is to find a critical function $\hat{w}$
that maximizes the probability of rejection on $A$, while
upper-bounding it on $H$ by some fixed ${\alpha\in[0,1]}$,
\begin{align}
  \label{maximin:test}
  \hat{w}\,\in\,\argmax_{w\in W}\,\inf\braces{P(w)\,|\,P\in A}
  \qquad
  \text{subject to }Q(w)\,\leq\,\alpha\quad\text{for all }Q\in H\;.
\end{align}
The Hunt-Stein theorem {\citep[e.g.][]{Bondar:Milnes:1981:1,LeCam:1986,Eaton:George:2021}} states the following:
\begin{fact}
  \label{fact:HS}
  Let $\group$ be amenable, and let both $H$ and $A$ be a $\group$-invariant and dominated by a
  $\group$-invariant, $\sigma$-finite measure $\mu$.
  If there is a critical function $\hat{w}$ that satisfies
  \eqref{maximin:test},
  there is an $\mu$-almost $\group$-invariant critical function $\overline{w}$
  that also satisfies \eqref{maximin:test}, and
  ${\inf_{\phi\in\group}\,\phi P(\hat{w})\;\leq\;P(\overline{w})\;\leq\;\sup_{\phi\in\group}\,\phi P(\hat{w})}$
  holds for all ${P\in M}$.
\end{fact}
In terms of our results, this is the case because the hypothesis implies that \eqref{maximin:test}
minimizes a $\group$-invariant convex lsc function $f$ (or rather, maximizes $-f$), and the orbitope
of $\hat{w}$ is compact since it is a weak* orbitope in $\L_\infty$.
In more detail, we can obtain the result from \cref{HS:general}:
Since $dM/d\mu$ is a subset of $\L_1(\mu)$, and $W$ is in the dual ${\L_1(\mu)'=\L_\infty(\mu)}$, we have
${P(w)=\sp{dP/d\mu,w}}$. Maximizing \eqref{maximin:test}
is then equivalent to minimizing a support function, namely
\begin{equation*}
  \sup_{P\in A}{\sp{-dP/d\mu,w}}\;=\;H(w|-dA/d\mu)\;,
  \quad\text{ over }\quad F:=\medcap\nolimits_{Q\in H}\braces{w\in W\,|\,Q(w)\leq\alpha}\;.
\end{equation*}
It is easy to check the assumptions of \cref{fact:HS} make both $f$ and $F$ convex and $\group$-invariant, and that
$F$ is closed. The existence of $\bar{w}$ hence follows from \cref{HS:general}, and applying \eqref{eq:H:equals:sup} to the bound in \cref{HS:general}
shows
\begin{align*}
  P(\bar{w})
  \;\leq\;
  \sp{dP/d\mu,\bar{w}}
  \;\leq\;
  H(dP/d\mu|\orb(\hat{w}))
  \;=\;
  \sup\nolimits_{\phi}\phi_*P(\hat{w})\;,
\end{align*}
which is the upper bound in \cref{fact:HS}. The lower bound follows analogously.

\begin{remark*}
  Observe that, by \cref{theorem:day}, the invariant solution $\overline{w}$ is the limit of a subsequence of \Folner averages
  ${\empavg_n(\hat{w})}$. In a limiting sense, the Hunt-Stein theorem is therefore an instance of the summation trick \eqref{summation:trick},
  although the relevant subsequence is in general not determined. Under suitable additional conditions---if \cref{result:contraction} applies---the
  solution becomes an actual limit.
\end{remark*}

\section{Orbitopes II: Probability measures}
\label{sec:probatopes}

The next class of orbitopes we consider are those of probability measures. These are closed convex subsets of the
set ${\P=\P(\Omega)}$ of probability measures on a Polish space $\Omega$.
To make sense of linearity, we consider the smallest vector space containing $\P$, which is the space
${\ca:=\text{span}\,\P}$ of finite signed measures on $\Omega$.
We let $\group$ act measurably on $\Omega$, and consider the induced action \eqref{transforming:functions} on
(signed) measures, i.e.~the linear action ${\mu\mapsto\phi_*\mu}$ for ${\phi\in\group}$.
Since $\phi_*\P\subset\P$, its restriction to $\P$ is again an action.
There are two natural topologies on $\P$,
one defined by the total variation norm ${\|\mu\|_{\TV}:=|\mu(\Omega)|}$ on $\ca$, and one by convergence in distribution on $\P$.
We denote the corresponding orbitopes 
\begin{align*}
  \orb_{\TV}(P)\;&:=\;\text{closure of $\ch\group(P)$ in total variation}
  \\
  \text{ and }
  \qquad
  \orb_{\cid}(P)\;\;&:=\;\text{closure of $\ch\group(P)$ in convergence in distribution}\;.
\end{align*}
We always have ${\orb_{\TV}(P)\subset\orb_{\cid}(P)}$, but as we show below, these orbitopes can differ significantly.

\subsection{Compactness}

If $\Omega$ is compact, then $\P$ is compact in convergence in distribution, so $\orb_{\cid}(P)$ is automatically
compact. If $\Omega$ is not compact, there is still a sharp criterion for compactness of orbitopes:
Say that $P$ is \kword{$\group$-tight} if
\begin{equation*}
  P(\phi^{-1}K_\varepsilon)\;>\;1-\varepsilon
  \qquad\text{ for each }\varepsilon>0\,,\text{ some compact }K_\varepsilon\subset\xspace\,,\text{ and
    all }\phi\in\group\;.
\end{equation*}
For illustration, let $P_z$ be a Gaussian with mean $z$ on $\mathbb{R}^2$.
Suppose $\group$ consists of rotations around the origin.
If $z$ is the origin, $P_z$ is $\group$-invariant. Otherwise,
$P_z$ is not $\group$-invariant, but it is $\group$-tight (choose
$K_\varepsilon$ as a sufficiently large closed ball around the origin).
If $\group$ is instead the group of shifts of $\mathbb{R}$, 
no Gaussian (and in fact no probability measure on $\mathbb{R}^2$) is
$\group$-tight.
\begin{proposition}
  \label{lemma:G:tight}
  If $\group$ acts continuously on $\Omega$, the orbitope $\Pi_{\cid}(P)$ is compact if and only if $P$ is $\group$-tight.
\end{proposition}
\begin{proof}
  See \cref{proofs:sec:probatopes}.
\end{proof}
Now consider the total variation topology. If we can find
a $\group$-invariant measure $\mu$ that dominates $P$, we can reduce the problem
of compactness of $\orb_{\TV}(P)$ to compactness in $\L_1(\mu)$.
\begin{proposition}
  \label{result:TV:orbitopes}
  Let $\group$ act measurably on $\Omega$, and fix any ${P\in\P}$.
  \\[.5em]
  (i) If ${P\ll\mu}$ holds for a $\group$-invariant $\sigma$-finite measure $\mu$, then
  \begin{equation*}
    \orb_{\TV}(P)\ll\mu
    \qquad\text{ and }\qquad
    d\Pi_{\TV}(P)/d\mu\;=\;\orb_1(dP/d\mu)\;,
  \end{equation*}
  where $\Pi_1$ denotes the norm orbitope in $\L_1(\mu)$.
  The orbitope ${\orb_{\TV}(P)}$ is compact if and only if $\orb_1(dP/d\mu)$ is norm compact in $\L_1(\mu)$.
  \\[.5em]
  (ii)
  If $\group$ is amenable and $\orb_{\TV}(P)$ is compact, every random element $\xi$ with law $P$
  can be coupled to a random element $\Psi$ of $\group$ such that $\,\smash{\phi(\Psi\xi)\equdist\Psi\xi}\;$ for all ${\phi\in\group}$.
\end{proposition}
\begin{proof}
  See \cref{proofs:sec:probatopes}.
\end{proof}

\subsection{Point masses and empirical measures}

To understand the geometry of orbitopes in $\P$ we consider the simple case of points masses.
\begin{proposition}[Orbitopes of point masses]
  \label{lemma:cid:orbitope:pointmass}
  Let ${\omega}$ be an element of ${\Omega}$.
  The total variation orbitope of $\delta_{\omega}$ is the set 
  \begin{equation*}
    \Pi_{\TV}(\delta_{\omega})
    \;=\;
    \braces{P=\tsum_{i\in\mathbb{N}}p_i\delta_{\phi_i {\omega}}\,|\,p_i\geq 0,\tsum p_i=1, \phi_i\in\group({\omega})}\,
  \end{equation*}
  of all purely atomic probability measures on the orbit $\group({\omega})$.
  The weak* orbitope and its extreme points are
  \begin{equation*}
    \Pi_{\cid}(\delta_{\omega})\;=\;\P(\overline{\group({\omega})})
    \quad\text{ and }\quad
    \ex\Pi_{\cid}(\delta_{\omega})\;=\;\braces{\delta_z\,|\,z\in\overline{\group({\omega})}}\;,
  \end{equation*}
  and ${\Pi_{\cid}(\delta_{\omega})}$ is a closed face of the convex set ${\P(\Omega)}$.
\end{proposition}
\begin{proof}
  See \cref{proofs:sec:probatopes}.
\end{proof}
\cref{lemma:orbitope:invariant}(iv) extends the result to probability measures with finite support, that is, to mixtures of point masses. For example:
\begin{corollary}
  Let ${\hat{P}_n:=n^{-1}\sum_{i\leq n}\delta_{\omega_i}}$ be the empirical measure of points ${\omega_1,\ldots,\omega_n}$
  in $\Omega$.
  Then ${\orb_{\cid}(\hat{P}_n)=\braces{n^{-1}\sum_{i\leq n}Q_i\,|\,Q_i\in\P(\overline{\group(\omega_i)})}}$.
\end{corollary}

  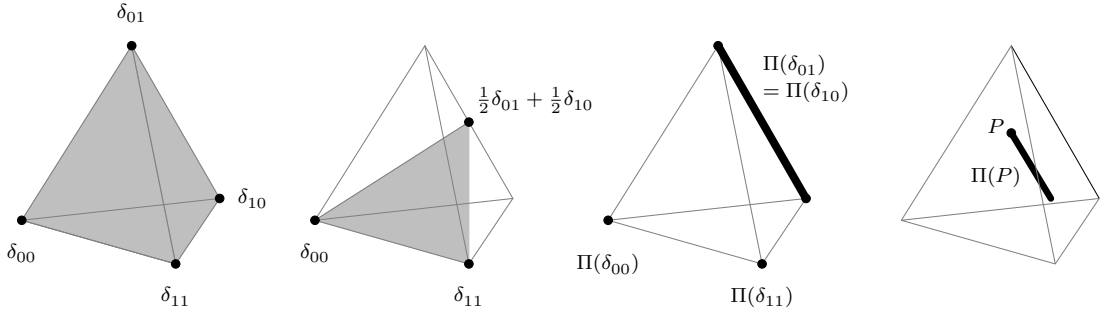
\begin{figure}
    \begin{center}
      \resizebox{\textwidth}{!}{
  \begin{tikzpicture}
    \begin{scope}[scale=3]
    \node (a) at (.5,.8) {};
    \node (b) at (0,0) {};
    \node (c) at (.9,0.1) {};
    \node (d) at (.7,-.2) {};
    \draw[gray] (b.center)--(c.center)--(d.center)--(b.center);
    \draw[gray,fill,opacity=.5] (b.center)--(a.center)--(d.center)--cycle;
    \draw[gray,fill,opacity=.5] (c.center)--(a.center)--(d.center)--cycle;
    \draw[gray] (b.center)--(a.center)--(d.center)--cycle;
    \draw[gray] (a.center)--(c.center);
    \node[circle,fill,draw,scale=.3] at (a) {};
    \node[circle,fill,draw,scale=.3] at (b) {};
    \node[circle,fill,draw,scale=.3] at (c) {};
    \node[circle,fill,draw,scale=.3] at (d) {};
    \node at ($(a)+(0,.15)$) {\scriptsize $\delta_{01}$};
    \node at ($(b)+(0,-.15)$) {\scriptsize $\delta_{00}$};
    \node at ($(c)+(.15,0)$) {\scriptsize $\delta_{10}$};
    \node at ($(d)+(0,-.15)$) {\scriptsize $\delta_{11}$};
    \end{scope}
    \begin{scope}[xshift=4cm,scale=3]
    \node (a) at (.5,.8) {};
    \node (b) at (0,0) {};
    \node (c) at (.9,0.1) {};
    \node (d) at (.7,-.2) {};
    \draw[gray] (b.center)--(c.center)--(d.center)--(b.center);
    \draw[gray] (b.center)--(a.center)--(d.center);
    \draw[gray] (a.center)--(c.center) node[pos=.5] (e) {};
    \draw[gray] (b.center)--(e.center);
    \draw[fill,gray,opacity=.5] (b.center)--(e.center)--(d.center)--(b.center);
    \node[circle,fill,draw,scale=.3] at (b) {};
    \node[circle,fill,draw,scale=.3] at (d) {};
    \node[circle,fill,draw,scale=.3] at (e) {};
    \node at ($(b)+(0,-.15)$) {\scriptsize $\delta_{00}$};
    \node at ($(d)+(0,-.15)$) {\scriptsize $\delta_{11}$};
    \node at ($(e)+(.3,0.1)$) {\scriptsize $\tfrac{1}{2}\delta_{01}+\tfrac{1}{2}\delta_{10}$};
    \end{scope}
    \begin{scope}[xshift=8cm,scale=3]
    \node (a) at (.5,.8) {};
    \node (b) at (0,0) {};
    \node (c) at (.9,0.1) {};
    \node (d) at (.7,-.2) {};
    \node (f) at ($(a)+(0,-.4)$) {};
    \draw (a.center)--(c.center) node[pos=.5] (e) {};
    \draw[gray] (b.center)--(c.center)--(d.center)--(b.center);
    \draw[gray] (b.center)--(a.center)--(d.center);
    \draw[line width=1.1mm, black] (a.center)--(c.center);

    \node[circle,fill,draw,scale=.3] at (a) {};
    \node[circle,fill,draw,scale=.3] at (b) {};
    \node[circle,fill,draw,scale=.3] at (c) {};
    \node[circle,fill,draw,scale=.3] at (d) {};
    \node at ($(b)+(0,-.2)$) {\scriptsize $\Pi(\delta_{00})$};
    \node at ($(d)+(0,-.15)$) {\scriptsize $\Pi(\delta_{11})$};
    \node at ($(e)+(.2,0.2)$) {\scriptsize
      \begin{tabular}{l}
        $\Pi(\delta_{01})$
        \\[.1em]
        $=\Pi(\delta_{10})$
    \end{tabular}};
    \end{scope}
    \begin{scope}[xshift=12cm,scale=3]
    \node (a) at (.5,.8) {};
    \node (b) at (0,0) {};
    \node (c) at (.9,0.1) {};
    \node (d) at (.7,-.2) {};
    \node (f) at ($(a)+(0,-.4)$) {};
    \draw (a.center)--(c.center) node[pos=.5] (e) {};
    \draw[line width=.9mm, black,line cap=round] (f.center)--($(c)+(-.22,-.0)$);
    \draw[gray] (b.center)--(c.center)--(d.center)--(b.center);
    \draw[gray] (b.center)--(a.center)--(d.center);
    \node[circle,fill,draw,scale=.3] at (f) {};
    \node at ($(f)+(-.07,.03)$) {\scriptsize $P$};
    \node at ($(f)+(-.07,-.2)$) {\scriptsize $\Pi(P)$};
    \end{scope}
  \end{tikzpicture}
  }
    \end{center}
  \caption{Orbitopes of coin flip pairs.
    \emph{Left}: The set $\P$ of distributions on $\braces{0,1}^2$ is the convex hull of its four point masses,
    and can be identified with a subset of ${\smash{\mathbb{R}^3}}$.
    \emph{Middle left}: The set $\P_\group$ of permutation-invariant distributions is a convex subset of $\P$.
    The orbits of $\group$ in $\smash{\braces{0,1}^2}$ are the sets $\braces{00}$, $\braces{11}$, and $\braces{01,10}$,
    and the extreme points of $\P_\group$ are the uniform distributions on these orbits.
    \emph{Middle right}: Orbitopes of point masses are faces of $\P$ (cf.~\cref{lemma:cid:orbitope:pointmass}), here
    two singletons and an edge.
    \emph{Middle right}: The orbitope of a measure $P$ in the interior.
  }
  \label{fig:bernoulli:orbitope}
  \end{figure}

  \begin{example}
    \label{example:point:masses}
    To illustrate how norm and weak* orbitopes may differ,
    consider the product space ${\Omega=\braces{0,1}^\mathbb{N}}$ of infinite binary sequences.
    We let the group $\mathbb{S}$ of finitely supported permutations of $\mathbb{N}$ act on $\Omega$ by permuting sequence indices.
    This is the setting of de Finetti's theorem: A random element of $\Omega$ whose law is $\mathbb{S}$-invariant
    is an exchangeable sequence. Denote by ${\Omega(m,n)\subset\Omega}$ the set of sequences containing $m$ zeros $n$ and ones.
    Thus, each sequence is either in ${\Omega(\infty,\infty)}$, or in ${\Omega(n,\infty)}$ or ${\Omega(\infty,n)}$ for some finite $n$.
    \par\vspace{.5em}
    {\hangindent=.75em\hangafter=0
    \noindent{\bf Lemma}.
    The point mass at a sequence $\omega$ has TV orbitope ${\Pi_{\TV}(\delta_{\omega})=\P(\mathbb{S}({\omega}))}$,
    whereas ${\Pi_{\cid}(\delta_{\omega})}$ is ${\P(\mathbb{S}({\omega}))}$ if ${\omega\not\in\Omega(\infty,\infty)}$, and
    ${\P(\Omega(\infty,\infty))}$ for ${\omega\in\Omega(\infty,\infty)}$.
    \par
    \hangindent=0em}
      \vspace{.5em}
      \noindent (See \cref{proofs:sec:probatopes} for the proof.)
      Since $\mathbb{S}$ is countable, each orbit $\mathbb{S}(\omega)$ is countable,
      but $\Omega(\infty,\infty)$
    is not. Thus, weak* orbitopes for sequences in $\Omega(\infty,\infty)$ are much larger than the
    total variation orbitopes of the same sequences, and also than the weak* orbitopes of sequences
    with finitely many 0s or 1s.
    \end{example}
    \begin{example}
 The set of probability measures on the smaller set ${\Omega=\braces{0,1}^2}$ is illustrated in 
\cref{fig:bernoulli:orbitope}. That all orbitopes are orthogonal to the subset $\P_\group$ of invariant measures
is due to the fact that ${\Omega}$ is  finite and we have identified $\P(\Omega)$ with a subset of the Hilbert space $\smash{\mathbb{R}^3}$.
There is no similar notion of orthogonality if $\Omega$ is infinite.
Orbitopes of point masses are closed faces of $\P$, as explained in \cref{lemma:cid:orbitope:pointmass}.   
  \end{example}

\subsection{Duality}

We briefly discuss a somewhat technical matter, namely
how orbitopes in $\P$ relate to our previous results on duality.
The answer is straightforward if the Polish space $\Omega$ is also compact---in this case,
$\ca$ is the dual space of the space ${(\C_b,\|\argdot\|_{\sup})}$ of bounded continuous functions
on $\Omega$. For a general Polish space, the answer is somewhat more involved. 
The dual of $\C_b$ is a larger space than $\ca$, namely
\begin{equation*}
  \C_b'\;=\;\ba_n\;:=\;
  \braces{\mu:\mathcal{A}\to\mathbb{R}\,|\,\mu\text{ finitely additive and inner regular}}
\end{equation*}
where inner regular means ${\mu(A)=\sup\braces{\mu(F)|F\subset A\text{ closed}}}$ for all ${A\in\mathcal{A}}$.
We denote the dual norm by ${\|\mu\|':=\sup_f|\sp{f,\mu}|/\|f\|_{\sup}}$. For details on the dual pair
${\sp{\C_b,\ba_n}}$, see \citep{Aliprantis:Border:2006}. Those relevant for our purposes are:
\begin{itemize}
\item The set $\ca$ is the linear subspace ${\ca=\braces{\mu\in\ba_n|\mu\text{ countably additive}}}$ of $\ba_n$.
\item The sets ${\ca}$ and ${\P}$ are not weak* closed in $\ba_n$ (unless $\Omega$ is compact).
  The weak* closure of a set ${M\subset\P}$ may therefore contain elements of $\ba_n$ that are not
  countably additive.
\item The restriction of the weak* topology to $\P$ is the topology of convergence in distribution. In other
  words, a set ${M\subset\P}$ is closed iff there is a closed set ${M'\in\ba_n}$ such that ${M=M'\cap\P}$.
\item In contrast, $\ca$ and $\P$ are closed in the dual norm, and the restriction of $\|\argdot\|'$ to $\ca$
  is precisely the total variation norm.
\item Every continuous linear functional on $\C_b$ is of the form ${f\mapsto\sp{f,\mu}}$ for some ${\mu\in\ba_n}$. For ${\mu\in\ca}$,
  we have ${\sp{f,\mu}=\int fd\mu}$, by a version of Riesz' theorem.
\end{itemize}
Now assume again that $\group$ acts continuously on $\Omega$, and consider the induced actions ${f\mapsto f\circ\phi^{-1}}$ and
${\mu\mapsto\mu\circ\phi^{-1}}$ on functions and finitely additive measures.
By \eqref{eq:image:measure:int},
\begin{equation*}
  \sp{\phi(f),\phi(\mu)}\;=\;\mint (f\circ\phi^{-1})d(\phi_*\mu)\;=\;\mint fd\mu\;=\;\sp{f,\mu}
  \qquad\text{ for each }\mu\in\ca\;,
\end{equation*}
so the restriction of the action to the subspace $\ca$ behaves like a dual action, and we show as part of the next result
that duality extends to the entire space $\ba_n$. 
For a probability measure $P$, we can now define orbitopes by taking the closure $\orb'(P)$ of $\ch\group(P)$ in the dual norm topology of $\ba_n$,
or weak* closure $\orb_*(P)$. All our results on orbitopes in dual spaces now apply to $\orb_*(P)$ (though not to $\orb'(P)$, since the weak* topology
is consistent with the dual pair, whereas the norm topology is not).\nolinebreak

\begin{proposition}
  \label{lemma:cid:duality}
  Let $\group$ act continuously on a Polish space $\Omega$, and let $P$ be a probability measure. Then $\orb_*(P)$ is weak* compact, and 
  \begin{equation*}
    \orb_{\TV}(P)\;=\;\orb'(P)
    \qquad\text{ and }\qquad
    \orb_{\cid}(P)\;=\;\orb_*(P)\cap\P\qquad\text{ for each }P\in\P\;.
  \end{equation*}
  If $\orb_*(P)$ and $\orb_{\cid}(P)$ differ, this difference consists entirely of set functions that are not countably additive.
  Moreover, $\orb_*(P)$ and $\orb_{\cid}(P)$ have the same support function, and
  \begin{equation*}
    H(f|\orb_{\cid}(P))\;=\;H(f|\orb_*(P))\;=\;H(P|\orb(f))
    \qquad\text{ for each }P\in\P\text{ and }f\in\C_b\;.
  \end{equation*}
\end{proposition}
The compactness of weak* orbitopes is an instance of \cref{lemma:alaoglu}. 
Since compactness of $\orb_*(P)$ implies that ${\orb_*(P)\cap\P=\orb_{\cid}(P)}$ is closed, but not that it is compact,
this does not contradict \cref{lemma:G:tight}.

\section{Application I: Kernel mean embeddings}
\label{sec:mmd}

Mean embeddings of probability measures are used in machine learning to represent distributions
\citep[e.g.][]{Sejdinovic:Sriperumbudur:Gretton:Fukumizu:2013,Simon-Gabriel:Barp:Schoelkopf:Mackey:2023}.
Like a probability density, a mean embedding represents a probability measure $P$ as
a function. Whereas a density is defined relative to a reference measure $\mu$ and lives in $\L_1(\mu)$,
a mean embedding is defined relative to a reproducing kernel, and lives in an RKHS.

\subsection{Invariant optimal embeddings}

A Hilbert space $\mathbf{H}$ of functions ${x:\Omega\rightarrow\mathbb{R}}$ 
is a reproducing kernel Hilbert space, or \kword{RKHS}, if there is a
positive definite function ${\kappa:\Omega\times\Omega\rightarrow\mathbb{R}}$ such that
\begin{equation*}
  x(\omega)\;=\;\sp{x,\kappa(\omega,\argdot)}\qquad\text{ for all }x\in\mathbf{H}\text{ and }\omega\in\Omega\;.
\end{equation*}
If so, $\kappa$ is a \kword{reproducing kernel} for $\mathbf{H}$. For each ${\omega\in\Omega}$, the function
${\kappa(\omega,\argdot)}$ is an element of $\mathbf{H}$, and the map
${\Delta:\Omega\rightarrow\mathbf{H}}$ defined by ${\omega\mapsto\kappa(\omega,\argdot)}$ is
called the \kword{feature map} \citep{Steinwart:Christmann:2008}.
Let $P$ be a probability measure on $\Omega$. A function ${m(P)\in\mathbf{H}}$ is the
\kword{mean embedding} of $P$ if it satisfies either of the equivalent properties
\begin{equation*}
  \mint f dP\;=\;\sp{f,m(P)}\;\text{ for all }f\in\mathbf{H}
  \quad\text{ or }\quad
  m(P)\;=\;\mint\Delta(\omega)P(d\omega)\;,
\end{equation*}
where the integral on the right is again a Bochner integral \citep{Sejdinovic:Sriperumbudur:Gretton:Fukumizu:2013}.
We define
\begin{equation*}
\mathcal{M}\;:=\;
\braces{x\in\mathbf{H}\,|\,x\text{ is mean embedding of a probability measure}}\;.
\end{equation*}
A kernel $\kappa$ is \kword{bounded} if ${\sup_{\omega}\kappa(\omega,\omega)<\infty}$, and
\kword{characteristic} if the map ${m}$ is injective on its domain in ${\P=\P(\Omega)}$.
\\[.2em]
{\textbf{Convention}.}
Call a reproducing kernel \kword{nice} if it is continuous, bounded, characteristic, 
and the function ${\kappa(\argdot,\omega)}$ vanishes at infinity for each ${\omega\in\Omega}$.
\\[.2em]
Now suppose a group $\group$ acts on $\Omega$.
We consider the induced action ${x\mapsto x\circ\phi}$ on functions ${x:\Omega\to\mathbb{R}}$.
Our next result makes the following assumptions.
\begin{equation}
  \label{mmd:conditions}
  \begin{split}
    \text{$\Omega$ is locally }&\text{compact Polish, $\group$ is nice and acts continuously on $\Omega$,}
    \\[-.2em]
    &\text{and $\kappa$ is nice and diagonally $\group$-invariant.}
  \end{split}
\end{equation}
Nice kernels on locally compact spaces are a standard assumption in the mean embedding literature
that ensures $m(P)$ exists and is well-behaved
\citep{Sejdinovic:Sriperumbudur:Gretton:Fukumizu:2013,Simon-Gabriel:Barp:Schoelkopf:Mackey:2023}.
As we will see below, diagonal invariance ensures the natural action ${x\mapsto x\circ\phi}$ on
functions also defines an action on the RKHS, and that this action is unitary.
For mean embeddings, our general results can be assembled into the following result; the proof
is given at the end of this section.
\begin{theorem}
  \label{::mmd:day}
  If \eqref{mmd:conditions} holds, the maps ${x\mapsto x\circ\phi}$, for each ${\phi\in\group}$,
  define a unitary action of $\group$ on $\mathbf{H}$. If $\group$ is amenable, the following holds:\\[.5em]
  (i) If a  convex, weakly lsc and $\group$-invariant function ${f:\mathcal{M}\to\mathbb{R}\cup\braces{\infty}}$
  has a minimizer $\hat{x}$, it is also minimized by the
  embedding $m(P)$ of a $\group$-invariant probability measure $\bar{P}$ on $\Omega$.
  \\[.5em]
  (ii) Let ${E,F\subset\mathcal{M}}$ be $\group$-invariant sets, where $F$ is closed and convex. If
  \begin{equation*}
    \text{minimize }\quad H(x|E)\qquad\text{ subject to }y\in F
  \end{equation*}
  has a solution $\hat{x}$, it also has a solution $m(\bar{P})$ where $\bar{P}$ is $\group$-invariant.
  \\[.5em]
  In either case, ${\empavg_n(\hat{x})\xrightarrow{n\to\infty}m(\bar{P})}$ holds pointwise and in the norm of
  $\mathbf{H}$, and if $\hat{P}$ is a probability measure with ${m(\hat{P})=\hat{x}}$, then
  ${\empavg_n(\hat{P})\to\bar{P}}$ in distribution.
\end{theorem}

\subsection{Tools: Results on invariant embeddings}
\label{sec:mmd:tools}

We now establish a few properties of invariance in reproducing spaces, which we then use to prove \cref{::mmd:day}.
The first two results do not require assumption \eqref{mmd:conditions}.
If a group $\group$ acts on $\Omega$, it always induces an action ${x\mapsto x\circ\phi}$ on functions ${x:\Omega\to\mathbb{R}}$,
but the restriction of this action to $\mathbf{H}$ is an action. A sufficient condition for this to be true is a diagonally
invariant kernel, which even makes the action unitary:
\begin{lemma}[Invariant kernels]
  \label{lemma:invariant:kernel}
  Let a $\group$ act on a set $\Omega$, and let $\kappa$ be a kernel on $\Omega$.
  If $\kappa$ is diagonally invariant, the maps ${x\mapsto x\circ\phi}$ define
  an action of $\group$ on $\mathbf{H}$. This action is unitary and makes the feature map equivariant,
  \begin{align*}
    \sp{x\circ\phi,y\circ\phi}
    \;&=\;
    \sp{x,y}
    \quad\text{ and }\qquad
    \Delta\circ\phi\;=\;\phi^{-1}\circ\Delta
    &&\text{ for all }\phi\in\group\;.
  \end{align*}
  If $\kappa$ is even separately invariant, so is the inner product,
  i.e.\ ${\sp{x\circ\phi,y\circ\psi}=\sp{x,y}}$ for all ${\phi,\psi\in\group}$,
  and $\Delta$ is $\group$-invariant.
\end{lemma}
\begin{proof}
  See \cref{proofs:sec:mmd}.
\end{proof}
We already know, by \cref{fact:mean:ergodic:theorem} that \Folner averages on Hilbert
space are well-behaved. If the space is reproducing, convergence holds even pointwise, and we can obtain
a reproducing kernel for the subspace $\mathbf{H}_\group$ of invariant elements as a limit:
\begin{lemma}[\Folner averages in reproducing spaces]
  \label{RKHS:ergodic:theorem}
  Let an amenable group $\group$ act measurably on a measurable space $\Omega$, and let $\kappa$ be a measurable and diagonally invariant kernel
 on $\Omega$. For each ${x\in\mathbf{H}}$, 
    \begin{align*}
      \empavg_n(x)\;=\;\mfrac{1}{|\A_n|}\mint_{\A_n}x\circ\phi|d\phi|
      \quad\xrightarrow{n\rightarrow\infty}\quad
      \bar{x}\qquad\text{ pointwise and in the norm of }\mathbf{H}\;,
    \end{align*}
    where $\bar{x}$ is the projection of $x$ onto the closed linear subspace $\mathbf{H}_\group$. The limit
    \begin{equation*}
      \bar{\kappa}(\omega,\upsilon)\;:=\;
      \lim_n\,\mfrac{1}{|\A_n|}\mint_{\A_n}\!\!\kappa(\omega,\psi\upsilon)|d\psi|
      \;=\;
      \empavg_n(\kappa(\omega,\argdot))(\upsilon)
      \qquad\text{ for }\omega,\upsilon\in\Omega
    \end{equation*}
    is a reproducing kernel for $\mathbf{H}_\group$. It is separately $\group$-invariant and measurable, and is continuous
    if $\kappa$ is continuous.
\end{lemma}
\begin{proof}
  See \cref{proofs:sec:mmd}.
\end{proof}
Nice kernels are used in the mean embedding literature because they guarantee good properties of embeddings. Here
is a summary of such guarantees:
\begin{fact}[{\citep{Sejdinovic:Sriperumbudur:Gretton:Fukumizu:2013,Simon-Gabriel:Barp:Schoelkopf:Mackey:2023}}]
  \label{fact:mmd}
  Let $\Omega$ be Polish and locally compact, and let $\kappa$ be a nice kernel.
  Then every probability measure on $\Omega$ has a unique mean embedding, the map ${P\mapsto m(P)}$ is one-to-one, and
  the function ${\text{\rm MMD}(P,Q):=\|m(P)-m(Q)\|}$, called the maximum mean discrepancy,
  is a metric on $\P$ that metrizes convergence in distribution.
\end{fact}

We note two consequences that will prove useful in the next section:
\begin{lemma}[Properties of the embedding map]
  \label{lemma:mmd:isometry}
  If $\Omega$ is locally compact Polish and $\kappa$ is nice,
  the map ${m:\P\rightarrow\mathcal{M}}$ is a linear isometry if $\P$ is metrized by MMD. For each ${P\in\P}$, the
  image measure under the feature map has barycenter ${\sp{\Delta_*P}=m(P)}$. The embedding map is equviariant,
  ${m(P)\circ\tau^{-1}=m(\tau_*P)}$, under every 
  measurable bijection ${\tau:\Omega\to\Omega}$ that leaves $\kappa$ diagonally invariant.
\end{lemma}
\begin{proof}
  See \cref{proofs:sec:mmd}.
\end{proof}

\subsection{The set of invariant mean embeddings}

We now consider mean embeddings of $\group$-invariant measures, and in particular the geometry of the set
\begin{equation*}
\mathcal{M}_\group\;=\;m(\P_\group)\qquad\text{ where }\P_\group\;=\;\braces{P\in\P(\Omega)\,|\,P\text{ is $\group$-invariant}}\;.
\end{equation*}
This set may be empty, since some actions do not have invariant probability measures.
If \eqref{mmd:conditions} holds, $m$ is equivariant by \cref{lemma:mmd:isometry}, and
we have ${\mathcal{M}_\group=\mathcal{M}\cap\mathbf{H}_\group}$. In other words,
$\mathcal{M}_\group$ can equivalently be defined
as the set of mean embeddings that are $\group$-invariant as functions on $\Omega$.

Before we describe the geometry of $\mathcal{M}_\group$, we consider that of the larger set
$\mathcal{M}$. Since $\mathcal{M}$ is an isometric image of $\P$, it inherits geometric properties of 
$\P$ (which can be found in \citep{Aliprantis:Border:2006}):
\begin{corollary}
  \label{result:mean:embeddings:geometry}
  If $\Omega$ is locally compact Polish and $\kappa$ nice, $\mathcal{M}$ is a closed convex subset of $\mathbf{H}$,
  and its extreme points are precisely the mean embeddings of point masses.
  If ${F\subset\Omega}$ is closed, the set
  ${\mathcal{M}(F):=\braces{m(P)\,|\,P(F)=1}}$ is a closed face of the convex set $\mathcal{M}$.
  The map ${\omega\mapsto m(\delta_{\omega})}$ is an isomorphism of $\Omega$ and $\ex\,\mathcal{M}$.
\end{corollary}
The smaller ${\mathcal{M}_\group\subset\mathcal{M}}$ is again convex, and can also be characterized by its extreme points:
\begin{theorem}
  \label{result:mmd:extremal}
  If \eqref{mmd:conditions} holds, ${\mathcal{M}_\group}$ is a closed convex subset of ${\mathbf{H}}$,
  and its set of extreme points is measurable.
  For each $\group$-invariant probability measure $P$ on $\Omega$, the following are equivalent:\\[.5em]
  (i) $m(P)$ is an extreme point of $\mathcal{M}_\group$.\\[.5em]
  (ii) For each $\group$-invariant ${x\in\mathbf{H}}$,
  \begin{equation*}
    \sp{m(P),x}\;=\;x(\omega)\qquad\text{ holds for }P\text{-almost all }\omega\in\Omega\;.
  \end{equation*}
  (iii) $P$ is an extreme point of $\P_\group$.
  \\[.5em]
  A function ${x:\Omega\to\mathbb{R}}$ is in $\mathcal{M}_\group$ if and only if ${x=\int_{\ex\mathcal{M}_\group}z\mu_x(dz)}$
  for some probability measure $\mu_x$ on ${\ex\,\mathcal{M}_\group}$. If so, $\mu_x$ uniquely determined by $x$.
\end{theorem}
\begin{proof}
  See \cref{proofs:sec:mmd}.
\end{proof}
The final statement of the result says that $m(P)$ is the barycenter $\sp{\eta_P}$. Since $\mathcal{M}_\group$ is closed, convex and metrizable,
this is almost a special case of Choquet's theorem, except for the fact that $\mathcal{M}_\group$ need not be compact. The extreme points of $\mathcal{P}_\group$ in
(iii) are characterized by the ergodic decomposition theorem (see e.g.\ A1.3 in \citep{Kallenberg:2005}, 8.20 in \citep{Einsiedler:Ward:2011}),
which can be summarized as follows:
\begin{fact}
  \label{::ergodic:decomposition}
  If a nice group $\group$ acts measurably on a Polish space $\Omega$, the set ${\ex\,\P_\group}$ is measurable.
  A measure ${P\in\P}$ is $\group$-invariant if and only if ${P=\int_{\ex\P_\group}Q\mu_P(dQ)}$ for
  some probability measure $\mu_P$ on ${\ex\,\P_\group}$.
  This measure is uniquely determined by $P$. A measure ${Q\in\P}$ is in ${\ex\,\P_\group}$ if and only if it is $\group$-ergodic.
\end{fact}
de Finetti's theorem is an example of \cref{::ergodic:decomposition} where ${\Omega=I^{\mathbb{N}}}$ is a space of
sequences, say on ${I=[0,1]}$, and $\group$ is the group of finitely supported permutations acting on the sequence
indices. By de Finetti's theorem, $P$ is exchangeable---that is, $\group$-invariant---if and only if
${P=\int_{\P(I)}R^{\infty}\nu_P(dR)}$ for some probability measure $\nu_P$ on $\P(I)$.
In other words, the ergodic measures are precisely the distributions ${Q=R^{\otimes\infty}}$
of i.i.d.~sequences. The proof of \cref{result:mmd:extremal} shows that $m$ maps ${\ex\,\P_\group}$ isometrically to
${\ex\,\mathcal{M}_\group}$, and $\eta_P$ is the image measure ${m_*\mu_P}$. We therefore have
\begin{equation*}
  P\text{ is $\group$-invariant }
  \quad\Leftrightarrow\quad
  m(P)\;=\;\mint_{\ex\mathcal{M}_\group}z\eta_P(dz)\;=\;\sp{m_*\mu_P}\;.
\end{equation*}
With the properties of mean embeddings established above, \cref{::mmd:day} becomes an example of our general results:
\begin{proof}[Proof of {\cref{::mmd:day}}]
  We collect a few facts we have already established:
  \begin{itemize}
  \item By \eqref{lemma:invariant:kernel}, the action is well-defined an unitary. In particular,
    it has bounded orbits, and all orbitopes are weakly compact, by \cref{result:orbitope:unitary:action}.
  \item The invariant element of each orbitope $\orb(x)$ is the projection $\bar{x}$
    of $x$ onto $\mathbf{H}_\group$, again by \cref{result:orbitope:unitary:action}.
  \item The set $\mathcal{M}$ is norm-closed and convex by \eqref{result:mean:embeddings:geometry},
    and therefore weakly closed, since the norm and weak topology have the same closed convex sets.
    It is $\group$-invariant, since $\P$ is $\group$-invariant.
  \end{itemize}
  If $\hat{x}$ is the minimizer of $f$ in (i), $\orb(\hat{x})$ contains a $\group$-invariant
  minimizer $\bar{x}$ by \cref{theorem:day}. (ii) has a $\group$-invariant minimizer ${\bar{x}\in\orb(\hat{x})}$
  by \cref{HS:general}.
  By \cref{RKHS:ergodic:theorem}, ${\empavg_n(\hat{x})\to\bar{x}}$ holds pointwise and in norm.
  Since $\mathcal{M}$ is closed convex and $\group$-invariant, and since
  ${\bar{x}}$ is the projection of $\hat{x}$, we have
  \begin{equation*}
    \bar{x}\in\orb(\hat{x})\subset\mathcal{M}
    \qquad\text{ and therefore }\qquad
    \bar{x}\in\mathcal{M}\cap\mathbf{H}_\group\;=\;\mathcal{M}_\group\;,
  \end{equation*}
  so $\bar{x}$ is the embedding of a $\group$-invariant distribution $\bar{P}$.
  Since $m$ is an isometry, it commutes with the Bochner integral, which shows
  \begin{equation*}
    m\Bigl(\mfrac{1}{|\A_n|}\mint_{\A_n}\phi_*P|d\phi|\Bigr)\;=\;\mfrac{1}{|\A_n|}\mint_{\A_n}m(P)\circ\phi|d\phi|
    \quad\text{ or in short }\quad
    m\circ\empavg_n\;=\;\empavg_n\circ m\;.
  \end{equation*}
  By \cref{fact:mmd}, convergence in norm implies ${\empavg_n(\hat{P})\to\bar{P}}$ in distribution.
\end{proof}
\begin{remark*}
  Although the vector space $\text{span}\,\P$ has not inner product structure, and therefore no notion
  of orthogonal projection, the identity ${m\circ\empavg_n=\empavg_n\circ m}$ in the proof above shows that
  \begin{equation*}
    \lim\empavg_n(P)\;=\;\bigl(m^{-1}\circ\text{(projection onto $\mathbf{H}_\group$)}\circ m\bigr)\,(P)\;.
  \end{equation*}
  We can hence read the limit on the left as a form of projection onto $\P_\group$.
\end{remark*}

\section{Application II: Invariant couplings}
\label{sec:mk}

Let $P_1$ and $P_2$ be probability measures on Polish spaces $\Omega_1$ and
$\Omega_2$. A \kword{coupling} of these measures is a joint
distributions $P$ on ${\Omega:=\XtimesY}$ with marginals $P_1$ and $P_2$.
Let ${\Lambda=\Lambda(P_1,P_2)}$ be the set of all such couplings.
A \kword{cost} is a lsc function ${c:\XtimesY\rightarrow[0,\infty]}$, 
and the \kword{risk} of a coupling $P$ is the expectation $P(c)$. Regarded as a functional
${P\mapsto P(c)}$, the risk is linear and lsc on $\Lambda$.

\subsection{Invariance}

Suppose a group $\group$ acts measurably on $\Omega_1$ and on $\Omega_2$,
and leaves $P_1$ and $P_2$ invariant. The actions define an action on the product space,
\begin{equation*}
  \phi(\omega_1,\omega_2)\;:=\;(\phi\otimes\phi)(\omega_1,\omega_2)\;=\;(\phi \omega_1,\phi \omega_2)
  \qquad\text{ for }(\omega_1,\omega_2)\in\Omega\text{ and }\phi\in\group\;.
\end{equation*}
The set of couplings of $P_1$ and $P_2$ invariant under this action is then
\begin{equation*}
  \Lambda_\group\;:=\;\Lambda\cap\P_\group
  \quad\text{ where }\quad
  \P_\group\;:=\;\braces{P\in\P(\Omega)\,|\,(\phi\otimes\phi)_*P=P\text{ for all }\phi\in\group}
\end{equation*}
Since invariance of the marginals does not imply couplings are invariant (see \cref{example:transport}(ii)),
the set $\Lambda_\group$ is in general a proper subset of $\Lambda$. It is non-emtpy, since the product measure
${P_1\otimes P_2}$ is always invariant. Elements of $\Lambda_\group$ are known as
\kword{joinings} in ergodic theory, and various aspects of $\Lambda_\group$ are well-studied \citep{Glasner:2003}.

\subsection{The set of invariant couplings}

In light of \cref{theorem:day}, we are interested in how a linear lsc functionals, such as the risk,
behave on the extreme points of $\Lambda_\group$. Denote by
$\Sigma=\Sigma_\group$ the $\sigma$-algebra of $\group$-invariant Borel sets in $\Omega$, and
by $\Sigma_i$ the $\group$-invariant $\sigma$-algebra on $\Omega_i$. Recall from \cref{sec:L1} that $\L_p(\Sigma,P)$
is the subspace of $\group$-invariant elements of $\L_p(P)$.
\begin{theorem}[Extremal invariant couplings]
  \label{result:extremal:couplings}
  Let a nice group $\group$ act measurably on $\Omega_1$ and on $\Omega_2$,
  let $P_1$ and $P_2$ be $\group$-invariant, and let ${P\in\Lambda_\group}$. Then the following are equivalent:
  \\[.2em]
  (i) $P$ is an extreme point of $\Lambda_\group$.
  \\[.2em]
  (ii) The set 
  $\braces{g_1+g_2\,|\,g_i\in\L_1(\Sigma_i,P_i)}$ is norm-dense in $\L_1(\Sigma,P)$.
  \\[.2em]
  (iii) There is no non-zero ${f\in\L_\infty(\Sigma,P)}$ such that
  \begin{equation*}
    \mint(g_1+g_2)fdP\;=\;0\qquad\text{ for all }g_i\in\L_1(\Sigma_i,P_i)\;.
  \end{equation*}
  (iv) There is no non-zero ${f\in\L_\infty(\Sigma,P)}$ such that, for ${(\xi_1,\xi_2)\sim P}$,
  \begin{equation*}
    \label{eq:parthasarathy:condition}
    \mean[f(\xi_1,\xi_2)|\xi_1]=0=\mean[f(\xi_1,\xi_2)|\xi_2]
    \qquad\text{almost surely.}
  \end{equation*}
\end{theorem}
\begin{proof}
  See \cref{proofs:sec:mk}. The result generalizes results of J. Lindenstrauss and
  K.~R. Parthasarthy; see \cref{sec:related} for a comparison and references.
\end{proof}

\subsection{Invariant optimal couplings}

As an application of \cref{theorem:day,result:extremal:couplings}, we recall the Monge-Kantorovich theorem
\citep[e.g.][]{Rachev:Rueschendorf:1998}: There is a coupling $P^*$ that minimizes the risk over $\Lambda$ and satisfies
\begin{equation}
  \label{eq:MK}
  \inf\braces{P(c)\,|\,P\in\Lambda}
  \;=\;
  P^*(c)
  \;=\;
  \sup\braces{P_1(f_1)+P_2(f_2)\,|\,(f_1,f_2)\in\Gamma(c)}\;,
\end{equation}
where the supremum is taken over the set of minorants
\begin{equation*}
  \Gamma(c)\,:=\,\braces{f_1+f_2\leq c\text{ holds }P_1\otimes P_2\text{--a.s.}\,|\,f_i\in\L_1(P_i)\text{ for }i=1,2}\;.
\end{equation*}
Suppose we restrict $\Lambda$ to the subset $\Lambda_\group$, and $\Gamma(c)$ similarly to the subset
of $\group$-invariant minorants
\begin{equation*}
  \Gamma_\group(c)\;:=\;\braces{(f_1,f_2)\in\Gamma(c)\,|\,f_i\in\L_1(\Sigma_i,P_i)}\;.
\end{equation*}
This may turn the equality
\eqref{eq:MK} into an inequality ${\inf_{\Lambda_\group}>\sup_{\Gamma_\group}}$.
For risks invariant under an amenable group, there is no such duality gap:
\begin{corollary}[Invariant optimal couplings]
  \label{result:kantorovich}
  Let a nice amenable group $\group$ act
  continuously on $\Omega_1$ and $\Omega_2$. If $P_1$ and $P_2$ are
  $\group$-invariant and a cost $c$ satisfies
  \begin{equation}
    \label{eq:risk:kantorovich}
    (\phi\otimes\phi_*P)(c)\;=\;P(c)
    \qquad
    \text{ for all }\phi\in\group\text{ and }P\in\Lambda\;,
  \end{equation}
  the risk ${P\mapsto P(c)}$ is minimized over $\Lambda$ 
  by a $\group$-invariant coupling $\bar{P}$ that satisfies 
  \begin{equation*}
    \inf\braces{P(c)\,|\,P\in\Lambda_\group}
    \;=\;
    \bar{P}(c)\;=\;
    \sup\braces{P_1(f_1)+P_2(f_2)\,|\,(f_1,f_2)\in\Gamma_{\group}(c)}\;.
  \end{equation*}
  This coupling can always be chosen as an extreme point of $\Lambda_{\group}$.
\end{corollary}
\begin{proof}
  The convex set $\Lambda$ is compact in $\P(\XtimesY)$
  \citep[][2.2.1]{Rachev:Rueschendorf:1998}. It is also $\group$-invariant, since $P_1$ and $P_2$ are.
  The linear lsc functional ${g(P):=P(c)}$ on $\Lambda$ is $\group$-invariant by \eqref{eq:risk:kantorovich}.
  By \cref{theorem:day}, the subset $\Lambda_\group$ is compact and
  convex, and ${\min_{\ex\Lambda}g=\min_{\ex\Lambda_\group}}$.
  If $\bar{P}$ is an extreme point of $\Lambda_\group$ at which the minimum is attained, then
  \begin{equation*}
    \bar{P}(c)=\sup\nolimits_{\,\Gamma(c)}P_1\otimes P_2=\sup\nolimits_{\,\Gamma_\group(c)}P_1\otimes P_2
  \end{equation*}
  where the first identity holds by the Monge-Kantorovich theorem, and the second
  because $\Gamma_\group(c)$ is dense in $\Gamma(c)$ by \cref{result:extremal:couplings}.
\end{proof}

We may also trade off invariance of the marginals against invariance of $c$. The next
result considers marginals that are not invariant, but close enought to being so that
their total variation orbitopes contain invariant elements.
\begin{corollary}
  \label{result:approximate:tv:coupling}
  Let $Q_1$ and $Q_2$ be two probability measures that are not invariant, but have compact total variation orbitopes.
  Suppose the risk is even separately invariant,
  \begin{equation*}
    Q(c\circ\phi\otimes\psi)\;=\;Q(c)\qquad\text{ for all }\phi,\psi\in\group\text{ and }Q\in\Lambda(Q_1,Q_2)\;.
  \end{equation*}
  If $\group$ is amenable, restricting $\Gamma(c)$ to $\Gamma_\group(c)$ does not introduce a duality gap,
  \begin{equation*}
  \inf\braces{Q(c)\,|\,Q\in\Lambda(Q_1,Q_2)}\;=\;\sup\braces{Q_1(f_1)+Q_2(f_2)\,|\,(f_1,f_2)\in\Gamma_{\group}(c)}\;.
\end{equation*}
\end{corollary}
\begin{proof}
  See \cref{proofs:sec:mk}.
\end{proof}

\begin{example}
  \label{example:transport}
  (i) Let $P$ be the law of a coupling ${(\xi_1,\xi_2)}$, where $\xi_1$ and $\xi_2$ are real-valued
  stochastic processes indexed by $\mathbb{Z}$, i.e.\ random elements of
  ${\Omega_1=\Omega_2=\mathbb{R}^\mathbb{Z}}$. Let ${\group=\mathbb{Z}}$ act on the index set by addition.
  The coupling $P$ is then in $\Lambda_\group$ if it is jointly stationary,
  \begin{equation}
    \label{eq:mcgoff:nobel}
    (\xi_1(s),\xi_2(s))\;\equdist\;(\xi_1(s+t),\xi_2(s+t))\qquad\text{ for all }s,t\in\mathbb{Z}\;.
  \end{equation}
  A cost that satisfies \eqref{eq:risk:kantorovich} can be defined, for example, as
  \begin{equation}
    \label{eq:cost:mcgoff:nobel}
    c(\xi_1,\xi_2)=h(\xi_1(1),\xi_2(1))
    \qquad\text{ for some lsc }\quad
    h:\mathbb{R}\times\mathbb{R}\rightarrow[0,\infty)\;.
  \end{equation}
  Since the coordinate functions on $\mathbb{R}^\mathbb{Z}$ are continuous, $c$ is lsc on $\Omega_1\times\Omega_2$.
  \\[.5em]
  (ii) To see that couplings of invariant marginals need not be invariant, choose stationary processes $\xi_1$ and $\xi_2$ in above, independently with i.i.d.\ entries.
  Couple them by setting ${\xi_2(0):=\xi_1(0)}$. That does not change the marginal distributions, and both marginals are stationary,
  but $(\xi_1,\xi_2)$ is not.
  \\[.5em]
  (iii) McGoff and Nobel \citep{McGoff:Nobel:2020} study optimality of coupled dynamical systems, indexed by $\mathbb{N}$ rather than $\mathbb{Z}$, that
  can be represented in terms of our definitions as follows:
  Choose $\xi_1$ and $\xi_2$ as in (i), and let $P^+$ be the law of the restriction ${(\zeta_1,\zeta_2):=(\xi_1(s),\xi_2(s))_{s>0})}$ to positive indices.
  If $(\xi_1,\xi_2)$ satisfies \eqref{eq:mcgoff:nobel}, then $P^+$ is also stationary, in the sense that \eqref{eq:mcgoff:nobel} holds for ${t\geq 0}$.
  It is well known that, given stationarity, $P$ and $P^+$ determine each other uniquely \citep{Kallenberg:2005}.
  Phrased in this terminology, McGoff and Nobel \citep{McGoff:Nobel:2020} 
  optimize the risk
  \begin{equation*}
    R(\zeta_1,\zeta_2):=\mean[h(\zeta_1(0),\zeta_2(0))]
    \qquad\text{ for some measurable }\quad
    h:\mathbb{R}\times\mathbb{R}\rightarrow[0,\infty)
  \end{equation*}
  over all stationary laws $P^+$. If $h$ is lsc, this matches \eqref{eq:cost:mcgoff:nobel}, and we observe that 
  \begin{equation*}
    P^+(c)=P(c)
    \quad\text{ and }\quad
    \inf\braces{P^+(c)\,|\,P^+\text{ stationary}}\;=\;\inf\braces{P(c)\,|\,P\in\Lambda_\group}\;.
  \end{equation*}
  \cref{result:kantorovich} shows that $R$ is optimized by a jointly stationary coupling $P^+$, which
  is the restriction to positive indices of an extreme point $P$ of $\Lambda_\group$.
\end{example}

\section{Cocycles}
\label{sec:cocycles}

In machine learning problems, the summation trick is often used to obtain equivariant (rather than invariant) functions.
Using a simple device from algebra, called a cocycle, we can transform equivariance and other symmetry properties
into invariance under a surrogate action, which then makes our other results applicable.

\subsection{Surrogate actions}

Let $\group$ and $\mathbb{H}$ be groups, and $\mathcal{S}$ a set. A
map ${\theta:\mathbb{G}\times\mathcal{S}\rightarrow\mathbb{H}}$ that satisfies
\begin{equation*}
  \label{eq:cocycle}
  \theta(\psi\phi,s)\,=\,\theta(\psi,\phi s)\circ\theta(\phi,s)\quad\text{and}\quad
  \theta(\text{identity element of }\group,s)=\,s\quad\text{ for }s\in\mathcal{S}
\end{equation*}
is called a \kword{cocycle} \citep{Zimmer:1984}. We use this definition as follows:
Let $\mathcal{T}$ be another set, and $\mathcal{F}$ the set of all functions ${x:\mathcal{S}\rightarrow\mathcal{T}}$.
We let ${\group}$ act on $\mathcal{S}$ and $\mathbb{H}$ on $\mathcal{T}$, and consider properties of functions $x$
that can be formulated as
\begin{equation}
  \label{eq:cocycle:equation}
    \theta(\phi,s)\circ x(s)\;=\;x\circ\phi^{-1}(s)\qquad\text{ for all }\phi\in\group\text{ and }s\in\mathcal{S}\;.
\end{equation}
We rewrite \eqref{eq:cocycle:equation} as invariance under a surrogate action, by defining a map
\begin{equation}
  \Theta:\group\times\mathcal{F}\rightarrow\mathcal{F}
  \qquad\text{ as }\qquad
  \Theta(\phi,x)(s)\;:=\;\theta(\phi,s)\circ x\circ\phi(s)\;.
\end{equation}
If and only if $\theta$ is a cocycle, $\Theta$ is a valid action of $\group$ on the function set $\mathcal{F}$ \citep{Zimmer:1984}.
Clearly, $x$ satisfies \eqref{eq:cocycle:equation} if and only if it is invariant under the action $\Theta$.
The orbitope and \Folner average of $x$ under this action are
\begin{equation*}
  \orb^{\theta}(x)\;:=\;\cch\braces{\Theta(\phi,x)\,|\,\phi\in\group}
  \quad\text{ and }\quad
  \empavg^{\theta}_n(x)\;=\;\mfrac{1}{|\A_n|}\mint_{\A_n}\Theta(\phi,x)|d\phi|\;.
\end{equation*}
We can therefore apply \cref{theorem:day} to find functions that satisfy \eqref{eq:cocycle:equation} as follows.
\begin{corollary}
  \label{result:day:cocycle}
  Let a nice amenable group $\group$ act on a locally convex space $X$ whose elements are functions ${\mathcal{S}\rightarrow\mathcal{T}}$.
  Let $\theta$ be a cocycle such that ${\Theta(\phi,\argdot)}$ is linear and continuous for each ${\phi\in\group}$.
  If $\orb^{\theta}(x)$ is compact, $\empavg^{\theta}_n(x)$ exists for each $n$,
  and the sequence ${(\empavg_{n}^{\theta}(x))_n}$ has a convergent subsequence whose
  limit ${\bar{x}}$ is in ${\orb^{\theta}(x)}$ and satisfies \eqref{eq:cocycle:equation}.
\end{corollary}
A way to ensure linearity of $\Theta$ is as follows:
A cocycle is \kword{simple} if ${\theta(\phi,s)=\theta(\phi)}$ for all ${\phi\in\group}$.
If $\mathcal{T}$ is a vector space, $\mathbb{H}$ acts linearly on $\mathcal{T}$, and $\theta$
is simple, then $\Theta$ is always linear.
\begin{example}
  \label{example:cocycles}
  (i) If $\group=\mathbb{H}$ and $\theta$ is the identity, \eqref{eq:cocycle:equation} is $\group$-invariance $x$.
  \\[.2em]
  (ii) For ${\theta(\phi,s)=\theta}$, we obtain equivariance. Suppose $\group$ is amenable, and the action on $\mathcal{T}$ is linear.
  If the closed convex hull $\orb^{\theta}(x)$ of the functions ${\phi^{-1}\circ x\circ\phi}$ is compact, it contains
  an equivariant function.
  \\[.2em]
  (iii)
  Let $x$ be a function with multiple arguments, $\group$ the group of permutations of these arguments, 
  and $\mathbb{H}$ the multiplicative group $\braces{-1,1}$. Set ${\theta(\phi,s)=\text{sign}(\phi)}$.
  Then $x$ is skew-symmetric iff it satisfies \eqref{eq:cocycle:equation}.
  \\[.2em]
  (iv) As an example of a cocycle that is not simple, let
  $\mathcal{S}$ is a $\sigma$-algebra and ${\mathcal{T}=\mathbb{R}}$. All probability measure on $\mathcal{S}$ are then elements
  of $\mathcal{F}$. A probability measure $P$ is \kword{quasi-invariant} under $\group$ if ${\phi_*P\ll P}$ for all ${\phi\in\group}$,
  or in other words, if the image measure $\phi_*P$ has a density under $P$. This density, regarded as a function 
  ${\theta(\phi,s):=d(\phi^{-1}_*P)/dP(s)}$ of ${\phi\in\group}$ and ${s\in\mathcal{S}}$, is a cocycle \citep{Zimmer:1984}.
\end{example}
We can now apply all results derived for invariance to the surrogate action $\Theta$. For example, the
mean ergodic theorem (\cref{fact:mean:ergodic:theorem}) becomes:
\begin{corollary}
  \label{cocycle:mean:ergodic:theorem}
  Let $\xspace$ be a Hilbert space of functions from a measurable space $\mathcal{S}$ into a vector space $\mathcal{T}$.
  Let a nice amenable group $\group$ act measurably on $\mathcal{S}$, and let a group $\mathbb{H}$ act linearly on $\mathcal{T}$.
  If ${\cc:\group\rightarrow\mathbb{H}}$ is a simple cocycle that satisfies
  \begin{equation*}
    \label{cocycle:sp:condition}
    \sp{x\circ\phi^{-1},y\circ\phi^{-1}}\;=\;\sp{\cc(\phi)\circ x,\cc(\phi)\circ y}\qquad\text{ for all }x,y\in\xspace\text{ and }\phi\in\group\;,
  \end{equation*}
  then ${\xspace_{\cc}=\braces{x\in\xspace\,|\,x\text{ satisfies \eqref{eq:cocycle:equation}}}}$ is a closed linear subspace of $\xspace$, and
  \begin{equation*}
    \|\,\empavg^{\theta}_n(z)\;-\;\bar{x}\,\|\;\xrightarrow{n\rightarrow\infty}\;0
    \qquad\text{ holds for each }x\in\xspace\text{ and all }z\in\orb^{\theta}(x)\;,
  \end{equation*}
  where $\bar{x}$ is the orthogonal projection of $x$ onto $\xspace_{\cc}$.
  In particular, each orbitope $\orb^\theta(x)$ contains an element that satisfies \eqref{eq:cocycle:equation}.
\end{corollary}
We can hence regard $\empavg_n^{\theta}$ as an approximate projector onto $\xspace_{\cc}$ that becomes
exact asymptotically. If $\group$ is compact, we can choose ${\A_n=\group}$ for all $n$, and the
projector is exact. If $\group$ is compact and ${\theta(\phi,s)=\text{identity}}$, so that ${X_\theta=X_\group}$,
this projector is also known as a Reynolds operator \citep{Sturmfels:2008}.

\subsection{Illustration: Equivariant conditional probabilities}
\label{sec:pkernels}

As a more detailed example of the use of surrogate actions, we consider the existence of equivariant conditional distributions.
Suppose $\group$ acts measurably on $\Omega$, let $P$ be a probability measure on $\Omega$,
and let $\eta$ be a probability kernel, i.e.\ a measurable map
${\eta:\Omega\rightarrow\P(\Omega)}$.
For each ${t\in\Omega}$,
the value ${\eta(t)}$ is a probability measure, and we write ${\eta(A,t):=\eta(t)(A)}$ for a Borel set $A$.
The kernel is $P$-almost $\group$-equivariant if
\begin{equation}
  \label{eq:action:on:pkernels}
  \eta(\phi^{-1}\argdot,t)\;=\;\eta(\argdot,\phi t)
  \qquad\text{ for all }\phi\in\group\text{ and }P\text{-almost all }t\;,
\end{equation}
or if ${\phi_*\eta=\eta\circ\phi}$ in short. If $\xi$ and $\zeta$ are two random elements of $\Omega$,
where ${\zeta\sim P}$ and $\eta(\argdot,t)$ is the conditional distribution of $\xi$ given ${\zeta=t}$,
almost equivariance of $\eta$ means
\begin{equation*}
  \mathbb{P}(\xi\in\argdot|\zeta=\phi t)\;=\;\mathbb{P}(\phi\xi\in\argdot|\zeta=t)\qquad\text{ for each }\phi\in\group\text{ and almost all }t\in\Omega\;.
\end{equation*}
Given a kernel $\eta$, an equivariant $\bar{\eta}$ can be constructed as follows:
\begin{enumerate}
\item Choose the cocycle ${\theta(\phi,\eta):=\phi^{-1}}$, which defines
${\Theta(\phi,\eta)=\phi^{-1}_*\eta\circ\phi}$. A kernel is $\group$-equivariant if and only if it is invariant under $\Theta$.

\item Equip the set of probability kernels with its natural topology---this is the ``weak topology''
  that arises, for example, in the context of stable convergence and central limit theorems for
  dependent variables \citep{Haeusler:Luschgy}. We show
that this makes the orbitope of $\eta$ under the action $\Theta$ compact.

\item Verify $\Theta$ is linear and continuous. By \cref{result:day:cocycle},
the orbitope contains a probability kernel $\bar{\eta}$
that satisfies ${\Theta(\phi,\bar{\eta})=\bar{\eta}}$ for all $\phi$, and is hence $\group$-equivariant.
\end{enumerate}
Filling in the technical details (see the proof in \cref{proofs:sec:cocycles}) shows the following:
\begin{proposition}
  \label{result:equivariant:kernel}
  Let $\group$ be a amenable, and let $P$ be a $\group$-invariant probability measures on $\Omega$. If $\eta$ is a probability kernel
  whose marginal distribution ${\int\eta(\argdot,t)P(dt)}$ on $\Omega$ is $\group$-tight, there
  exists a $\group$-equivariant kernel $\bar{\eta}$ that satisfies
  \begin{equation*}
    \mint h(s,t)\bar{\eta}(ds,t)P(dt)
    \;\leq\;
    \sup\nolimits_{\phi\in\group}\,\mint h(\phi s,\phi t)\eta(ds,t)P(dt)
  \end{equation*}
  whenever ${h:\Omega^2\rightarrow\mathbb{R}}$ is measurable and ${s\mapsto h(s,t)}$ is lsc
  for all ${t\in\Omega}$.
\end{proposition}

\section{Related work}
\label{sec:related}

The summation formula \eqref{summation:trick} has been used for a long time, 
see e.g.~\citet{Minsky:Papert:1969}. The term summation trick is used by \citet{Diaconis:1988}.
Early applications to equivariant neural networks are due to \citet{Shawe-Taylor:1989} and \citet{Wood:Shawe-Taylor:1996}.
More recent examples include \citep{Cohen:Welling:2016,Kondor:Trivedi:2018}.
\\[.2em]
\emph{Orbitopes}.
The study of convex hulls of orbits of compact matrix groups in $\mathbb{R}^d$ has, at least in certain
special cases, a long history, see for example \citet{Atiyah:1982}. More recent work includes
\citep{Longinetti:Sgheri:Sottile,Farran:Robertson:1994,Sanyal:Sottile:Sturmfels:2011}.
The term \emph{orbitope} is introduced by \citet*{Sanyal:Sottile:Sturmfels:2011}, who
study the convex geometry of such objects in $\mathbb{R}^n$.
\\[.2em]
\emph{Amenability} was introduced by von Neumann to generalize certain properties of compact
groups, see e.g.~\citep{Grigorchuk:Harpe:2017,Bekka:delaHarpe:Valette:2008}.
The Hunt-Stein theorem is among the earliest applications, and led to a period of
intense interest in amenability in statistics, in particular in the context of Wald's decision theory
\citep{Bondar:Milnes:1981:1,LeCam:1986,Torgersen:1991,Eaton:George:2021}.
This literature uses von Neumann's definition, which is as follows:
Given a nice group $\group$, define $\L_\infty(\group)$ with respect to Haar measure. A \kword{mean}
is a linear functional ${\chi:\L_\infty(\group)\rightarrow\mathbb{R}}$ that
satisfies ${\chi(1)=1}$ and ${\chi(x)\geq 0}$ for all non-negative ${x\in\L_\infty(\group)}$.
Call $\group$ amenable if has an invariant mean, i.e.~one that satisfies
${\chi(x\circ\phi)=\chi(x)}$ for all ${\phi\in\group}$.
\Folner showed that a nice group is amenable if and only if it contains a \Folner sequence
\citep{Grigorchuk:Harpe:2017}. We follow common practice in modern ergodic theory
and use this characterization as a definition \citep[e.g.][]{Einsiedler:Ward:2011,Weiss:2003:1,Kerr:Li:2016}.
It has not been adopted in statistics, possibly because the Hunt-Stein theorem predates F{\o}lner's work.
\\[.2em]
\emph{Day's theorem} (\cref{fact:day}) seems to be known mostly to specialists, but an equivalent fact
is more widely appreciated, namely that every continuous action of an amenable group on a compact
set has an invariant probability measure \citep[e.g.][]{Einsiedler:Ward:2011,Bekka:delaHarpe:Valette:2008}.
(Combining this with Choquet's theorem and the fact that invariant measures have invariant barycenters
is one way to prove \cref{fact:day}.) Le Cam \citep{LeCam:1986} uses Day's theorem---though without
attribution to Day---to prove the Hunt-Stein theorem.
The proof constructs a convex $\group$-invariant risk with compact sublevel sets, 
and applies Day's theorem to these sublevel sets. This is precisely the idea described in the introduction,
minus the use of \Folner sequences.
(In our terminology, each of Le Cam's sublevel sets contains the orbitope of the minimax solution $\hat{w}$ in
the Hunt-Stein theorem.)
Another variant of the general idea appears in variational
analysis:
\begin{fact}[Symmetric criticality principle, \citet{Palais:1979}]
  Let $\group$ be a group of isometries of a Riemannian manifold $M$ and let ${f:M\to\mathbb{R}}$ be
  a $\mathbf{C}^1$ function invariant under $\group$. Then the set $M_\group$ of $\group$-invariant elements of $M$
  is a totally geodesic smooth submanifold of $M$, and if ${\bar{x}\in M_\group}$ is a critical point (a point at which the
  differential vanishes) of the restricted function
  $f|M_\group$, then $\bar{x}$ is in fact a critical point of $f$.
\end{fact}
\noindent This statement is weaker than those we use, as it
does not guarantee existence of $\bar{x}$. Accordingly, compactness or amenability of $\group$
is not required, and convexity is not used since the critical point is not required to be a minimizer.
Various generalizations exist in physics a variational analysis, see e.g.~\citep{Willem:1996}.
\\[.2em]
\emph{Dual actions}. Linear actions, also called representations, are predominantly studied in the case
where $\xspace$ is a Hilbert space or Euclidean. Some work on representations on Banach
spaces exists \citep[e.g.][]{Lyubich:1988}, but overall, much less seems to be known than in the Hilbert case.
Dual actions are also known as dual representations,
or contragredient representations \citep{Hall:2003}. For more on unitary actions, see \citep{Einsiedler:Ward:2011,Kerr:Li:2016}.
\\[.2em]
\emph{\cref{::annihilators} and \cref{result:contraction}}. Both proofs adapt
arguments used to prove the mean ergodic theorem for countable amenable
groups in Hilbert spaces
\citep[e.g.][]{Einsiedler:Ward:2011,Kerr:Li:2016}. See in particular the elegant proof of \citet[][Theorem 2.1]{Weiss:2003:1}.
\\[.2em]
\emph{Invariant mean embeddings} for compact groups are described in \cite{Raj:et:al:2017}.
The fact that $\group$-ergodic probability measures are the extreme points of the set of $\group$-invariant
distributions is of fundamental importance to ergodic theory, and goes back to R.~H. Farrell and V.~S. Varadarajan,
see \citep{Einsiedler:Ward:2011,Kallenberg:2005}.
\\[.2em]
\emph{\cref{result:extremal:couplings}} generalizes results by J. Lindenstrauss \citep{Lindenstrauss:1965}
and by K.~R. Parthasarathy \citep{Parthasarathy:2007}. 
Lindenstrauss proves equivalence of (i)--(iii) for
doubly-stochastic measures without group invariance---in our terminology, for the special case ${\XtimesY=[0,1]^2}$
with uniform marginals and $\group=\braces{\text{identity}}$. Doubly-stochastic measures are
known as \kword{copulas} in statistics, and as
\kword{permutons} in combinatorics. 
Parthasarathy establishes the equivalence (i) $\Leftrightarrow$ (iv), using methods
of quantum stochastic calculus. (Our proof does not require quantum probability.)
For ${\Omega_1=\Omega_2=\braces{1,\ldots,n}}$,
couplings with uniform marginals are doubly-stochastic matrices, and $\Lambda$ is
the Birkhoff polytope. 
Parthasarathy shows that the Birkhoff characterization of double-stochastic
matrices can be deduced from (iv).
\\[.2em]
\emph{\cref{result:kantorovich} (invariant couplings)}. Other results on invariant couplings include those
in \citep{Lim:2020,Ghoussoub:Moameni:2014} on invariant optimal martingale couplings (which have much more
structure than our couplings) for two specific compact groups, those of rotations and cyclic permutations.
\\[.2em]
\emph{Cocycles} are fundamental tool of ergodic theory
\citep[e.g.][]{Avila:Santamaria:Viana:Wilkinson:2013,Bowen:Nevo:2015,Kerr:Li:2016,Zimmer:1984}.
Versions of the definition have appeared in statistics under different names---Torgersen
\citep[][\S 6.5]{Torgersen:1991}, for example, formulates a ``generalized equivariance''
of decisions (a special case of
\eqref{eq:cocycle:equation} where $\theta$ is simple and $x$ a decision function),
and deduces a decision consistency requirement that is a special case of \eqref{eq:cocycle}.
Recent work of Dance and Bloem-Reddy shows that the mechanism by which certain interventions propagate through a
causal model is a cocycle, and specifying only this cocycle can be more robust than specifying the model \citep{Dance:Bloem-Reddy:2024}.
The idea of a pushing an ergodic theorem through a cocycle in \cref{cocycle:mean:ergodic:theorem}
is adapted from a similar (but much more sophisticated) one of Bowen and Nevo \citep{Bowen:Nevo:2015},
who do so for a pointwise ergodic theorem.

\vspace{1em}\noindent{\bf Acknowledgments}. This work was supported by the Gatsby Charitable Foundation.

\bibliography{references}

\newpage
\appendix

\section{Examples of amenable groups and cocycles}
\label{appendix:amenable}

\subsection{Amenable groups}

Many groups used in statistics and machine learning are 
amenable, and this section lists some examples. For more on the mathematical implications of
amenability, see \citep{Bekka:delaHarpe:Valette:2008,Weiss:2003:1,Grigorchuk:Harpe:2017}.
See \citep{Austern:Orbanz:2022} on how the averages $\empavg_n$ can be interpreted
as estimators, and for a range of examples in statistics.
\\[.5em]
1) All finite and compact groups are amenable (since we can choose all sets in the
  \Folner sequence as the group itself). All results throughout
  therefore hold for compact groups (though many become trivial).
  Compact groups relevant to statistics include the \kword{symmetric
    groups} $\mathbb{S}_n$, which consists of all permutations of the
  set $\braces{1,\ldots,n}$, and the \kword{orthogonal group}
  $\mathbb{O}_d$ of rotations of $\mathbb{R}^d$.
\\[.5em]
2) The \kword{finitary symmetric group} $\mathbb{S}=\cup_n\mathbb{S}_n$, which consists of all
finitely supported permutations of $\mathbb{N}$.
The subgroups $\mathbb{S}_n$ form a \Folner sequence.
\\[.5em]
3) The \kword{discrete shifts groups} ${(\mathbb{Z}^d,+)}$. The sets
$\braces{-n,\ldots,n}^d$ 
and $\braces{0,\ldots,n}^d$ both form \Folner sequences.
\\[.5em]
4) The \kword{continuous shift groups} ${(\mathbb{R}^d,+)}$, with \Folner sequences
${[-n,n]^d}$ or ${[0,n]^d}$.
\\[.5em]
5) The \kword{Euclidean groups} of all isometries of $\mathbb{R}^d$, which
can be identified with ${\mathbb{R}^d\times\mathbb{O}_d}$.
The sets ${[0,n]^d\times\mathbb{O}_d}$ form a \Folner
sequence.
\\[.5em]
6) The \kword{scale group} $(\mathbb{R}_{>0},\,\cdot\,)$, with
the sets ${[1,e^n]}$ as approximating sequence. This group is
homomorphic to ${(\mathbb{R},+f)}$ via the group homomorphism
${x\mapsto\exp(x)}$.
\\[.5em]
5) The \kword{finitary special orthogonal group} $\mathbb{SO}_\infty$, which can be
identified with ${\cup_{n\in\mathbb{N}}\,\mathbb{SO}_n}$, and has \Folner sequence ${(\mathbb{SO}_n)_n}$.
Invariance under this group is called rotatability and characterizes certain Gaussian processes
\citep{Kallenberg:2005}.
\\[.5em]
6) A \kword{finitary unitary group} can be defined similarly 
\citep{Bougarde:Najnudel:Nikeghabali:2013}.
\\[.5em]
7) All \kword{abelian}, all \kword{nilpotent}, and all \kword{solvable} groups are
amenable \citep{Grigorchuk:Harpe:2017}.
\\[.5em]
8) The group of non-singular, upper-triangular matrices is amenable (as it is solvable).
\\[.5em]
9) Any multiplicative group of upper triangular matrices whose diagonal
entries are all 1 is amenable (since it is nilpotent).
\\[.5em]
10) A special case is the \kword{Heisenberg group} over an algebraic field $\mathbb{K}$ \citep{Zimmer:1984},
\begin{equation*}
  \braces{M[a,b,c]\,|\,a,b,c\in\mathbb{K}}
  \quad\text{ where }\quad
  M[a,b,c]\;:=\;\left(\begin{smallmatrix} 1 & a & b \\[.1em] 0 & 1 & c \\[.1em] 0 & 0 & 1\end{smallmatrix}\right)\;.
\end{equation*}
The sets ${\A_n:=\braces{M(a,b,c)||a|,|b|,|c|\leq n}}$ form a \Folner sequence
if $\mathbb{K}$ is $\mathbb{Z}$ or $\mathbb{R}$.
\\[.5em]
11) The \kword{crystallographic groups}. A group $\group$ of isometries of $\mathbb{R}^k$
is crystallographic if it tiles $\mathbb{R}^k$ with a convex polytope ${M\subset\mathbb{R}^k}$,
that is, the polytopes ${\phi M}$ for ${\phi\in\group}$ cover $\mathbb{R}^k$ entirely and
only their boundaries overlap. There are 17 such groups for ${k=2}$, 230 for ${k=3}$, and a finite
number for each ${k\in\mathbb{N}}$. Those on $\mathbb{R}^3$ describe the possible geometries of crystals,
and are fundamental to materials science. See \citep{Adams:Orbanz:2023} for applications in machine learning
and for references.
\\[.5em]
12) The \kword{lamplighter group} \citep[e.g.][]{Lindenstrauss:2001}. Each group element is a pair
${\phi=(I_\phi,j_\phi)}$ of a finite set
${I_\phi\subset\mathbb{Z}}$ and a number ${j_\phi\in\mathbb{Z}}$. 
The operation is ${\psi\phi:=(I_\psi\vartriangle I_\phi,\,j_\psi+j_\phi)}$,
with unit element ${e=(\varnothing,0)}$.
The lamplighter metaphor imagines an infinite row of street
lights indexed by $\mathbb{Z}$. Each element $\phi$ specifies
all lamps indexed by $I_\phi$ as
on, and a lamplighter stands at
lamp $j_\phi$. An element $\psi$ acts on $\phi$ by toggling
the state of all lamps in $I_\psi$, and moving the lamplighter by
$j_\psi$ steps.

\subsection{Examples of cocycles}
\label{appendix:cocycles}

We have already seen in \cref{example:cocycles} that equivariance, and skew-symmetry can be represented
as cocycles. We mention a few more case, but the list is far from exhaustive.
\\[.5em]
1) \emph{Quasi-invariance}. A function is quasi-invariant under $\group$ if it is invariant up to positive scaling.
Thus, $x$ is quasi-invariant iff it is $\Theta$-invariant for any map ${\theta(\phi,s)=\theta(\theta)}$
into the multiplicative group ${(\mathbb{R}_{>0},\cdot)}$. 
\\[.5em]
2) \emph{Characters}. Choose $\mathbb{H}$ as the complex circle
${\braces{z\in\mathbb{C}|\,|z|=1}}$ with multiplication as operation.
A continuous function ${\theta:\group\rightarrow\mathbb{H}}$
with ${\theta(e_{\group})=1}$ is a \kword{character} of $\group$. 
This implies \eqref{eq:cocycle}, so the characters
of $\group$ are precisely the continuous simple cocycles into $\mathbb{H}$. For ${\group=(\mathbb{R},+)}$,
every character is of the form ${\phi\mapsto e^{i\psi\phi}}$ for some ${c\in\mathbb{R}}$, which is, of course, the Fourier kernel.
\\[.5em]
3) \emph{The modulus}. On every nice group exists a
unique function
${\theta:\group\rightarrow\mathbb{R}}$, called the \kword{modulus} of $\group$,
that satisfies
\begin{equation*}
  \mint f(\phi\psi^{-1})|d\phi|
  \,=\,
  \theta(\psi)
  \mint f(\phi)|d\phi|
  \qquad\text{ for all }\phi,\psi\in\group
\end{equation*}
for every compactly supported continuous ${f:\group\rightarrow(0,\infty)}$, see \citep{Eaton:1989}.
The definition implies ${\theta(\phi\psi)=\theta(\phi)\theta(\psi)}$, so the modulus is a simple cocycle.
\\[.5em]
4) The \emph{coboundary} ${(\phi x-x)(s)=:\theta(\phi,s)}$ of a function $x$ (\cref{:::annihilators}) is a cocycle.
\\[.5em]  
5) {\em Gradient fields}.
An isometry of $\mathbb{R}^n$ is a map ${\phi(\omega)=A_\phi\omega+b_\phi}$, where $A_\phi$ is an orthogonal matrix and ${b_\phi\in\mathbb{R}^n}$.
Let $\group$ be a group of isometries.
If a differentiable function ${f:\mathbb{R}^n\rightarrow\mathbb{R}}$ is $\group$-invariant, its gradient
vector field ${F(\omega)=(\nabla x)(\omega)}$ satisfies ${A_\phi F=F\circ\phi}$ for all ${\phi\in\group}$.
This is $\Theta$-invariance for the simple cocycle ${\theta(\phi,\omega):=A_{\phi}^{t}}$.
See \citep{Adams:Orbanz:2023} for more on such $\Theta$-invariant fields.

\newpage

\section{Proofs for Section \ref{sec:day}}
\label{proofs:sec:day}

In this section, we prove \cref{lemma:orbitope:invariant} and \cref{theorem:day}. Both proofs use basic
facts from Choquet theory:
\begin{fact}[e.g. {\citep[][1.2 and 1.5]{Phelps:2001}}]
  \label{fact:choquet}
  Let $K$ be a compact set in a locally convex Hausdorff space. A point is in the closed convex hull $\cch K$ if and
  only if it is the barycenter of a Radon probability measure on $K$. If ${\cch K}$ is also compact,
  every Radon probability on $K$ has a barycenter, and all extreme points of $\cch K$ are in $K$.
\end{fact}
\begin{proof}[Proof of \cref{lemma:orbitope:invariant}]
  \step The orbit $\group(x)$ is $\group$-invariant. Since each ${\phi\in\group}$ is linear continuous,
  it commutes with convex combinations and limits, so ${\orb(x)=\cch\group(x)}$ is $\group$-invariant.
  A set ${A\subset\xspace}$ is hence an open subset of $\Pi(x)$ iff $\phi A$ is. Since the interior
  is the union of all open subsets, it is $\group$-invariant.
  Consequently, ${\partial\Pi(x)=\Pi(x)\setminus\Pi(x)^\circ}$ is invariant.
  \\[.5em]
  \step
  If $A$ is $\group$-invariant, it is a disjoint union of orbits, so ${x\in A}$ iff ${\group(x)\subset A}$.
  If $A$ is also closed and convex, then ${\group(x)\subset A}$ iff ${\cch\group(x)\subset A}$. This shows (i).
  \\[.5em]
  \step The two statements in (ii) follow respectively from the fact that a Banach norm topology is completely metrizable, which implies that
  closed convex hulls of compact sets are compact \citep[][5.35]{Aliprantis:Border:2006}, and from the Krein-Smulian theorem \citep[][6.35]{Aliprantis:Border:2006}.
  \\[.5em]
  \step If $\orb(x)$ is compact, it is a compact convex set in a locally convex Hausdorff space, and (iii) follows from \cref{fact:choquet}.
  \\[.5em]
  \step For (iv), note that linearity of the action implies ${\group(ax_1+bx_2)=a\group(x_1)+b\group(x_2)}$ for any linear combination of vectors $x_1$ and $x_2$.
  (Here, ${a\group(x_1)}$ is the scaled set $\braces{az|z\in\group(x_1)}$, and $+$ on the right denotes the Minkowski sum.)
  It follows that
  \begin{equation*}
    \ch\group(x)\;=\;\ch\group\bigl(\tsum c_ix_i\bigr)\;=\;\tsum c_i\ch\group(x_i)\;,
  \end{equation*}
  and since the action, sum and scaling are continuous, ${\cch\group(x)=\msum c_i\cch\group(x_i)}$.
\end{proof}

To prove \cref{theorem:day}, we need a few auxiliary results. The first is a variant of Jensen's inequality for barycenters.
\begin{lemma}
  \label{lemma:jensen}
  Let $P$ be a Radon probability measure on a locally convex Hausdorff space $X$ whose barycenter exists.
  Then ${f(\sp{P})\leq P(f)}$ holds for every proper convex lsc function ${f:X\to\mathbb{R}}$.
  If the closed convex hull of a set ${M\subset X}$ is compact, we also have
  ${f\sp{P}\leq\sup\braces{f(x)|x\in M}}$ for every propability measure $P$ on $\overline{M}$.
\end{lemma}

\begin{proof}
  \step
  A proper convex lsc function is the pointwise supremum ${f(x)=\sup a(x)}$, taken over all affine continuous
  functions ${a:X\rightarrow\mathbb{R}}$ that satisfy $a<f$ \citep[][7.6]{Aliprantis:Border:2006}.
  Since ${P(a)=a\sp{P}}$ holds by the definition of barycenters whenever $a$ is linear, it also
  holds if $a$ is affine. We hence have
  ${f\sp{P}=\sup_{a<f}a\sp{P}=\sup_{a<f}P(a)\leq P(f)}$.
  \\[.5em]
  \step By Bauer's maximum principle, each $a$ attains its maximum on the compact convex set ${K:=\cch M}$ at an extreme point, so
  \begin{equation*}
    f\sp{P}
    =
    \sup_{a<f}P(a)
    \leq
    \sup_{a<f}\sup_{x\in K}a(x)
    \leq
    \sup_{a<f}\sup_{x\in\ex K}a(x)
    =
    \sup_{x\in\ex K}\sup_{a<f}a(x)
    =
    \sup_{x\in\ex K}f(x)\;.    
  \end{equation*}
  All extreme points of $K$ are in the closure of $M$ \citep[][1.5]{Phelps:2001}, so ${\sup_{\ex K}f=\sup_{\overline{M}}f}$, and since $f$ is lsc, 
  ${\sup_{\,\overline{M}}f=\sup_Mf}$.
\end{proof}

\def\pax{\eta_A}

Our strategy for the proof of \cref{theorem:day} is to represent the \Folner average $\empavg_n$ as the barycenter of a probability measure
supported on the orbitope. The next result provides such a measure:
\begin{lemma}
  \label{lemma:integrals:1}
  Let a nice group $\group$ act continuously on a Hausdorff space $X$. Let 
  ${A\subset\group}$ be compact. Then for each ${x\in X}$, 
  \begin{equation}
    \label{eq:pax}
    \eta_{A}(B)\;:=\;\frac{|\braces{\phi\in A\,|\,\phi(x)\in B}|}{|A|}\qquad\text{ for }B\subset X\text{ Borel}
  \end{equation}
  defines a Radon probability measure on $X$ that satisfies 
  \begin{equation*}
    |\pax(f-f\circ\psi)|\;\leq\;\mfrac{|A\vartriangle\psi A|}{|A|}\sup\nolimits_{\phi\in A}|f(\phi(x))|
  \end{equation*}
  for every continuous function ${f:\xspace\to\mathbb{R}}$ and every ${\psi\in\group}$.
\end{lemma}
\begin{proof}
  The measure ${U_{\!A}(\argdot):=|\argdot\cap A|/|A|}$ is the uniform distribution on $A$, and is Radon since
  $\group$ is Polish. We can write $\pax$ as the image measure ${\pax=U_{\!A}\circ t_x^{-1}}$ of 
  under the map ${t_x(\phi):=\phi(x)}$.
  Since $\group$ acts continuously, ${t_x:\group\rightarrow X}$ is continuous.
  Since $A$ is compact, a set ${C\subset A}$ is compact if and only if $t_x(C)$ is. It follows that
  \begin{equation*}
    \sup\braces{\pax(D)|D\subset B\text{ compact}}
    \;=\;
    \sup\braces{U_{\!A}(C)|C\subset t_x^{-1}B\text{ compact}}
    \;=\;
    U_{\!A}(t_x^{-1}B)
  \end{equation*}
  for any Borel set ${B\subset\xspace}$, so $\pax$ is Radon. If ${\psi\in\group}$, substituting into the definition shows
  \begin{equation*}
    \pax(\psi^{-1}B)
    \;=\;
    \frac{|\braces{\phi\in A\,|\,\psi\phi(x)\in B}|}{|A|}
    \;=\;
    \frac{|\braces{\phi\in\psi^{-1}A\,|\,\phi(x)\in B}|}{|\psi^{-1}A|}
    \;=\;
    \eta_{\psi^{-1}A}(B)
  \end{equation*}
  If ${f:X\rightarrow\mathbb{R}}$ is continuous, ${\phi\mapsto f(\phi(x))}$ is continuous and hence bounded on $A$,
  which shows $f$ is ${\pax}$-integrable. We then have
  \begin{equation*}
    \begin{split}
      |\pax(f-f\circ\psi)|
      \;&=\;
      |\pax(f)-\psi_*\pax(f)|\\[.4em]
      \;&=\;
      |\pax(f)-\eta_{\psi^{-1}A}(f)|\\
    \;&=\;
    \mfrac{1}{|A|}
    \big|\mint_{A\vartriangle\psi^{-1} A} f(\phi(x))|d\phi|\big|
    \;\leq\;
    \mfrac{|A\vartriangle \psi^{-1} A|}{|A|}
    \sup_{\phi\in A}|f(\phi(x))|
    \end{split}
  \end{equation*}
    for each ${\psi\in\group}$.
\end{proof}

\begin{lemma}
  \label{lemma:integrals:2}
  Let a nice group $\group$ act linearly and continuously on a locally convex Hausdorff space $X$, and let $\orb(x)$ be compact.
  Then the integral ${|A|^{-1}\int_A\phi{x}|d\phi|}$ exists for each compact ${A\subset\group}$,
  is the barycenter of $\pax$, and is contained in $\orb(x)$.
\end{lemma}
\begin{proof}
  Define the continuous map $t_x$ as above. Since $t_x(A)$ is compact, every Radon measure concentrated
  on $t_x(A)$ has a barycenter in the closed convex hull of $t_x(A)$ \citep{Phelps:2001}.
  Since $\pax$ is such a Radon measure, its barycenter $\sp{\pax}$ exists and is in ${\cch t_x(A)\subset\orb(x)}$.
  For any continuous linear ${\ell:X\to\mathbb{R}}$, substituting into the barycenter definition shows
  \begin{equation*}
    \ell(\sp{\pax})
    \;=\;
    \mint\ell(z)\pax(dz)
    \;=\;
    \mint_A\ell(t_x(\phi))\mfrac{|d\phi|}{|A|}
    \;=\;
    \mfrac{1}{|A|}\mint_A\ell(\phi(x))|d\phi|\;,
  \end{equation*}
  which is just the definition of the integral.
\end{proof}

\begin{proof}[Proof of \cref{theorem:day}]
  Since $x$ is fixed, we abbreviate ${\orb=\orb(x)}$, and denote by $M$ the closure of the orbit $\group(x)$.
  For each $n$, denote by ${\eta_n:=\pax}$ the measure defined in \eqref{eq:pax} for the \Folner set ${A=\A_n}$.
  \\[.5em]
  \step
  By \cref{lemma:integrals:2}, we have ${\empavg_n(x)=\sp{\eta_n}}$. In particular, $\empavg_n(x)$ exists and is in $\orb$.
  \\[.5em]
  \step
  Let $\P(M)$ be the set of Radon probability measures on $M$. Since $\orb$ is compact Hausdorff, so is the closed subset $M$, which makes $\P(M)$ compact Hausdorff in vague
  convergence \citep{Phelps:2001}. Since ${\eta_n(M)=1}$ for each $n$, the sequence $(\eta_n)$ is in $\P(M)$. By compactness, it has a subsequence $(\eta_{i(n)})$ that converges to a limit ${\bar{\eta}\in\P(M)}$.
  By \cref{fact:choquet}, its barycenter ${\bar{x}:=\sp{\bar{\eta}}}$ exists and is in ${\cch M=\orb}$.
  \\[.5em]
  \step
  Since $M$ is compact, vague convergence on $M$ means ${\eta_n(f)\to\bar{\eta}(f)}$ for every continuous ${f:\xspace\to\mathbb{R}}$.
  Consider any ${\psi\in\group}$. \cref{lemma:integrals:1} and the \Folner propert \eqref{eq:folner} then show
  \begin{equation*}
    |\bar{\eta}(f-f\circ\psi)|\;=\;\lim_n|\eta_n(f-f\circ\psi)|\;\leq\;\lim_n\mfrac{|\A_n\vartriangle\psi^{-1}\A_n|}{|\A_n|}\sup|f|\;=\;0\;.
  \end{equation*}
  We hence have ${\psi_*\bar{\eta}(f)=\bar{\eta}(f)}$ for every continuous $f$. 
  Since the vague topology separates points in $\P(\orb)$, it follows that ${\psi_*\bar{\eta}=\bar{\eta}}$, so the limit $\bar{\eta}$ is a $\group$-invariant measure.
  For any linear continuous ${\ell:X\to\mathbb{R}}$, the map ${\ell\circ\psi}$ is again linear continous, since the action is linear and continous.
  \begin{equation*}
    \ell(\psi(\sp{\bar{\eta}})\;=\;\bar{\eta}(\ell\circ\psi)\;=\;\psi_*\bar{\eta}(\ell)\;=\;\bar{\eta}(\ell)\;=\;\ell(\bar{\eta})
    \quad\text{ for }\ell:X\to\mathbb{R}\text{ linear continous }
  \end{equation*}
  since $\psi$, and hence $\ell\circ\psi$, is linear and continuous. As the linear continuous maps separate points in $\xspace$, that implies ${\psi\bar{\eta}=\bar{\eta}}$.
  Thus, the barycenter $\bar{x}$ is $\group$-invariant.
  \\[.5em]
  \step
  Consider the convex lsc function $f$ in \eqref{eq:day}. If $f=\infty$ everywhere on $\orb$, \eqref{eq:day} is trivially true.
  If ${f<\infty}$ somewhere, $f$ is proper. Since ${\bar{\eta}\in\P(M)}$ and hence ${\bar{\eta}(M)=1}$, \cref{lemma:jensen} shows
  \begin{equation*}
    f(\bar{x})\;\leq\;\sup\nolimits_{z\in\group(x)}f(z)\;=\;\sup\nolimits_{\phi\in\group}f(\phi z)\;.
  \end{equation*}
  In summary, we have shown that (i) and (ii) hold.
  \\[.5em]
  \step
  \def\ell{{g}}
  Since $K$ is closed convex and $\group$-invariant, the orbitopes of all ${x\in K}$ are in $K$, and hence compact. It follows by (ii) that $K_\group$ is not empty.
  Since the action is linear and continuous, the set $\xspace_\group$ of all invariant elements of $\xspace$ is a closed linear subspace, and
  ${K_\group=K\cap\xspace_{\group}}$ is a hence a compact convex set.
  By Bauer's extremum principle \citep[][]{Aliprantis:Border:2006}, $\ell$ attains its infimum at an extreme point $z_\ell$ of $K$.
  Let $\bar{z}$ be an invariant element of $\orb(z_\ell)$. The sublevel sets $[\ell\leq r]$ of $\ell$ are convex and $\group$-invariant (since $\ell$ is),
  and closed (since $\ell$ is lsc). It follows that
  \begin{equation*}
    \bar{z}\;\in\;\orb(z_\ell)\;\subset\;K\cap[\ell\leq\ell(z_\ell)]\;=\;\argmin_{K}\ell
    \quad\text{ and hence }\quad
    \min_K\ell\;=\;\min_{K_\group}\ell\;.
  \end{equation*}
  Since the set of minimizers of a convex function on a convex set is either extreme or empty, the minimium is attained at an extreme point of $K_\group$.
\end{proof}

\section{Proofs for Section \ref{sec:duality}}
\label{proof:sec:duality}

\begin{proof}[Proof of \cref{lemma:alaoglu}]
  \step The action has bounded orbits iff ${\sup_\phi\|\phi(x)\|<\infty}$ for each $x$, and hence iff the set ${\mathcal{T}:=\braces{\phi:\xspace\to\xspace|\phi\in\group}}$ of bounded linear operators
  is bounded pointwise on $X$. That is the case iff $\mathcal{T}$ is has bounded operator norm \citep[][6.14]{Aliprantis:Border:2006}. Similarly, the dual action defines a set $\mathcal{T}'$,
  and has bounded orbits iff $\mathcal{T}'$ is norm-bounded. Since $\mathcal{T}'$ consits of the adjoints of maps in $\mathcal{T}$, and a linear operator and its adjoint have the same norm,
  an action has bounded orbits if and only if its dual action does. This shows (i).
  \\[.5em]
  \step Substituting the definition of dual actions into that of the dual norm shows (ii).
  \\[.5em]
  \step
  Fix any ${y\in\yspace}$. If the orbits are bounded, the convex hull $\ch\group(y)$ is contained in a closed norm ball $B$. Since the weak* closure $\orb(y)$ is also norm-closed, it
  is likewise in $B$. By Alaoglu's theorem, $B$ is weak* compact, and the same hence holds for $\orb(y)$.
\end{proof}

\begin{proof}[Proof of \cref{::annihilators}]
  The set ${\orb(x)-x}$ is the closed convex hull of all vectors ${\phi(x)-x}$.
  Since the actions are dual, such vectors satisfy
  \begin{equation*}
    \sp{\phi(x)-x,y}\;=\;\sp{x,\phi^{-1}y-y}\qquad\text{ for all }y\in\yspace\;.
  \end{equation*}
  The vectors ${\phi y-y}$, for all $\phi$ and ${y\in\yspace}$, clearly form a vector space, say $V$.
  Observe that $x$ is $\group$-invariant if and only if ${\phi(x)-x=0}$ for all ${\phi}$, or equivalently, iff
\begin{equation*}
  \sp{\phi(x)-x,y}\;=\;0\qquad\text{ for all }\phi\in\group\text{ and }y\in\yspace\;.
\end{equation*}
It follows that $V$ is a subspace of the annihilator ${\xspace_\group^\perp}$ of $\xspace_\group$. Since
$x$ is only invariant if it annihilates all elements of $V$, we also have ${V^\perp=\xspace_\group}$.
Since ${V^{\perp\perp}}$ is the closure of $V$ \citep[][5.107]{Aliprantis:Border:2006}, this shows
${\xspace_\group^\perp}$ is the closure of ${\braces{\phi y-y\,|\,\phi\in\group\text{ and }y\in\yspace}}$.
This implies in particular ${\xspace_\group^\perp}$ contains ${\orb(y)}$.
Similarly, $\yspace_\group^\perp$ is the weak* closure of all functions ${\phi(x)-x}$ in $\xspace$, and
${\yspace_\group^\perp}$ contains ${\orb(x)}$.
\end{proof}

\begin{proof}[Proof of \cref{result:contraction}]
  \step
  We first use the weak topology. Let $\orb$ be a weakly compact orbitope in $\xspace$.
  Then each ${x\in\orb}$ has a compact orbitope, since ${\orb(x)\subset\orb}$.
  By \cref{lemma:integrals:2}, that implies that ${\empavg_n(x)}$ exists and is a barycenter,
  and \cref{lemma:jensen} shows
  \begin{equation*}
    f(\empavg_n(x))\;\leq\;\mfrac{1}{|\A_n|}\mint_{\A_n}f(x\circ\phi)|d\phi|
    \qquad\text{ for }f:\xspace\to\mathbb{R}\text{ proper convex lsc.}
  \end{equation*}
  The norm function is proper, convex, and weakly lsc \citep[][6.22]{Aliprantis:Border:2006}. It is also $\group$-invariant, since the action is isometric.
  It follows that
  \begin{equation}
    \label{eq:proof:contraction}
    \|\empavg_n(x)\|\;\leq\;\mfrac{1}{|\A_n|}\mint_{\A_n}\|x\circ\phi\|\,|d\phi|\;=\;\|x\|
    \qquad\text{ for each }x\in\orb\;.
  \end{equation}
  \step Now consider the norm topology. Let $x$ be a point in $\orb$. If a sequence $(x_i)$ in $\orb$ converges in norm to $x$, then
  \begin{align*}
    \|\empavg_nx\|
    \;&=\;
    \|\empavg_n(x-x_i+x_i)\|
    &&\\
    \;&=\;
    \|\empavg_n(x-x_i)+\empavg_nx_i\|
    &&\text{ since ${z\mapsto\empavg_n(z)}$ linear}\\
    \;&=\;
    \liminf_i\|\empavg_n(x-x_i)+\empavg_nx_i\|
    &&\text{ independent of }i\\
    \;&\leq\;
    \liminf_i\|\empavg_n(x-x_i)\|+\|\empavg_nx_i\|
    &&\text{ the norm is subadditive}\\
    \;&\leq\;
    \liminf_i\|x-x_i\|+\|\empavg_nx_i\|
    &&\text{ by \eqref{eq:proof:contraction}}
    \\
    \;&\leq\;
    \liminf_i\|\empavg_nx_i\|
    &&\text{ since }\|x-x_i\|\rightarrow 0\;.
  \end{align*}
  (This does not assume norm continuity of the integral, since $\empavg_n$ is eliminated before the limit is applied.)
  We have therefore shown that $\|\empavg_n(\argdot)\|$ is norm lsc on $\orb$.
  \\[.5em]
  \step For elements of ${\group(x)-x}$, we have
    \begin{align*}
      \|\empavg_n(\phi(x)-x)\|
      \;&=\;
      \|\mint_{\A_n^{-1}}\psi(\phi(x)-x)|d\psi|\|
      &&
      \text{definition of }\empavg_n\\
      \;&=\;
      \|\mint_{\A_n^{-1}\phi}\psi x|d\psi|-\int_{\A_n^{-1}}\psi x|d\psi|\|
      &&
      \\
      \;&\leq\;
      \mint_{\A_n^{-1}\vartriangle\phi\A_n^{-1}}\|\psi(x)\||d\psi|
      &&
      \text{triangle inequality}\\
      \;&=\;
      \mint_{\A_n^{-1}\vartriangle\phi\A_n^{-1}}\|x\||d\psi|
      &&\text{norm is $\group$-invariant}\\
      \;&=\;
      \mfrac{|\A_n\vartriangle\phi^{-1}\A_n|}{|\A_n|}\|x\|
      &&
      \text{since ${|A|=|A^{-1}|}$ for ${A\subset\group}$}\;.
    \end{align*}
    Since each point $z\in\ch\group(x)$ is a finite convex combination ${z=\sum_{i\leq k}c_k\phi_kx}$,
    \begin{equation}
      \label{eq:proof:contraction:2}
      \|\empavg_n(z-x)\|\;=\;\|\tsum_{i\leq k}c_k\empavg_n(\phi_kx-x)\|
      \;\leq\;
      \max_{i\leq k}\mfrac{|\A_n\vartriangle\phi_i^{-1}\A_n|}{|\A_n|}\|x\|\;.
    \end{equation}
  \step
  All arguments in step 1 and 2 apply analogously on the dual space $Y$, since the dual norm is weak* lsc
  \citep[][6.22]{Aliprantis:Border:2006}. It follows that the function ${y\mapsto\|\empavg_n(y)\|}$ is lsc in the dual norm at every
  point $y$ whose orbitope is weak* compact. Since the action is isometric, that is true for all $y$ by \cref{lemma:alaoglu},
  so norm lower semicontinuity holds everywhere on $\yspace$. Moreover, \eqref{eq:proof:contraction:2} holds for each ${y\in\yspace}$ and each
  $z$ in the convex hull of $\group(y)$. We have therefore established (ii).
  \\[.5em]
  \step On the primal space $\xspace$, the function ${\|\empavg_n(\argdot)\|}$ is also weakly lsc, since the norm and weak topology have the same convex lsc functions
  \citep[][5.99]{Aliprantis:Border:2006}. 
  \\[.5em]
  \step
  It remains to verify the contraction ${\|\empavg_n(\orb)\|\to 0}$ in (i). Let $x$ be a point in $\xspace$ with ${\orb(x)=\orb}$. Since the convex hull $\ch\group(x)$ is 
  dense in $\orb(x)$, it hence suffices to show ${\|\empavg_n(z_1-z_2)\|\to 0}$ for all ${z_1,z_2\in\ch\group(x)}$. By
  \eqref{eq:contraction:2}, we have
  \begin{align*}
    \|\empavg_n(z_1-z_2)\|\;&=\;\|\empavg_n(z_1-x+x-z_2)\|\\
    \;&\leq\;\|\empavg_n(z_1-x)\|+\|\empavg_n(z_2-x)\|\;\leq\;
    \mfrac{|\A_n\vartriangle\phi^{-1}\A_n|}{|\A_n|}\|x\|
  \end{align*}
  for some ${\phi\in\group}$.
\end{proof}

\section{Proofs for Section \ref{sec:hilbert:Lp}}
\label{proofs:sec:hilbert:Lp}

\begin{proof}[Proof of \cref{result:orbitope:unitary:action}]
  The action is isometric by \eqref{H:inner:product:invariant}. By \cref{lemma:alaoglu}, $\orb(x)$ is weakly compact.
  Since ${\phi(x)=\bar{x}+\phi(x)^\perp}$, and ${\sp{\phi(x)^\perp,\bar{y}}=\sp{x^\perp,\bar{y}}=0}$ for ${\bar{y}\in\xspace_\group}$, we have
  \begin{equation*}
    \orb(x)\;\subset\;\bar{x}+\text{closure}(\text{span}\,\group(x^\perp))\;\perp\;\xspace_\group\;.
  \end{equation*}
  Orthogonality implies all ${y\in\orb(x)}$ have the same projection onto $\xspace_\group$, so
  ${\bar{y}=\bar{x}}$.
  Consider a convex combination of two points $\psi x$ and $\phi x$ on the orbit. Since $\group$ is a group,
  we can choose $\psi$ as identity without loss of generality. Minkowski's inequality shows
  \begin{equation*}
    \|\lambda x+(1-\lambda)\phi x\|\;\leq\; \lambda\|x\|+(1-\lambda)\|\phi x\|\;=\;\|x\|\;,
  \end{equation*}
  so all $y$ in the convex hull of $\group(x)$ satisfy
  $\|y\|\leq\|x\|$. Since the norm function is continuous, the same extends to all $y$ in the closure of the convex hull. 
  If $\group$ is amenable, $\orb(x)\cap\xspace_\group$ has at least one element, by \cref{theorem:day}. By orthogonality,
  there can be only one such element, and this element is $\bar{x}$.
\end{proof}

\section{Proofs for Section \ref{sec:probatopes}}
\label{proofs:sec:probatopes}

\begin{proof}[Proof of \cref{lemma:G:tight}]
  $P$ is $\group$-tight if and only if ${Q(K_\varepsilon)>1-\varepsilon}$ holds
  for some compact ${K_\varepsilon}$ and all ${Q\in\group(P)}$, for each ${\varepsilon>0}$.
  If and only if that is true, the same holds for all convex
  combinations
  of such $Q$. The convex hull of $\group(P)$ is hence a tight family
  of measures in the sense of Prokhorov's theorem iff $P$ is
  $\group$-tight, so its closure is compact iff $P$ is $\group$-tight.
\end{proof}
For the proof of \cref{result:TV:orbitopes}(i), we note that, for two measures ${\nu\ll\mu}$, 
\begin{equation}
  \label{transformed:density}
  \frac{d(\phi_*\nu)}{d(\phi_*\mu)}
  \;=\;
  \frac{d\nu}{d\mu}\circ\phi^{-1}
  \qquad\mu\text{-a.e. for each }\phi\in\group
\end{equation}
holds by \eqref{eq:image:measure:int}.
The proof of (ii) uses the orbit coupling theorem of H. Thorisson:
\begin{fact}[Thorisson \citep{Thorisson:1996}]
  \label{fact:thorisson}
  Let a locally compact Polish group $\group$ act measurably on a Polish space $\Omega$,
  and let $\Sigma_\group$ be the $\sigma$-algebra of $\group$-invariant Borel sets.
  Two probability measures $P$ and $Q$ coincide on $\Sigma_\group$ if and only if, for every
  ${\zeta\sim P}$, there is a random element $\Psi$ of $\group$ such that
  ${\Psi\zeta\sim Q}$.
\end{fact}

\begin{proof}[Proof of \cref{result:TV:orbitopes}]
  Let ${\mathcal{P}_\mu}$ be the set of all probability measures absolutely continuous with respect to $\mu$.
  \\[.2em]
  \step
  Since $\mu$ is $\group$-invariant, \eqref{transformed:density} implies that  $\P_\mu$ is $\group$-invariant,
  so ${P\in\P_\mu}$ implies ${\group(P)\subset\P_\mu}$. Morevoer, $\P_\mu$ is clearly convex, and
  since convergence in total variation preserves absolute continuity, it is closed. The closed
  convex hull ${\orb_{\TV}(P)}$ is hence again in $\P_\mu$, so $\mu$ dominates $\orb_{\TV}(P)$.
  \\[.2em]
  \step
  Let ${r:\P_\mu\rightarrow\L_1(\mu)}$ be the map ${Q\mapsto dQ/d\mu}$ from a measure to its density.
  By the Radon-Nikodym theorem, $r$ is a linear bijection of $P_\mu$ and ${\L_1(\mu)}$, and
  both $r$ and its inverse are norm-norm continuous \citep[][\S 13.6]{Aliprantis:Border:2006}.
  By \eqref{transformed:density}, we also have ${r(\phi_*Q)=r(Q)\circ\phi}$, since $\mu$ is $\group$-invariant.
  In other words, $r$ commutes with norm-closures, convex combinations, and elements of $\group$. It follows that
  \begin{equation*}
    d\orb_\TV(P)/d\mu
    \;=\;
    r(\orb_\TV(P))
    \;=\;
    \orb_1(r(P))
    \;=\;
    \orb_1(dP/d\mu)\;.
  \end{equation*}
  \step It remains to show (ii). 
  Let $\Sigma_\group$ be the $\sigma$-algebra of $\group$-invariant Borel sets.
  Then $P$ and $\phi_*P$ coincide on $\Sigma_\group$ for every ${\phi\in\group}$, since
  ${P(\phi^{-1}A)=P(A)}$ if ${A\in\Sigma_\group}$.
  If measures ${P_1,P_2,\ldots}$ coincide with $P$ on $\Sigma_\group$, the same holds for any convex combination of these measure,
  and for their limit in total variation (since total variation
  implies convergence on each Borel set). Thus, all measures in $\orb_{\TV}(P)$ have the same restriction to
  $\Sigma_\group$. By \cref{theorem:day}, $\orb_{\TV}(P)$ contains a $\group$-invariant measure $\bar{P}$, and we
  can apply \cref{fact:thorisson} to couple $P$ and $\bar{P}$.
  \end{proof}

\begin{proof}[Proof of \cref{lemma:cid:orbitope:pointmass}]
  Since ${\phi\delta_x=\delta_{\phi(x)}}$ for ${\phi\in\group}$, the orbit of $\delta_x$ is
  \begin{equation*}
    \group(\delta_x)\;=\;\braces{\delta_z\,|\,z\in\group(x)}
    \quad\text{ and }\quad
    \ch\group(\delta_x)\;=\;\P(\group(x))\;.
  \end{equation*}
  \step Since convergence in total variation implies setwise convergence, $\P(\group(x))$ is closed
  in total variation, which shows (i).
  \\[.2em]
  \step
  Since $\Omega$ is Polish, every closed subset is a $G_\delta$ set, and hence again Polish \citep{Kechris:1995}.
  The orbit closure $\overline{\group(x)}$ of every point $x$ is hence separable and metrizable.
  \\[.2em]
  \step
  Let ${x_1,x_2,\ldots}$ be points in $\xspace$. By the definition of convergence in distribution,
  \begin{align*}
    \delta_{x_n}\,\darrow\,\delta_{x}
    \quad\Leftrightarrow\quad
    f(x_n)\rightarrow f(x)
    \text{ all bounded continuous }f
    \quad\Leftrightarrow\quad
    x_n\,\rightarrow\,x \text{ in }\xspace\;.
  \end{align*}
  That shows ${\delta_z\in\overline{\group(\delta_x)}\Leftrightarrow z\in\overline{\group(x)}}$
  and hence
  \begin{equation*}
    \Pi_{\cid}(\delta_x)\;=\;\cch\overline{\group(\delta_x)}\;=\;\cch\braces{\delta_z|z\in\overline{\group(x)}}
    \;.
  \end{equation*}
  The convex hull of $\braces{\delta_z|z\in\overline{\group(x)}}$ are the finitely supported probability measures
  on $\overline{\group(x)}$. Since the orbit closure is metrizable, the finitely supported probabilities are 
  weak* dense in all probability measures \citep[][15.10]{Aliprantis:Border:2006}, so
  \begin{equation*}
    \Pi_{\cid}(\delta_x)
    \;=\;
    \cch\braces{\delta_z|z\in\overline{\group(x)}}
    \;=\;
    \P(\overline{\group(x)})\;.
  \end{equation*}
  Since the extreme points of the set of probability measures on a metrizable space are the
  point masses \citep[][15.9]{Aliprantis:Border:2006}, the
  extreme points of $\Pi_{\cid}(\delta_x)$ are the point masses on the orbit closure.
  That shows (ii).
  \\[.2em]
  \step
  Since $\Omega$ is separable and metrizable,
  the set of probability measures on a closed subset
  is a closed face of ${\P(\Omega)}$ \citep[][15.19]{Aliprantis:Border:2006}.
  In particular, ${\P(\overline{\group(x)})}$ is such a closed face.
\end{proof}

\begin{proof}[Proof of the Lemma in \cref{example:point:masses}]
    By \cref{lemma:cid:orbitope:pointmass},
    we can determine the orbitopes ${\Pi_{\TV}(\delta_x)}$ and ${\Pi_{\cid}(\delta_x)}$
    by determining the orbit $\mathbb{S}(x)$ and its closure ${\overline{\mathbb{S}(x)}}$.
    The closure is taken in the product topology on $\Omega$, and hence in element-wise convergence of sequences.
    Let $\mathcal{T}_{\mathbb{N}}$ be the set of injections ${\tau:\mathbb{N}\rightarrow\mathbb{N}}$. Then
    \begin{equation*}
      z=\lim\nolimits_k\phi_k x\;\text{ for some }\phi_1,\phi_2,\ldots\mathbb{S}
      \quad\Leftrightarrow\quad
      z=\tau x\;\text{ for some }\tau\in\mathcal{T}_{\mathbb{N}}
    \end{equation*}    
    To see this, note that any bijection $\tau$ can be represented in this way, and that an infinite sequence of
    permutations can delete entries of $x$ in the limit by ``swapping them out to infinity''. Thus,
    ${\overline{\mathbb{S}(x)}=\mathcal{T}_{\mathbb{N}}x}$.
    If ${x\in\Omega(n,\infty)}$, we have
    \begin{equation*}
      \mathbb{S}(x)\;=\;\Omega(m,\infty)
      \quad\text{ and }\quad
      \overline{\mathbb{S}(x)}\;=\;
      \mathcal{T}_{\mathbb{N}}x\;=\;\cup_{k\leq m}\Omega(k,\infty)
    \end{equation*}
    and analogous statements hold for ${\Omega(\infty,n)}$. For every ${x\in\Omega(\infty,\infty)}$, there is
    some bijection ${\tau_x\in\mathcal{T}_{\mathbb{N}}}$ such that ${\tau_x x=(0,1,0,1,\ldots)}$, so
    \begin{equation*}
      \mathbb{S}(x)\;\subset\;\Omega(\infty,\infty)
      \quad\text{ and }\quad
      \overline{\mathbb{S}(x)}\;=\;\mathcal{T}_{\mathbb{N}}x\;=\;\Omega(\infty,\infty)\;.
      \qedhere
    \end{equation*}
\end{proof}

\begin{proof}[Proof of \cref{lemma:cid:duality}]
  \step Since $\Omega$ is Polish, it is a normal Hausdorff space \citep[][3.21]{Aliprantis:Border:2006}. That implies
  the norm dual of $(\C_b,\|\argdot\|_{\sup})$ is ${(\ba_n,\|\argdot\|_{\TV})}$, and that $\ca$ is a norm-closed in $\ba_n$ \citep[][14.10]{Aliprantis:Border:2006}.
  The dual norm is ${\|\mu\|'=|\mu(\Omega)|}$ \citep{Aliprantis:Border:2006}, so its
  restriction to $\ca$ is $\|\argdot\|_{\TV}$.
  \\[.5em]
  \step Since convergence in total
  variation preserves non-negativity and total mass, $\P$ is norm-closed conves subset of $\ca$, and therefore also of $\ba_n$.
  It follows that ${\orb_{\TV}(P)=\orb'(P)}$.
  \\[.5em]
  \step Since each ${\phi\in\group}$ defines a bijection of $\Omega$, we have ${\|f\circ\phi\|_{\sup}=\|f\|_{\sup}}$, so the action is isometric.
  By \cref{lemma:alaoglu}, this implies $\orb_*(P)$ is compact.
  \\[.5em]
  \step Let $\tau$ denote the topology of convergence in measure on $\ca$. Then the restriction of the weak* topology to $\ca$ is $\tau$, and $\P$ is closed
  in $\tau$ \citep{Aliprantis:Border:2006}. A set ${A\subset\P}$ is therefore $\tau$-closed iff it is the restriction ${A=A'\cap\P}$ of a weak*-closed set ${A\subset\ba_n}$.
  Since closures of convex sets are convex, $\orb_*(P)$ is the smallest weak*-closed set that contains ${\ch\group(P)}$.
  Its restriction $\orb_*(P)\cap\P$ is then the smallest $\tau$-closed set that contains ${\ch\group(P)\cap\P=\ch\group(P)}$, which is just the definition of $\orb_{\cid}(P)$.
  \\[.5em]
  \step As $\P$ is $\tau$-closed, we have ${\orb_*(P)\cap\P=\orb_*(P)\cap\ca}$, so ${\orb_*(P)\setminus\orb_{\cid}(P)\subset\ba_n\setminus\ca}$.
  \\[.5em]
  \step
  Since a set and its closure  have the same support function, we have
  \begin{equation*}
    H(f|\orb_{\cid}(P))
    \;=\;
    H(f|\ch\group(P))
    \;=\;
    H(f|\orb_*(P))
    \qquad\text{ for all }f\in\C_b\;.
  \end{equation*}
  The identity ${H(f|\orb_*(P))=H(P|\orb(f))}$ follows by substituting into \eqref{eq:support:function:duality}.
\end{proof}

\section{Proofs for Section \ref{sec:mmd}}
\label{proofs:sec:mmd}

We first prove the three general lemmas in \cref{sec:mmd:tools}.

\begin{proof}[Proof of \cref{lemma:invariant:kernel}]
  \step
  The linear maps ${T_\phi:x\mapsto x\circ\phi}$ define an action of $\group$ on the vector space
  of all functions ${x:\Omega\to\mathbb{R}}$. Define the set
  ${\mathbf{H}_0:=\text{span}\braces{\kappa(\omega,\argdot)|\omega\in\Omega}}$. This 
  is a norm-dense linear subspace of $\mathbf{H}$ \citep{Steinwart:Christmann:2008}.
  Since $\kappa$ is diagonally invariant,
  \begin{equation*}
    T_\phi(\kappa(\omega,\argdot))\;=\;\kappa(\omega,\argdot)\circ\phi=\kappa(\phi^{-1}\omega,\argdot)\;\in\;\mathbf{H}_0\;.
  \end{equation*}
  This extends to the span $\mathbf{H}_0$ by linearity of $T_\phi$, so ${T_\phi(\mathbf{H}_0)\subset\mathbf{H}_0}$.
  The restrictions of the maps $T_\phi$ to $\mathbf{H}_0$ hence define a linear action of $\group$ on $\mathbf{H}_0$.
  \\[.5em]
  \step
  The reproducing property ${x(\omega)=\sp{x,\kappa(\omega,\argdot)}}$ and diagonal invariance of $\kappa$ imply
  \begin{align*}
    \sp{\kappa(\upsilon,\argdot)\circ\phi,\kappa(\omega,\argdot)\circ\phi}
    \;&=\;
    \sp{\kappa(\phi^{-1}\upsilon,\argdot),\kappa(\phi^{-1}\omega,\argdot)}\\
    \;&=\;
    \kappa(\phi^{-1}\upsilon,\phi^{-1}\omega)
    \;=\;
    \kappa(\upsilon,\omega)
    \;=\;
    \sp{\kappa(\upsilon,\argdot),\kappa(\omega,\argdot)}\;.
  \end{align*}
  Since the inner product is bilinear and continuous on $\mathbf{H}$, this again extends to $\mathbf{H}_0$
  by linearity, and to $\mathbf{H}$ by continuity, which shows
  \begin{equation*}
    \sp{x\circ\phi,y\circ\phi}=\sp{x,y}
    \quad\text{ and hence }\quad
    \|x\circ\phi\|=\|x\|
    \quad\text{ for }x,y\in\mathbf{H}\;.
  \end{equation*}
  The restrictions of the maps $T_\phi$ to $\mathbf{H}$ are hence a unitary action of $\group$ on $\mathbf{H}$.
  \\[.2em]
  \step
  Since the feature map is ${\Delta(\omega)=\kappa(\omega,\argdot)}$, we have
  \begin{equation*}
    \Delta(\phi\omega)\;=\;\kappa(\phi\omega,\argdot)\;=\;\kappa(\omega,\argdot)\circ\phi^{-1}\;=\;(T_{\phi^{-1}}\Delta)(\omega)
  \end{equation*}
  for each ${\phi\in\group}$, so $\Delta$ is $\group$-equivariant.
  \\[.2em]
  \step If $\kappa$ is separately invariant, an analogous argument yields
  ${\sp{x\circ\phi,y\circ\psi}=\sp{x,y}}$ for all ${\phi,\psi\in\group}$, and $\group$-invariance of $\Delta$.
\end{proof}

\begin{proof}[Proof of \cref{RKHS:ergodic:theorem}]
  \step Since the inner product is diagonally invariant by \cref{lemma:invariant:kernel}, the mean ergodic theorem
  (\cref{cocycle:mean:ergodic:theorem}) implies norm convergence of $\empavg_n(x)$ to $\bar{x}$. For pointwise convergence,
  let ${\delta_\omega(x):=x(\omega)}$ be the evaluation functional. Since $\mathbf{H}$ is an RKHS, $\delta_\omega$ is norm-continuous
  as a map ${\mathbf{H}\to\mathbb{R}}$ \citep[][4.18]{Steinwart:Christmann:2008}. It follows that
  \begin{equation*}
    (\empavg_n(x))(\omega)
    \;=\;
    \delta_\omega(\empavg_n(x))\;\xrightarrow{n\to\infty}\;\delta_\omega(\text{norm }\lim\empavg_n(x))\;=\;\bar{x}(\omega)\;.
  \end{equation*}
  \step
  Now consider the function $\bar{\kappa}$. By diagonal invariance of $\kappa$, we have
  \begin{equation*}
    \kappa(\phi\omega,\psi\upsilon)\;=\;\kappa(\omega,\phi^{-1}\psi\upsilon)\;=\;\kappa(\omega,\argdot)\circ\phi^{-1}\psi(\upsilon)
  \end{equation*}
  Since $\empavg_n$ averages over group elements, ${\lim_n\empavg_n(x\circ\phi)=\lim_n\empavg_n(x)}$, and hence
  \begin{align*}
    \bar{\kappa}(\phi\omega,\psi\argdot)
    \;&=\;
    \lim\empavg_n(\kappa(\omega,\argdot)\circ\phi^{-1}\psi)
    \;=\;
    \lim\empavg_n(\kappa(\omega,\argdot))
    \;=\;
    \bar{\kappa}(\omega,\argdot)\;,
  \end{align*}
  so $\bar{\kappa}$ is separately $\group$-invariant.
  \\[.2em]
  \step Since $\mathbf{H}_\group$ is a closed subspace of $\mathbf{H}$, the restriction of the continuous functional
  $\delta_\omega$ to $\mathbf{H}_\group$ is again continuous, so $\mathbf{H}_\group$ is an RKHS. To show $\bar{\kappa}$ is a
  reproducing kernel for $\mathbf{H}_\group$, we must show it has the reproducing property for elements of $\mathbf{H}_\group$.
  Recall that any ${z\in\mathbf{H}}$ has a unique decomposition ${z=\bar{z}+z^\perp}$, where ${\bar{z}\in\mathbf{H}_\group}$
  and $z^\perp$ is orthogonal. Since ${\kappa_\omega:=\kappa(\omega,\argdot)}$ is in $\mathbf{H}$, it decomposes as
  ${\kappa_\omega=\bar{\kappa}_\omega+\kappa_\omega^\perp}$. As we have just shown above, the limit $\bar{\kappa}$ is the projection
  ${\bar{\kappa}(\omega,\argdot)=\bar{\kappa}_\omega}$. For any ${x\in\mathbf{H}_\group}$, we hence have
  \begin{align*}
    \sp{x,\bar{\kappa}(\omega,\argdot)}
    \;&=\;
    \sp{\bar{x}+x^\perp,\bar{\kappa}_\omega}
    &&\text{ since ${\bar{\kappa}(\omega,\argdot)=\bar{\kappa}_\omega}$}\\
    \;&=\;
    \sp{\bar{x},\bar{\kappa}_\omega}\,+\,\sp{x^\perp,\kappa_\omega^\perp}
    &&\text{ since ${x^\perp=0}$}\\
    \;&=\;
    \sp{x,\kappa}\;=\;x(\omega)
    &&\text{ since $\kappa$ has the reproducing property,}
  \end{align*}
  so $\bar{\kappa}$ has the reproducing property for $\mathbf{H}_\group$.
  \\[.5em]
  \step A kernel is measurable (resp.~separately continuous and bounded) if and only if all functions in $\mathbf{H}$
  are measurable (resp.~separately continuous and bounded) \citep[][4.24 and 4.28]{Steinwart:Christmann:2008}.
  It follow that if $\kappa$ has either property, it is inherited
  by the kernel of a subspace.
\end{proof}

\begin{proof}[Proof of \cref{lemma:mmd:isometry}]
  Linearity of $m$ follows from linearity of the Bochner integral,
  \begin{equation*}
    m(aP+bQ)\;=\;
    \mint\Delta(\omega)(aP(d\omega)+bQ(d\omega))
    \;=\;
    a m(P)+b m(Q)
    \quad\text{ for all }a,b\in\mathbb{R}\;.
  \end{equation*}
  Since $\mathcal{M}=m(\P)$ by definition, 
  and since $\kappa$ is a nice kernel, it is a bijection ${\P\rightarrow\mathcal{M}}$.
  By \cref{fact:mmd}, it is an isomorphism, and an isometry if $\P$ is metrized by MMD.
  If ${h:\mathbf{H}\rightarrow\mathbb{R}}$ is linear and continuous, it commutes with the integral \citep[][11.45]{Aliprantis:Border:2006},
  so
  \begin{equation*}
    \Delta_*P(h)=\mint h(x)\Delta_*P(d\omega)
    =
    \mint h\circ\Delta(\omega) P(d\omega)
    =
    h(\mint\Delta(\omega)P(d\omega))
    =
    h(m(P))\;,
  \end{equation*}
  which shows ${m(P)}$ is the barycenter of ${\Delta_*P}$.
  If $\tau$ leaves $\kappa$ diagonally invariant, the map ${x\mapsto x\circ\tau^{-1}}$ leaves the
  inner product of $\mathbf{H}$ diagonally invariant by \cref{lemma:invariant:kernel}. For each ${x\in\mathbf{H}}$,
  we therefore have
  \begin{align*}
    \sp{m(P)\circ\tau^{-1},x}
    \;=\;
    \sp{m(P),x\circ\tau}
    \;=\;
    \mint x\circ\tau dP
    \;=\;
    \mint x d(\tau_*P)
    \;=\;
    \sp{m(\tau_*P),x}\;.
  \end{align*}
  It follows that ${m(P)\circ\tau^{-1}=m(\tau_*P)}$.
\end{proof}

To establish \cref{result:mmd:extremal}---specifically, the characterization (ii) of the extreme points---we need an auxiliary
result, which we establish in two steps in the following lemmas.
\begin{lemma}
  Let $\Omega$ be locally compact Polish, and let $\C_b$ be the Banach space of bounded continuous functions on
  $\Omega$, equipped with the supremum norm. If $\kappa$ is a bounded and continuous characterstic kernel,
  then $\mathbf{H}$ is dense in $\C_b$. If $P$ is a probability measure on $\Omega$, then $\mathbf{H}$ is also
  dense in $\L_1(P)$.
\end{lemma}
\begin{proof}
\step
Let $\ca$ be the set of finite signed measures on $\Omega$. Recall that every such measure has a unique
decomposition ${\mu=\alpha\mu_+-\beta\mu_{-}}$, where ${\alpha,\beta\in[0,\infty)}$
and ${\mu_+,\mu_{-}\in\P}$. Since the map ${m:\P\to\mathbf{H}}$ is linear, continuous, and injective, we can hence
extend it to a map
\begin{equation*}
m:\ca\to\mathbf{H}\qquad\text{ as }\qquad m(\mu)\;:=\;\alpha{m(\mu_+)}-\beta{m(\mu_{-})}
\end{equation*}
which is again linear, continuous and injective. By linearity of the integral, the mean embedding property
${\int xdP=\sp{x,m(P)}}$ extends to
\begin{equation}
  \label{eq:signed:measure:mean:embedding}
  \int xd\mu=\sp{x,m(\mu)}\qquad\text{ for all }\mu\in\ca\text{ and }x\in\mathbf{H}\;.
\end{equation}
\step
Since ${m(\ca)\subset\mathbf{H}\subset\C_b}$, it suffices to show $m(\ca)$ is dense in $\C_b$.
Since $m$ is linear, $m(\ca)$ is vector subspace
of the Banach space $\C_b$. It is hence dense if the only continuous linear functional ${\ell:\C_b\to\mathbb{R}}$
that vanishes on $m(\ca)$ is the constant function $0$ \citep[][5.81]{Aliprantis:Border:2006}.
\\[.2em]
\step As $\Omega$ is locally compact Polish, the dual space of $\C_b$ is $\ca$, and each $\ell$ is of the form
${\ell(f)=\mu_{\ell}(f)}$ for some ${\mu_\ell\in\ca}$ \citep[][14.]{Aliprantis:Border:2006}. 
If we choose ${f=m(\mu_\ell)}$, \eqref{eq:signed:measure:mean:embedding} shows
\begin{equation*}
  \ell(m(\mu_\ell))\;=\;\mint m(\mu_\ell)d\mu_{\ell}\;=\;\sp{m(\mu_{\ell}),m(\mu_{\ell})}\;=\;\|m(\mu_{\ell})\|^2\;.
\end{equation*}
If $\ell$ vanishes on $m(\ca)$, we hence have ${\|m(\mu_\ell)\|=0}$ and therefore ${m(\mu_{\ell})=0}$.
Since $m$ is injective, that is only true if ${\mu_{\ell}=0}$, and hence if
${\ell=0}$. Thus, $m(\ca)$, and hence $\mathbf{H}$, is dense in $\C_b$.
\\[.2em]
\step If $P$ is a probability measure on $\Omega$, then $\C_b$ is dense in $\L_1(P)$ \citep[][1.35]{Kallenberg:2001}.
Since $\mathbf{H}$ is dense in $\C_b$, and convergence in the supremum norm on $\C_b$ implies convergence in the $\L_1$-norm,
$\mathbf{H}$ is dense in $\L_1(P)$.
\end{proof}
\begin{lemma}
  \label{lemma:mmd:invariants:constant}
  Let a nice group $\group$ act continuosly on $\Omega$. A $\group$-invariant probability measure $P$ is
  $\group$-ergodic if and only if ${P(f)=f}$ holds $P$-a.e.\ for every $\group$-invariant function ${f\in\mathbf{H}}$.
\end{lemma}
\begin{proof}
  Since $P$ is invariant, the $\L_1$-norm is $\group$-invariant, which makes the induced action of $\group$ on $\L_1$
  linear and continuous. It follows that the subset $(\L_1(P))_\group$ of $\group$-invariant elements is a closed linear
  subspace. Since $\C_b$ is dense in $\L_1(P)$ and the subspace is closed, it follows that $\C_b\cap(\L_1(P))_\group$
  is dense in $(\L_1(P))_\group$. The elements of $(\L_1(P))_\group$ are the equivalence classes of $\group$-invariant
  functions. The dense subset $\C_b\cap(\L_1(P))_\group$ are the those equivalence classes that contain a continuous
  and a $\group$-invariant function. Since the action is continous, that is the case if and only if the class contains
  a continuous invariant function. 
\end{proof}

\begin{proof}[Proof of \cref{result:mmd:extremal}]
  Since $\mathcal{M}_\group$ is the intersection of the set $\mathcal{M}$, which is closed and convex by
  \cref{result:mean:embeddings:geometry}, and the closed linear space $\mathbf{H}_\group$, it is closed and convex. It is also the
  isometric image ${m(\P_\group)}$, by \cref{lemma:mmd:isometry},
  which implies $m(P)$ is extreme in $\mathcal{G}_\group$ if and only if $P$ is
  extreme in $\P_\group$. By \cref{::ergodic:decomposition},
  the set of extreme points of $\P_\group$ is measurable, and an invariant measure $P$ is extreme if and only if
  it is ergodic. It follows that ${\ex\mathcal{M}_\group=m(\ex\P_\group)}$ is measurable, and that (i)$\Leftrightarrow$(iii).
  Since also (iii)$\Leftrightarrow$(i) by \cref{lemma:mmd:invariants:constant}, that proves the theorem. 
\end{proof}

\section{Proofs for Section \ref{sec:mk}}
\label{proofs:sec:mk}

\begin{fact}[Becker and Kechris {\citep[][5.2.1]{Becker:Kechris:1996}}]
  \label{fact:becker:kechris}
  If a Polish group $\group$ acts measurably on a Polish space
  $\Omega$, there exist Polish topologies on $\group$ and $\Omega$ that generate the same Borel sets
  as the original topologies and make the action continuous.
\end{fact}

\begin{proof}[Proof of \cref{result:extremal:couplings}]
  Equip the set ${\P=\P(\XtimesY)}$ with the topology of convergence in distribution, and let $\P_\group$ be the subset
  of $\group$-invariant distributions.
  We proceed as follows: We first prove (i) $\Leftrightarrow$ (iii) for continuous actions. We next generalize this equivalence to
  measurable actions, and then prove (ii) and (iv).\\[.2em]
  \step Assume the action on $\Omega$ is continuous. As we have noted in the previous proof, that
  makes ${\Lambda_\group=\Lambda\cap\P_\group}$ compact, convex and $\group$-invariant.
  Since $\Omega$ is Polish, $\P$ is Polish \citep[][15.15]{Aliprantis:Border:2006}, so the closed subset $\Lambda_{\group}$ is metrizable.
  \\[.2em]
  \step The next two steps adapt Lindenstrauss' elegant proof for uniform marginals in
  \citep{Lindenstrauss:1965}: By Choquet's theorem, there is a probability measure $\mu_P$ on the extreme points
  $\ex\Lambda_\group$ such that ${P=\int e\mu_P(de)}$. If and only if $P$ is not extreme, $\mu_P$ is supported
  on more than one point \citep[][1.4]{Phelps:2001}. There is hence some ${C\subset\ex\Lambda_\group}$ such that ${\lambda:=\mu_P(C)}$
  satisfies ${0<\lambda<1}$. Set
  \begin{equation*}
    \nu_1\;:=\;\mint_C e\mu_P(de)
    \qquad
    \nu_2\;:=\;\mint_{\ex\Lambda_\group\setminus C}e\mu_P(de)
    \qquad
    \nu\;:=\;\lambda\nu_1-\lambda\nu_2\;.
  \end{equation*}
  For Borel sets ${A\subset\Omega}$ and ${B_i\subset\Omega_i}$, we hence have
  \begin{equation*}
    |\nu(A)|\leq|\lambda\nu_1(A)-\lambda\nu_2(A)+\nu_2(A)|=P(A)\;,
  \end{equation*}
  and also ${\nu(B_1\times\Omega_2)=\lambda P(B_1)-\lambda P(B_1)=0}$
  since $\nu_1$ and $\nu_2$ have the same
  marginals. In short, $P$ is not extreme iff there is a signed measure $\nu$ with
  \begin{equation}
    \label{proof:extremal:couplings}
    |\nu(A)|\;\leq\;P(A)\qquad\text{ and }\qquad
    \nu(B_1\times\Omega_2)\;=\;0\;=\;\nu(\Omega_1\times B_2)\;.
  \end{equation}
  \step
  Suppose such a $\nu$ exists.
  By the Radon-Nikodym theorem, $\nu_1$ and $\nu_2$ have densities ${f_j=\nu_j/dP}$, so $\nu$ has
  density ${f=\lambda f_1-\lambda f_2}$. As ${|\nu|\leq P}$, the density satisfies ${|f|\leq 1}$,
  and hence ${f\in\L_\infty(P)}$. Since $\nu_1$ and $\nu_2$ have positive mass and disjoint support, ${f\neq 0}$.
  Since $\nu_1$ and $\nu_2$ are in $\Lambda_\group$, the measure $\nu$ is $\group$-invariant. By \eqref{transformed:density},
  \begin{equation*}
    f\circ\phi\;=\;\mfrac{d\nu}{dP}\circ\phi\;=\;\mfrac{d\nu}{dP}\;=\;f\qquad P\text{-a.s.}
  \end{equation*}
  so $f$ is $P$-almost surely $\group$-invariant. That shows ${f\in\L_\infty(\Sigma,P)\setminus\braces{0}}$.
  Thus, if a $\nu$ satisfying \eqref{proof:extremal:couplings} exists, (iii) does not hold.
  Conversely, suppose $f$ is a function that violates (iii), and set ${h:=f/\|f\|_{\infty}}$. Then ${\nu:=f P}$ satisfies
  \eqref{proof:extremal:couplings}. Thus, $P$ is not extreme iff (iii) does not hold.  
  \\[.3em]
  In summary, (i) $\Leftrightarrow$ (iii) holds if the action is continuous.
  \\[.3em]
  \step
  Suppose the action is measurable. We change topologies using the Becker-Kechris theorem (\cref{fact:becker:kechris}) to make it continuous.
  That does not change the quantities in (i)---the spaces $\L_1(\Sigma_i,P_i)$ and $\L_\infty(\Sigma,P)$, the set $\Lambda_\group$, and the measures $P$ and $\nu$---since
  all depend on the topology of $\Omega$ only through its Borel sets. 
  The equivalence (i) $\Leftrightarrow$ (iii) for measurable actions hence follows from that for continuous ones.
  \\[.2em]
  \step
  We show (ii) $\Leftrightarrow$ (iii):
  Consider the set ${M=\braces{g_1+g_2|g_i\in\L_1(\Sigma_i,P_i)}}$. Since ${M\subset\L_1(\Sigma,P)}$, and since the norm dual of $\L_1$ is ${\L_\infty}$,
  the annihilator of $M$ is the set
  \begin{equation*}
    M^\perp\;=\;\braces{f\in\L_\infty(P)\,|\,\sp{g_1+g_2,f}=\mint(g_1+g_2)fdP=0}\;.
  \end{equation*}
  Recall that the annihilator ${(M^\perp)^\perp}$ of the annihilator is the closure ${\overline{M}}$ \citep[][5.107]{Aliprantis:Border:2006}.
  By (iii), $P$ is extreme iff ${M^\perp=\braces{0}}$, and hence iff ${\overline{M}=\braces{0}^\perp=\L_1(\Sigma,P)}$.
  \\[.2em]
  \step
  It remains to show (iii) $\Leftrightarrow$ (iv). For any ${f\in\L_\infty(\Sigma,P)}$, or indeed any
  ${f\in\L_\infty(P)}$, the definition of conditional expectation shows
  \begin{equation}
    \label{P:L:translation:1}
    \mint g_i f dP\;=\;\mint g_i\mean[f|\xi_i]dP_i\qquad\text{ for }g_i\in\L_1(P_i)
    \text{ and }\xi_i\sim P_i\;.
  \end{equation}
  Suppose (iii) is violated: ${\int (g_1+g_2)fdP=0}$ holds for some ${f\in\L_\infty(P,\Sigma_\group)\setminus\braces{0}}$
  and all $g_i$. By \eqref{P:L:translation:1}, that is the case
  if and only if
  \begin{equation*}
    \mint g_1\mean[f|\xi_1]dP_1+\mint g_2\mean[f|\xi_2]dP_2\;=\;0\;.
  \end{equation*}
  Since we can set either $g_i$ to $0$, this is in turn equivalent to
  \begin{equation}
    \label{P:L:translation:2}
    \mint g_i\mean[f|\xi_i]dP_i\;=\;0\quad\text{ for }i=1,2\text{ and }g_i\in\L_1(P_i)\;.
  \end{equation}
  Since this holds for all $g_i$, it implies ${\mean[f|\xi_i]=0}$ almost surely, so (iv) is violated.
  Conversely, ${\mean[f|\xi_i]=0}$ implies \eqref{P:L:translation:2}, so (iii) is violated if (iv) is.
\end{proof}

\section{Proofs for Section \ref{sec:cocycles}}
\label{proofs:sec:cocycles}

The proof of \cref{result:equivariant:kernel} becomes a straightforward application of
\cref{theorem:day} once we choose a suitable topology on probability kernels. To this end,
denote by $\mathcal{K}_P$ the set of all probability kernels on $\Omega$, equipped with the
smallest topology that makes all functions 
\begin{equation*}
  \eta\;\mapsto\;\mint\!\!\mint f(s)g(t)\eta(ds,t)P(dt)\qquad\text{ for }f\in\mathbf{C}_b\text{ and }g\in\L_1(P)
\end{equation*}
continuous.  Some authors call this the 
weak topology defined by $P$, see \citep[][Ch.~2]{Haeusler:Luschgy}.
Let ${p(\eta):=\int\eta(\argdot,t)P(dt)}$ be the marginal of a kernel $\eta$ under $P$.
We need the following properties of the topology, which can be found in
\citep[][2.6 and 2.7]{Haeusler:Luschgy}:
\begin{fact}
  \label{fact:pkernels}
  Equip $\P(\Omega)$ with the topology of convergence in distribution.\\[.2em]
  (i) The map ${p:\mathcal{K}_P\rightarrow\P(\Omega)}$ is continuous, and
  a set ${K\subset\mathcal{K}_P}$ is relatively compact if and only if its image
  ${p(K)}$ is relatively compact in $\P(\Omega)$.
  \\[.2em]
  (ii) If ${l:\Omega^2\rightarrow\mathbb{R}}$ is measurable, bounded below, and ${s\mapsto l(s,t)}$ is lsc for each ${t\in\Omega}$,
  the map ${I_l:\mathcal{K}_P\rightarrow\mathbb{R}\cup\braces{\infty}}$ defined by
  ${I_l(\eta)\;:=\;\int l(s,t)\eta(ds,t)P(dt)}$
  is linear and lsc.
\end{fact}

\begin{proof}[Proof of \cref{result:equivariant:kernel}]
  The map $p$ is linear, in the sense that
  \begin{equation}
    \label{eq:aux:proof:pkernels}
    p(\lambda\eta_1+(1-\lambda)\eta_2)=\lambda p(\eta_1)+(1-\lambda)\eta_2
    \qquad\text{ for }\lambda\in[0,1]\;,
  \end{equation}
  for any kernels $\eta_1$ and $\eta_2$.
  Denote by 
  ${\Theta(\group,\eta)=\braces{\phi^{-1}_*\eta\circ\phi|\phi\in\group}}$ the
  orbit of $\eta$ under the surrogate action $\Theta$.
  \\[.2em]
  \step
  If $P$ is $\group$-invariant,
  \eqref{eq:action:on:pkernels} implies
  \begin{equation*}
  p(\eta\circ\phi)\,=\,p(\eta)
  \quad\text{and}\quad
  p(\phi\eta)\,=\,\phi_* p(\eta)\
  \quad\text{hence}\quad
  p(\Theta_\phi\eta)\,=\,\phi_*^{-1}p(\eta)\;.
  \end{equation*}
  It follows, with \eqref{eq:aux:proof:pkernels}, that 
  ${p(\ch\Theta(\group,\eta))=\ch\group(p(\eta))}$.
  \\[.2em]
  \step
  Since $p(\eta)$ is $\group$-tight, its orbitope ${\Pi_{\cid}(p(\eta))}$ is compact by
  \cref{lemma:G:tight}, so
  \begin{equation*}
    p(\ch\Theta(\group,\eta))=\ch\group(p(\eta))\subset\Pi_{\cid}(p(\eta))\;.
  \end{equation*}
  Thus, ${\ch\group(\eta)}$ is relatively compact. It follows, by \cref{fact:pkernels}, that
  the orbitope $\Pi(\eta)$ in $\mathcal{K}_P$ is compact.
  By \cref{theorem:day}, a $\Theta$-invariant kernel $\bar{\eta}$ exists that satisfies \eqref{eq:day}.
  \\[.2em]
  \step
  By \cref{fact:pkernels}, the map $I_h$ is linear and lsc. We can hence apply \eqref{eq:day}, which shows
  \begin{align*}
    I_h(\bar{\eta})
    \;\leq\;
    \sup\nolimits_\phi
    I_h(\Theta_\phi\eta)
    \;&=\;
    \sup\nolimits_\phi
    \mint h(\phi s,\phi t)\eta(ds,t)\phi_*P(dt)\\
    \;&=\;
    \sup\nolimits_\phi
    \mint h(s,t)\eta(\phi ds,\phi t)P(dt)
  \end{align*}
  where the last identity uses the fact that $P$ is $\group$-invariant.
\end{proof}

\begin{proof}[Proof of \cref{result:approximate:tv:coupling}]
  By \cref{theorem:day}, the orbitopes respectively contain $\group$-invariant measures ${P}_1$ and $P_2$.
  For each ${Q\in\Lambda(Q_1,Q_2)}$, we can construct a $\group$-invariant coupling ${P\in\Lambda(Q_1,Q_2)}$ with the same risk:
  Start with ${(\xi_1,\xi_2)\sim Q}$. Use \cref{result:TV:orbitopes} to choose random
  elements $\Phi_1$ and $\Phi_2$ of $\group$ such that ${\Phi_i\xi_i\sim P_i}$, and choose $P$ as the joint law
  of ${(\Phi_1\xi_1,\Phi_2\xi_2)}$. Then
  \begin{equation*}
    P(c)\;=\;Q(c)\qquad\text{ and }\qquad
    P_i(f_i)\;=\;Q_i(f_i)
    \quad\text{ for }(f_1,f_2)\in\Gamma_\group(c)\;.
  \end{equation*}
  We can hence substitute $(P_1,P_2)$ for $(Q_1,Q_2)$ without changing the infimum or supremum. Since $P$ is $\group$-invariant,
  it has $\group$-invariant marginals, and we can apply
  \cref{result:kantorovich}.
\end{proof}

\end{document}